\documentclass[a4paper,dvipsnames]{amsart}

\usepackage[english]{babel}
\usepackage[T1]{fontenc}

\usepackage{etoolbox}
\apptocmd{\sloppy}{\hbadness 10000\relax}{}{}
\apptocmd{\sloppy}{\vbadness 10000\relax}{}{}

\usepackage{mathtools}

\usepackage[a4paper,top=3cm,bottom=2cm,left=3cm,right=3cm,marginparwidth=1.75cm]{geometry}

\usepackage{amsmath}
\usepackage{graphicx}
\usepackage[dvipsnames]{xcolor}
\usepackage[pagebackref,hypertexnames=false, colorlinks, citecolor=BrickRed,linkcolor=MidnightBlue, urlcolor=BrickRed]{hyperref}

\usepackage[capitalize]{cleveref}

\usepackage{diagbox,multirow} 
\usepackage{relsize}
\usepackage{comment}  

\usepackage{lmodern}
\usepackage[english]{babel}
\usepackage{amsmath}
\usepackage{amsfonts}
\usepackage{colonequals} 
\usepackage{amssymb}
\usepackage{amsthm}
\usepackage{stmaryrd}
\usepackage[all]{xy}

\usepackage{enumitem}
\setlist[enumerate]{topsep=1mm,parsep=0pt,partopsep=0mm,itemsep=1mm,leftmargin=12mm,labelsep=1mm} 

\usepackage{hyphenat}

\usepackage{faktor}

\colorlet{Malte}{OliveGreen}
\colorlet{Uwe}{red}
\colorlet{Ania}{blue}

\newcommand{\Pol}[1]{\operatorname{Pol}(#1)}

\newcommand{\A}{\mathcal A} 
\newcommand{\B}{\mathcal B} 
\newcommand{\C}{\mathbb C} 
\newcommand{\e}{\varepsilon} 
\newcommand{\free}{\ast}
\newcommand{\tensor}{\otimes}
\newcommand{\gluedfree}{\mathbin{\tilde{\free}}}
\newcommand{\gluedtensor}{\mathbin{\tilde{\tensor}}}
\newcommand{\N}{\mathbb N} 
\newcommand{\R}{\mathbb R} 
\newcommand{\tr}{\operatorname{tr}}
\newcommand{\im}{\operatorname{Im}}

\newcommand{\Z}{\mathbb Z} 

\newcommand{\Mu}{\mathrm{M}}
\newcommand{\Kappa}{\mathrm{K}}

\newcommand{\YDaction}{\mathbin{\vartriangleleft}}

\newcommand{\orcidlogo}{{\includegraphics[width=\fontcharht\font`W]{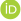}}}

\newtheorem{mainresult}{Theorem}

\Crefname{mainresult}{Theorem}{Theorems}
\newtheorem{theorem}{Theorem}[section]
\newtheorem{lemma}[theorem]{Lemma}
\newtheorem{corollary}[theorem]{Corollary}
\newtheorem{proposition}[theorem]{Proposition}

\theoremstyle{definition} 
\newtheorem{definition}[theorem]{Definition}
\newtheorem{remark}[theorem]{Remark}
\newtheorem{example}[theorem]{Example}
\newtheorem{observation}[theorem]{Observation}

\title[Free resolutions for free unitary quantum groups]{Free resolutions for free unitary quantum groups and universal cosovereign Hopf algebras}

\thanks{This work grew out of the research group meeting ``Cohomological properties of easy quantum groups'', conducted
02.11.2021 -- 08.11.2021 in B\k{e}dlewo, which was supported by the Stefan Banach International Mathematical Center, the National Science Centre and the Polish National Agency for Academic Exchange.}


\author{Isabelle Baraquin}
\address[I.B.]{Universit\'e de Franche-Comt\'e, CNRS, UMR 6623, LmB, F-25000 Besan\c{c}on, France, \href{https://orcid.org/0000-0002-4365-8270}{\orcidlogo\,0000-0002-4365-8270}
}
\email{isabelle.baraquin@univ-fcomte.fr}

\author{Uwe Franz}
\address[U.F.]{Universit\'e de Franche-Comt\'e, CNRS, UMR 6623, LmB, F-25000 Besan\c{c}on, France,
\href{https://orcid.org/0000-0001-8586-8025}{\orcidlogo\,0000-0001-8586-8025}}
\email{uwe.franz@univ-fcomte.fr}
\thanks{U.F.\ was supported by the ANR Project No.\ ANR-19-CE40-0002.}

\author{Malte Gerhold}
\address[M.G.]{Saarland University, Fachbereich Mathematik\\\href{https://orcid.org/0000-0003-4029-1108}{ 
   \orcidlogo\,0000-0003-4029-1108}}
   \email{gerhold@math.uni-sb.de}
\thanks{The work of MG was supported by the German Research Foundation (DFG) grant nos 465189426 and 397960675; it was carried out as a postdoctoral researcher at Saarland University, during the tenure of an ERCIM “Alain Bensoussan” Fellowship Programme at NTNU Trondheim, as a guest researcher at Saarland University in the scope of the SFB‐TRR 195, and as a postdoctoral scientific employee at University of Greifswald.}

\author{Anna Kula}
\address[A.K.]{Instytut Matematyczny, Uniwersytet Wrocławski, pl. Grunwaldzki 2, 50-348 Wrocław, Poland, \href{https://orcid.org/0000-0001-9688-6915}{\orcidlogo\,0000-0001-9688-6915}}
\thanks{A.K.\ was supported by the Polish National Science Center grant SONATA 2016/21/D/ST1/03010 and by the Polish National Agency for Academic Exchange in frame of POLONIUM program PPN/BIL/2018/1/00197/U/00021. }
\email{anna.kula@math.uni.wroc.pl}

\author{Mariusz Tobolski}
\address[M.T.]{Instytut Matematyczny, Uniwersytet Wrocławski, pl. Grunwaldzki 2, 50-348 Wrocław, Poland, \href{https://orcid.org/0000-0001-7738-8515}{\orcidlogo\,0000-0001-7738-8515}}
\thanks{M.T. was supported by the Polish National Science Center grant SONATINA 2020/36/C/ST1/00082.}
\email{mariusz.tobolski@math.uni.wroc.pl}

\date{}

\begin{document}
\begin{abstract}
    We find a finite free resolution of the counit of the free unitary quantum groups of van Daele and Wang and, more generally, Bichon's universal cosovereign Hopf algebras with a generic parameter matrix. This allows us to compute Hochschild cohomology with 1-dimensional coefficients for all these Hopf algebras. In fact, the resolutions can be endowed with a Yetter-Drinfeld structure. General results of Bichon then allow us to compute also the corresponding bialgebra cohomologies. Finding the resolution rests on two pillars. We take as a starting point the resolution for the free orthogonal quantum group presented by Collins, H{\"a}rtel, and Thom or its algebraic generalization to quantum symmetry groups of bilinear forms due to Bichon. Then we make use of the fact that the free unitary quantum groups and some of its non-Kac versions can be realized as a glued free product of a (non-Kac) free orthogonal quantum group with $\Z_2$, the finite group of order 2. To obtain the resolution also for more general universal cosovereign Hopf algebras, we extend Gromada's proof from compact quantum groups to the framework of matrix Hopf algebras. As a byproduct of this approach, we also obtain a projective resolution for the freely modified bistochastic quantum groups. Only a special subclass of free unitary quantum groups and universal cosovereign Hopf algebras decompose as a glued free product in the described way. In order to verify that the sequence we found is a free resolution in general (as long as the parameter matrix is generic, two conditions which are automatically fulfilled in the free unitary quantum group case), we use the theory of Hopf bi-Galois objects and Bichon's results on monoidal equivalences between the categories of Yetter-Drinfeld modules over universal cosovereign Hopf algebras for different parameter matrices.
\end{abstract}
\maketitle
\vspace*{-2.1em}
\tableofcontents

\section{Introduction}

Interest in homological properties of Hopf algebras in general and Hopf algebras of compact quantum groups in particular stems from several directions. On the one hand, geometric concepts such as dimension can be transferred to these rather abstract algebraic objects. On the other hand, the first and the second Hochschild cohomology with trivial coefficients, $\mathrm{H}^1(A,\C_\e),\mathrm{H}^2(A,\C_\e)$, play an important role in the classification of Lévy processes on a Hopf algebra $A$ in the sense of Schürmann \cite{schurmann}.\footnote{Indeed, L\'evy processes on a compact quantum group $\mathbb{G}$ are classified by their generating functionals $\phi\colon \mathrm{Pol}(\mathbb{G})\to\mathbb{C}$, which can always be completed to so-called Sch{\"u}rmann triples $(\phi,\eta,\phi)$, where $\pi\colon \mathrm{Pol}(\mathbb{G})\to B(H)$ is a $*$-homomorphism with values in the $*$-algebra of bounded linear operators on some Hilbert space $H$, $\eta\colon \mathrm{Pol}(\mathbb{G})\to H$ is a Hochschild one-cocycle (w.r.t.\ a certain $\mathrm{Pol}(\mathbb{G})$-bimodule structure on $H={_\pi H_\e}$), and the bilinear map $\mathrm{Pol}(\mathbb{G})\otimes\mathrm{Pol}(\mathbb{G})\ni a\otimes b \mapsto - \langle\eta(a^*),\eta(b)\rangle\in\mathbb{C}$ is the Hochschild coboundary of $\phi$. See \cite{schurmann,fgt15,hunt} for a detailed description.} 
If one is only interested in the first and the second Hochschild cohomology with specific coefficient bimodules, a viable approach is to calculate cocycles and coboundaries explicitly. This was successfully done in several cases, for example, in \cite{BFG17}, where vanishing of $\mathrm H^1(A,\C_\e),\mathrm H^2(A,\C_\e)$ for \emph{quantum permutation algebras} is proved, or in \cite{Mang23pre}, where Mang calculated $\mathrm H^1(A,\C_\e)$ for the Hopf algebras of arbitrary \emph{easy unitary quantum groups}. In \cite{VanDaeleWang1996}, van Daele and Wang constructed two families of \emph{universal} compact quantum groups: the \emph{free unitary quantum groups} $U_K^+$ and \emph{free orthogonal quantum groups} $O_L^+$, where $K,L\in\operatorname{GL}_n(\C)$ are complex invertible matrices and $\overline L L\in \C I_n$ is a multiple of the $n\times n$ identity matrix $I_n$.\footnote{Our definitions below follow conventions of Banica in \cite{Banica1996,Banica1997-free_unitary_quantum_group} and differ slightly from the definitions in \cite{VanDaeleWang1996}. The free unitary quantum groups are universal in the sense that every compact matrix quantum group is a quantum subgroup of some $U_K^+$.} We denote the associated Hopf algebras $\Pol{U_K^+}$ and $\Pol{O_L^+}$, respectively; for $n\in\N$, we refer to $U_n^+:=U_{I_n}^+$ and $O_n^+:=O_{I_n}^+$ as the free unitary and free orthogonal quantum group of \emph{Kac type}. In \cite{DFKS18,DFKS23}, the first and the second Hochschild cohomologies for the Hopf algebras $\Pol{U_K^+}$ and $\Pol{O_L^+}$ are determined with concrete formulas for cocycles and coboundaries for most parameter matrices $K,L$, however, the case of $U_K^+$ where $K^*K$ has three distinct Eigenvalues in ``geometric progression'' ($q^{-1},1,q$, with $q\in\R^+$) remained open.

By a general procedure, described in detail by Bichon \cite[Section 2.2]{Bic13}, all Hochschild cohomologies $\mathrm{H}^k(A,M)$ for arbitrary coefficient bimodules $M$ can be deduced from a resolution of the counit of a Hopf algebra $A$. In \cite{CHT09}, Collins, Härtel, and Thom present a finite resolution of the counit of the Hopf algebra $\Pol{O_n^+}$ associated with the Kac type free orthogonal quantum group $O_n^+$, allowing them to deduce information about Hochschild cohomology and $\ell^2$-Betty numbers. Bichon \cite{Bic13} generalized this to a class of Hopf algebras denoted by $\mathcal B(E)$, $E\in\operatorname{GL}_n(\C)$, (see \Cref{sec:B(E)_and_H(F)}) which includes the Hopf algebras of the possibly non-Kac free orthogonal quantum groups $O_L^+$, $\overline L L\in \C I_n$, defined by Banica \cite{Banica1996}. In this work, we use Bichon's free resolution for $\mathcal B(E)$ as a starting point to find projective or free resolutions for a number of related Hopf algebras; most notably, we find a finite free resolution for the \emph{free unitary quantum groups} $U_K^+$ and, more generally, its algebraic counterpart, the \emph{universal cosovereign Hopf algebras} $\mathcal H(F)$ introduced by Bichon in \cite{Bichon01} for all matrices $F\in\operatorname{GL}_n(\C)$ which are \emph{generic} (a certain condition equivalent to the Hopf algebras $\mathcal H(F)$ and $\mathcal B(E)$ for $F=E^tE^{-1}$ being cosemisimple). For the universal cosovereign Hopf algebras, Bichon was able to compute partial information on the Hochschild and Gerstenhaber-Schack cohomologies in \cite{Bichon18}, for example their corresponding cohomological dimensions, but even Hochschild cohomology for trivial coefficients was not completely known. With knowledge of the resolution, we can not only improve the computations in \cite{DFKS18,DFKS23} of the first and second Hochschild cohomologies for the Hopf algebras $\Pol{O_L^+}$ and 
$\Pol{U_K^+}$, but also gather a lot information about higher cohomologies of universal quantum groups and universal cosovereign Hopf algebras. 

Let us now summarize the main results of this paper.

\begin{mainresult}\label{main-result:resolution-for-H(F)}
  Let $F\in\operatorname{GL}_n(\C)$ be generic. Then the counit of $\mathcal H(F)$ has a finite free resolution of length 3. In particular, we have a finite free resolution of the counit of $\Pol{U_K^+}$ of length 3 for arbitrary $K\in \operatorname{GL}_n(\C)$.
\end{mainresult}

We start out by establishing such a resolution in the special case where $F$ is an \emph{asymmetry}, i.e.\ $F=E^{t}E^{-1}$ for some $E\in\operatorname{GL}_n(\C)$ (\Cref{sec:resolution-A(E),sec_glued,sec:resolution-H(F)}). In that case, $\mathcal H(F)$ can be realized as a \emph{glued free product} of $\mathcal B(E)$ with the group algebra $\C\Z_2$, which is a Hopf subalgebra of the free product $\mathcal A(E):=\mathcal B(E)\free \C\Z_2$. We can use the resolution for $\mathcal B(E)$ to first construct a projective resolution of $\mathcal A(E):=\mathcal B(E)\free \C\Z_2$ and apply a result of Chirvasitu \cite{Chirvasitu14} to turn it into a free resolution for $\mathcal H(F)$. This part requires a generalization of the work of Gromada \cite{Gromada2022} 
from compact matrix quantum groups to matrix Hopf algebras. 

Using the theory of Yetter-Drinfeld modules, we can finally show in \Cref{sec:YD-resolutions} that the formulas we obtain also establish a free resolution if $F$ is not necessarily of the form $F=E^{t}E^{-1}$, all we need is that $F$ is generic. 

As an application of our resolution, we calculate Hochschild cohomology for one-dimensional bimodules (\Cref{sec:Hochschild-cohomology-H(F)}) and bialgebra cohomology (\Cref{sec:bialgebra-cohomology}) of the Hopf algebras $\mathcal H(F)$ with generic $F$. We view $\mathcal H(F)$ as a matrix Hopf algebra with fundamental corepresentation $u$ and denote $v=(u^{-1})^t$.
\begin{mainresult}\label{main-result:Hochschild}
Let $F\in \operatorname{GL}_n(\C)$ be generic, and let $\tau\colon \mathcal H(F) \to \C$ be a character. Set $S:=\tau(u)$,  $T:=\tau(v)=S^{-t}$ and
denote by $\mathcal{K}$ the space of matrices commuting with $F^tS$. Let $$d=\operatorname{dim}\, \mathcal{K}, \quad p=\begin{cases}
    1 & \mbox{ if } S=T=I_n\\
    0 & \mbox{ otherwise; }
\end{cases}, \quad
t=\begin{cases}
    1 & \mbox{ if } F^2=\alpha T \mbox{ for some $\alpha\in\C$ }\\
    2 & \mbox{ otherwise. }
\end{cases}
$$
Then the Hochschild cohomology for $\mathcal H=\mathcal H(F)$ with 1-dimensional coefficients is given as follows:
\begin{align*}
    \dim \mathrm{H}^0(\mathcal H, {_\e\C_\tau}) &= p ,
    &\dim \mathrm{H}^1(\mathcal H, {_\e\C_\tau}) &= d+p-1 ,
    \\
    \dim \mathrm H^2(\mathcal H, {_\e\C_\tau}) &= d-t ,
    &\dim \mathrm H^3(\mathcal H, {_\e\C_\tau}) &= 2-t ,
\end{align*}
and $\dim \mathrm H^k(\mathcal H, {_\e\C_\tau}) = 0$ for all $ k \geq 4 $.
\end{mainresult}
In particular, for $F = K^*K = \operatorname{diag}(q^{-1},1,q)$, $q\in\R^+$, it turns out that the second Hochschild cohomology of $\Pol{U_{K}^+}\cong\mathcal H(F)$ with trivial coefficients is one-dimensional; this was the missing piece in \cite{DFKS23}, where the second Hochschild cohomology of all other free unitary quantum groups was determined. \Cref{main-result:Hochschild} also yields explicitly for every generic matrix $F$ a (one-dimensional) bimodule $M$ with $\mathrm{H}^3(\mathcal H(F),M)\neq0$, which confirms that the cohomological dimension of $\mathcal H(F)$ is three; a fact Bichon conjectured in \cite[Remark 5.15]{Bichon18} and proved by more abstract considerations in \cite[Theorem 8.1]{Bichon22}.
\begin{mainresult}\label{main-result:bialgebra-cohomology}
Let $F\in \operatorname{GL}_n(\C)$ be generic.
  Then the bialgebra cohomology of $\mathcal H=\mathcal H(F)$ is given as follows:
\begin{align*}
    \dim \mathrm{H}_{\mathrm b}^0(\mathcal H) &= 1 ,
    &\dim \mathrm{H}_{\mathrm b}^1(\mathcal H) &= 1 ,\\
    \dim \mathrm H_{\mathrm b}^2(\mathcal H) &= 0, 
    &\dim \mathrm H_{\mathrm b}^3(\mathcal H) &= 1, 
\end{align*}
and $\dim \mathrm H_{\mathrm b}^k(\mathcal H) = 0$ for all $ k \geq 4 $.
\end{mainresult}
In particular, the second bialgebra cohomology is trivial for all free unitary quantum groups $U_K^+$, $K\in\operatorname{GL}_n(\C)$, because $\Pol{U_K^+}\cong \mathcal H(K^*K)$ and $\tr (K^*K)>0$.

\begin{remark}
    After the completion of this paper, Julien Bichon communicated to us that our \Cref{thm:H(F)=B(E)gluedCZ2}, and hence \Cref{main-result:resolution-for-H(F),main-result:Hochschild}, hold with the weaker assumption of normalizability instead of genericity, cf.\ \cite[Prop.~4.3]{Bichon23pre}. 
\end{remark}

\section{Preliminaries}

\subsection{Notation and basic definitions}

$\mathbb{N}=\{1,2,\ldots\}$ denotes the set of positive integers. $\R$ and $\C$ denote the fields of real and complex numbers, respectively, and we write $\R^\times:=\R\setminus\{0\}$, $\C^\times:=\C\setminus\{0\}$.

All our vector spaces are over $\C$. For any vector space $X$, $M_{n,n'}(X)$ denotes the vector space of $n\times n'$-matrices with entries from $X$. The entries of $A\in M_{n,n'}(X)$ are denoted by $A_{ij}$ ($i=1,\ldots n; j=1,\ldots n'$).  The transpose of $A\in M_{n,n'}(X)$ is denoted by $A^t$, $(A^{t})_{ij}=A_{ji}$. If $n=n'$, we simply write $M_n(X)$. In case $X$ carries an algebra structure, matrix multiplication $M_{n,n'}(X)\times M_{n',n''}(X)\to M_{n,n''}(X)$ is defined, $(AB)_{ij}:=\sum A_{ik}B_{kj}$, in particular, $M_n(X)$ inherits an algebra structure. For a square matrix $A\in M_n(X)$, its \emph{trace} is $\tr(A)=\sum_{i=1}^n A_{ii}$. Obviously, $\tr(A)=\tr(A^t)$. Note that $\tr(AB)=\tr(A^t B^t)$ also holds for all $A\in M_{n,n'}(X),B\in M_{n',n}(X)$ even if $X$ is noncommutative (in which case $\tr(AB)$ and $\tr(BA)$ might not coincide). Given a map $f\colon X\to Y$ and a matrix $A\in M_{n,n'}(X)$, we usually write $f(A)$ for the entrywise application, i.e.\ $f(A)_{ij}:=f(A_{ij})$. There is a notable exception from this convention: If $X$ is a $*$-algebra and $A\in M_{n,n'}(X)$, we write $\overline A$ for the entrywise adjoint $(\overline A)_{ij}:=(A_{ij})^*$ and put $A^*:=(\overline A)^t$; $M_n(X)$ is then considered a $*$-algebra with involution $A\mapsto A^*$. 

The group of complex invertible matrices is denoted by $\operatorname{GL}_n(\C)\subset M_n(\C)$. For $A\in \operatorname{GL}_n(\C)$, we write $A^{-t}:=(A^{-1})^{t}=(A^{t})^{-1}$. (Be aware that $(A^{-1})^{t}=(A^{t})^{-1}$ does not hold in general when $A$ has entries in a noncommutative algebra.)

For $A$ an algebra, we denote the category of its right modules by $\mathcal{M}_A$, and the class of morphisms between two $A$-modules $M,N$ by
\[
\mathrm{hom}_A(M,N) = \{f\colon M\to N \text{ is }A\text{-linear}\}.
\]
Similarly, for $C$ a coalgebra, we write $\mathcal{M}^C$ for the category of its right comodules and
\[
\mathrm{hom}^C(M,N)= \{ f\colon M\to N \text{ is }C\text{-colinear}\}
\]
for the class of morphisms between two comodules $M,N$. When we work with the category $\mathcal{YD}^H_H$ of Yetter-Drinfeld modules (defined in \Cref{subsec:YD-generalities}) of a Hopf algebra $H$, whose objects are simultaneously $H$-modules and $H$-comodules, then the class of morphisms between two objects $M,N\in\mathcal{YD}^H_H$ is
\[
\mathrm{hom}^H_H (M,N)= \{ f\colon M\to N \text{ is } H\text{-linear and }H\text{-colinear}\}.
\]

For a right $A$-module $M$ and a left $A$-module $N$ over an algebra $A$, the module tensor product is the quotient space $M\otimes_A N= M\otimes N/(m\otimes an=ma\otimes n)$. Dually, for a right $C$-module $M$ and a left $C$-module $N$ over a coalgebra $C$ with coactions denoted by $\gamma_M\colon M\to M\otimes C,\gamma_N\colon N\to C\otimes N$, the cotensor product is the subspace $M\mathbin{\Box}_C N:=\{X\in M\otimes N: \gamma_M\otimes \mathrm{id}(X)=\mathrm{id}\otimes\gamma_N(X)\}\subset M\otimes N$, cf.\ \cite[Definition 8.4.2]{Montgomery1993}. When the coalgebra $C$ is understood from the context without doubt, we omit the subscript $C$ and simply write $M\mathbin\Box N$. 

Actions on modules are mostly written by juxtaposition without any symbol or signalized by a dot.
Comultiplication and counit of a coalgebra are typically denoted by $\Delta$ and $\e$, respectively. The antipode of a Hopf algebra is denoted by $S$. The coaction of a coalgebra on a comodule is typically denoted by $\gamma$.
We will use the $\Sigma$-free Sweedler notation, i.e.\ in a coalgebra $C$ which abbreviates the coproduct $\Delta(a)=\sum_i a_{(1),i}\otimes a_{(2),i}$ as $a_{(1)}\otimes a_{(2)}$; also, for a comodule $M\in \mathcal M^C$, the coaction $\gamma(m)=\sum_i m_{(0),i}\otimes m_{(1),i}$ is written $m_{(0)}\otimes m_{(1)}$. 

For a Hopf algebra $H$, the \emph{trivial} right module / left module / bi-module is denoted by $\C_\e$ / ${_\e\C}$ / ${_\e\C_\e}$, respectively, and is defined as the vector space $\C$ with right and / or left action $zh=z\e(h), hz=\e(h)z$ for $h\in H,z\in\C$.

In several instances we write a direct sum of objects $V_k$ ($k=1,\ldots, n$) in some category of vector spaces (for example $\mathcal M_A$ or $\mathcal {YD}_A^A$) as a column \[\begin{pmatrix}
    V_1\\\vdots \\V_n
\end{pmatrix}:=\bigoplus_{k=1}^n  V_k;\] this way, a linear map $f\colon \bigoplus_{k=1}^n V_k \to \bigoplus_{\ell=1}^m W_\ell$ inherits a natural matrix decomposition into components $f_{\ell k}\colon V_k\to W_{\ell}$ such that $f(\bigoplus_k v_k)=\bigoplus_\ell \sum_k f_{\ell k} (v_k)$. 

For linear maps $f_k\colon V_k\to W_k$, $k=1,2$, we write 
\[f_1\oplus f_2\colon V_1\oplus V_2\to W_1\oplus W_2, \quad\begin{pmatrix}
    v_1\\v_2
\end{pmatrix}\mapsto \begin{pmatrix}
    f(v_1)\\f(v_2)
\end{pmatrix}\]
while for linear maps $f_k\colon V_k\to W$ with the same codomain $W$, we write (by slight abuse of notation)
\[f_1+f_2\colon V_1\oplus V_2\to W, \quad\begin{pmatrix}
    v_1\\v_2
\end{pmatrix}\mapsto f(v_1)+f(v_2).\]

For an algebra $A$ and an $A$-bimodule $M$, the \emph{Hochschild cohomology} of $A$ with \emph{coefficients} in $M$ is the cohomology of the complex 
\begin{align}\label{eq:Hochschild-defining-complex}
  0\rightarrow \hom (\C,M)
  \xrightarrow{\partial_0} \hom (A,M)
  \xrightarrow{\partial_1} \hom (A^{\otimes 2},M)
  \xrightarrow{\partial_2} \ldots 
  \xrightarrow{\partial_k} \hom (A^{\otimes (k+1)},M) 
  \xrightarrow{\partial_{k+1}} \ldots
\end{align}
with the coboundary map $\partial_k\colon \hom (A^{\otimes k},M) \to \hom 
(A^{\otimes (k+1)},M)$ defined as
\begin{multline*}
 \partial_k (f) (a_0\otimes a_1\otimes \ldots \otimes a_k) =
a_0.f(a_1\otimes \ldots \otimes a_k)+
\sum_{j=1}^{k} (-1)^{j} f(a_0\otimes a_1\otimes a_{j-1}a_j \otimes \ldots 
\otimes a_k)
\\ + (-1)^{k+1}
f(a_0\otimes a_1\otimes \ldots \otimes a_{k-1}).a_k;
\end{multline*}
i.e.\ the $k$-th Hochschild cohomology is the vector space
$$\mathrm H^k(A,M):=\ker \partial_k/\operatorname{im} \partial_{k-1}.$$

For a Hopf algebra $H$ and a Yetter-Drinfeld module $M\in\mathcal {YD}_H^H$, the Gerstenhaber-Schack cohomology and bialgebra cohomology are denoted by $\mathrm{H}_{\mathrm{GS}}(H,M)$ and $\mathrm H_{\mathrm b}(H):=\mathrm{H}_{\mathrm{GS}}(H,\C_\varepsilon)$, respectively (cf.\ \Cref{subsec:YD-generalities} for the corresponding definitions).

A \emph{resolution} of an $A$-module 
$N$ over an algebra $A$ is an exact sequence of $A$-modules
$$\ldots \rightarrow P_{n+1}\xrightarrow{\Phi_{n+1}} P_n \xrightarrow{\Phi_n} \ldots \xrightarrow{\Phi_2} P_1\xrightarrow{\Phi_1}
P_0 \xrightarrow{\Phi_0} N \rightarrow 0$$
and will be abbreviated $P_*\xrightarrow{\Phi} N \rightarrow 0$.
A resolution is called \emph{finite} if there is only a finite number of 
non-zero modules $P_k$; it is called \emph{free} if all the modules $P_k$ are 
free, and it is called \emph{projective} if all the modules $P_k$ are 
projective. (Recall that a module $P$ is \emph{projective} if there is a free module $F$ 
and a module $N$ such that $F\cong P\oplus N$; this happens if and only if 
the functor $\hom_A(P,-)$ is exact.) Given a Hopf algebra $H$, a resolution of the trivial right module $\C_\e$ is also called a \emph{resolution of the counit} or a \emph{resolution for $H$}.

We briefly fix some notation for the main Hopf algebras under consideration in this article, full definitions and more details how these are related to free unitary and free orthogonal quantum groups will be given in \Cref{sec:B(E)_and_H(F)}.
By $\mathcal{B}(E)$, with $E\in\operatorname{GL}_n(\mathbb{C})$ we denote the Hopf algebra, which, as an algebra, is the universal unital algebra defined by the relation $E^{-1}x^tEx=I=xE^{-1}x^tE$, where $x=(x_{ij})_{1\le i,j\le n}$ is the matrix of generators. For the special case $E=I_n$, one recovers $\mathcal{B}(I_n)\cong \mathrm{Pol}(O^+_n)$ the Hopf algebra of the free orthogonal quantum group $O_n^+$.
$\mathcal{A}(E)$ is defined as the free product $\mathcal{A}(E)=\mathcal{B}(E)*\mathbb{C}\mathbb{Z}_2$ and for $E=I$ this is isomorphic to the Hopf algebra of the freely modified bistochastic compact quantum group, $\mathcal A(I)\cong\Pol{B_{n+1}^{\#+}}$. Note that the algebra $\C\Z_2$ has one generator $g$ with relation $g^2=1$.  $\mathcal{H}(F)$, with $F\in\operatorname{GL}_n(\mathbb{C})$, is the universal unital algebra generated by the relations $u v^t = v^t u = I$, $vFu^tF^{-1}=Fu^tF^{-1}v=I$, where $u=(u_{ij})_{1\le i,j\le n}$, $v=(v_{ij})_{1\le i,j\le n}$, and $\mathcal{H}(F)$ has a natural Hopf algebra structure that makes $u$ and $v$ corepresentations. Under some assumption on $F$, we have $\mathcal{H}(F)=\mathcal{B}(E)\gluedfree\mathbb{C}\mathbb{Z}_2\cong\operatorname{Pol}(U^+_K)$, with $F=E^tE^{-1}=K^*K$, where $\gluedfree$ denotes the \emph{glued free product}, see \Cref{def-glued}. In particular, $\mathcal{H}(I_n)\cong \mathrm{Pol}(U_n^+)\cong \Pol{O_n^+}\gluedfree \C\Z_2$. With few local exceptions, which will be made explicit, $x$ / $g$ / $x,g$ / $u,v$ always denote the (matrices of) standard generators of $\mathcal B(E)$ / $\C\Z_2$ / $\mathcal A(E)$ / $\mathcal H(F)$.

We close this section with some more conventions and observations, summarized in a lemma, that will appear ubiquitously in our calculations.  

\begin{lemma}\label{lemma:traces-free-modules}
  Let $A$ be a unital algebra and let $\tau\colon A\to \C$ be a character, i.e.\ a nonzero multiplicative linear functional. (We also write $\tau$ for the map $M_n(A)\to M_n(\mathbb C)$ given by entrywise application, i.e.\ $\tau(D)=(\tau(D_{ij}))_{i,j=1}^{n}$ if $D=(D_{ij})_{i,j=1}^{n}\in M_n(A)$.)
  \begin{enumerate}
  \item  For all $C,D\in M_n(A)$, the following holds.
    \begin{gather*}
      \tr(\tau(C))=\tau(\tr(C))\\
      \tau(CD)=\tau(C\tau(D))=\tau(\tau(C)D)=\tau(C)\tau(D)\\
      \tau(\tr(CD))=\tau(\tr(DC))=\tau(\tr((DC)^t))=\tau(\tr(C^tD^t))
    \end{gather*}  
  \item Given a free $A$-module $F$ with finite module basis $(e_{j})_{j\in J}$, the mapping
    \begin{align*}
      \hom_A(F,\C_\tau)\ni f &\mapsto (f(e_j))_{j\in J} \in \C^J
    \end{align*}
    is an isomorphism of vector spaces with inverse
    \[\C^J\ni(\lambda_j)_{j\in J} \mapsto \lambda^\tau,\quad \lambda^\tau\left(\sum_{j\in J} e_j a_j\right)=\sum_{j\in J}\lambda_j\tau(a_j) \]
  \item 
    For a free module of the form $M_{n_1}(A)\oplus\ldots \oplus M_{n_k}(A)$ with basis \[\left(e_{ij}^{(r)}:r\in\{1,\ldots k\},i,j\in\{1,\ldots n_r\}\right)\] formed by the standard matrix units, the notation from (2) specializes to
    \[(\Lambda_1,\ldots,\Lambda_k)^\tau(M_1,\ldots, M_k)=\sum_{r=1}^k \tau\tr(\Lambda_r^t M_r).\]
  \item Let $F_1,F_2$ be free $A$-modules with bases $(e_{j}^{(1)})_{j\in J_1}, (e_{j}^{(2)})_{j\in J_2}$, respectively. Then, for every $\Phi\in\hom_A(F_1,F_2)$, there is a unique linear map $\Phi^*\colon \C^{J_2}\to \C^{J_1}$ such that, for all $\lambda\in \C^{J_2}$, 
    \[\lambda^\tau\circ\Phi=\bigl(\Phi^*(\lambda)\bigr)^{\tau}.\]
    In particular, denoting by \[-\circ\Phi\colon \hom_A(F_2,\C_\tau)\to \hom_A(F_1,\C_\tau),\quad f\mapsto f\circ\Phi,\] the precomposition with $\Phi$, the given isomorphisms $\hom(F_i,\C_\tau)\cong\C^{J_i}$ induce isomorphisms \[\ker(-\circ\Phi)\cong \ker(\Phi^*),\quad\im(-\circ\Phi)\cong\im(\Phi^*).\]
    
  \end{enumerate}
\end{lemma}

\begin{proof}\ 
  \begin{enumerate}
  \item After direct verification of the first two lines of equations, the third is obvious because application of $\tau$ turns the matrices into scalar matrices. Note that the trace property would not in general hold without application of $\tau$ if $A$ is noncommutative.
  \item Of course, any $f\in \hom_A(F,\C_\tau)$ is determined by its values on the basis elements, $A$-linearity yields
    \[f\left(\sum e_j.a_j\right)=\sum f(e_j).a_j=\sum f(e_j) \tau(a_j).\]
    This also shows that $f=\lambda^\tau$ with $\lambda_j:=f(e_j)$. On the other hand, $\lambda^\tau(e_j)=\lambda_j\tau(1)=\lambda_j$ proves that the given maps are mutually inverse isomorphisms.
  \item 
    If the free modules are given in matrix form and $\Lambda_r=(\lambda_{ij}^{(r)})_{i,j=1}^{n_r}\in M_{n_r}(\C)$, $M_r=(M_{ij}^{(r)})_{i,j=1}^{n_r}\in M_{n_r}(\C)$ for $r=1,\ldots, k$, we easily see that    \[(\Lambda_1,\ldots,\Lambda_k)^\tau(M_1,\ldots, M_k)=\sum_{r=1}^k \tau\tr(\Lambda_r^t M_r)=\tau\sum_{i,j,r} \lambda_{ij}^{(r)}m_{ij}^{(r)}=\sum_{r=1}^k \tau\tr(\Lambda_r^t M_r).\]
  \item Obvious.\qedhere
  \end{enumerate}
\end{proof}

\subsection{The Hopf algebras \texorpdfstring{$\mathcal B(E)$}{B(E)}, \texorpdfstring{$\mathcal A(E)$}{A(E)}, and \texorpdfstring{$\mathcal H(F)$}{H(F)} and related compact quantum groups}
\label{sec:B(E)_and_H(F)}

In this section we provide the definition of the three main families of Hopf algebras we are going to use throughout the paper. As detailed below, these are Hopf algebraic generalizations of the free orthogonal quantum groups, the freely modified bistochastic quantum groups, and the free unitary quantum group, respectively.

\begin{definition}\label{def:B(E)}
  Let $E\in \operatorname{GL}_n(\mathbb{C})$ be an invertible matrix, $n\geq2$. The Hopf algebra $\mathcal B(E)$ is the universal algebra generated by the entries $x_{ij}$ of a matrix of generators $x=(x_{ij})_{i,j=1}^n$ subject to the relations
  \begin{equation} \label{eq_relations_in_BE}
    E^{-1}x^tEx=I=xE^{-1}x^tE;
  \end{equation}
  here $x^t$ denotes the transpose of the matrix $x$, $(x^t)_{ij}:=x_{ji}$.
  The Hopf algebra structure on $\mathcal B(E)$ is given by
  \begin{align*}
    \Delta(x_{ij})&=\sum_kx_{ik}\otimes x_{kj},& \varepsilon(x_{ij})&=\delta_{ij},&S(x)&=E^{-1}x^tE.  \end{align*}
\end{definition}

The algebra $\mathcal B(E)$ corresponds to the quantum symmetry group of the bilinear form associated with $E$ (see \cite{DVLa90}). Let us note that $\mathcal B(E)$ is $\Z_2$-graded as it is a quotient of a naturally $\Z_2$-graded free algebra generated by the $x_{ij}$ by the even relations \eqref{eq_relations_in_BE}.

Clearly, for any $G\in\operatorname{GL}_n(\mathbb{C})$, the Hopf algebras $\mathcal{B}(E)$ and $\mathcal{B}(G^tEG)$ are isomorphic, via the unique isomorphism defined by $x\mapsto G^{-1}xG$.

Recall that the unique dense $*$-Hopf algebras related to compact quantum groups are called CQG-algebras, cf.\ \cite[Definition 11.9]{klimykschm}. See also \cite[Theorem 11.27]{klimykschm} for conditions which guarantee that $*$-Hopf algebra is a CQG-algebra.

If $\overline{E}E= \lambda I$ for some $\lambda \in \R^{\times}$, then $\B(E)$, with the involution determined by $\overline{x}=E^txE^{-t}$, is the CQG-algebra of the universal orthogonal quantum group $O^+_L$, defined in \cite{Banica1996}, with $L=E^t$.  The latter is defined as the universal unital $*$-algebra generated by $x=(x_{ij})_{i,j=1}^n$ subject to the relations
\[
x^*x=I=xx^*,\quad \overline{x}=L x L^{-1}.
\]
In general, the Hopf algebra $\mathcal B(E)$ admits an involution w.r.t.\ which it is a CQG-algebra if and only if there exist $M\in \operatorname{GL}_n(\C)$, $\lambda\in \R^{\times}$, and $\mu \in \C^*$ such that $\overline{M}M=\lambda I$, $M^tE^*M=E$, and  $\mu M^{-t}E$ is positive, see  \cite[Proposition 6.1 and the remarks below]{Bichon03}. The $*$-structure is given by $\overline{x}=MxM^{-1}$. In this case, $\mathcal{B}(E)$ is isomorphic to $\mathcal{B}(\widetilde{E})$, with $\widetilde{E}=KE^{-t}K^t=G^tEG$, where $K$ is any matrix such that $\mu M^{-t} E = K^*K$ and $G=E^{-1}K^t$. The choice of $K$ implies $ E =\mu^{-1} M^t K^*K$, and therefore
\begin{align*}
\overline{\widetilde{E}}\widetilde{E} &= \overline{K}E^{*,-1} K^* K E^{-t} K^t = \overline{K} \left(\overline{\mu}\overline{M}^{-1} (K^*K)^{-1}\right) K^* K \left(\mu M^{-1} \overline{K}^{-1}K^{-t}\right) K^t \\
&= |\mu|^2 \overline{K} \left(M\overline{M}\right)^{-1}\overline{K}^{-1} = \frac{|\mu|^2}{\lambda} I,
\end{align*}
which implies that $\mathcal{B}(\widetilde{E})$ (and therefore also $\mathcal{B}(E)$) is isomorphic to $\mathrm{Pol}(O_L^+)$, with $L=\widetilde{E}^t=KE^{-1}K^t$.

\begin{remark}
    While not directly relevant for the remainder of this article, it is worthwhile noting that the possibilities for the parameters $\mu$ and $\lambda$ are quite restricted. Indeed, in the notation of the preceding paragraph,  $M^tE^*M=E$ is equivalent to $E^{*,-1} = M E^{-1} M^t$ and this yields 
    \[\lambda\mu I=\overline M ME^{-1}M^t K^*K=\overline M E^{*,-1} K^*K=\overline \mu I,\]
    i.e., $\lambda\mu=\overline \mu$. Since $\lambda \in\mathbb R$, there are two distinct cases: either $\lambda=1$ and $\mu=\overline{\mu}\in\mathbb R$, or $\lambda=-1$ and $\mu=-\overline{\mu}\in i\mathbb R$. Cf.\ the classification of free orthogonal quantum groups as presented by De Rijdt in \cite[Remark 1.5.2]{DeRijdt2007}. 
\end{remark}

\begin{definition}
  For $F\in \operatorname{GL}_n(\mathbb{C})$, $n\geq2$, the {\em universal cosovereign Hopf 
    algebra} $\mathcal H(F)$ is the universal algebra generated by the entries of 
  $u=(u_{ij})_{i,j=1}^n$ and $v=(v_{ij})_{i,j=1}^n$ subject to the relations
  \begin{equation} \label{eq_relations_in_HF}
    uv^t=v^tu=I,\qquad vFu^tF^{-1}=Fu^tF^{-1}v=I.
  \end{equation}
  The Hopf algebra structure is defined by
  \begin{align*}
    \Delta(u_{ij})=\sum_ku_{ik}\otimes u_{kj},&\quad \Delta(v_{ij})=\sum_k 
                                                v_{ik}\otimes v_{kj},
    \\
    \varepsilon(u_{ij})=\varepsilon(v_{ij})=\delta_{ij},&\quad S(u)=v^t,\quad 
                                                          S(v)=Fu^tF^{-1}.
  \end{align*}
\end{definition}

The notion of universal cosovereign Hopf algebras was introduced in \cite{Bichon01}. 
Obviously, $\mathcal H(F)=\mathcal H(\lambda F)$ for any $\lambda\in \C^*$. Moreover, $\mathcal H(F)\cong \mathcal H(F^{-t})$ and $\mathcal H(F)\cong\mathcal H(GFG^{-1})$ for every $G\in\operatorname{GL}_n(\C)$, see \cite[Proposition 3.3]{Bichon01}.

Recall that for $K\in \operatorname{GL}_n(\C)$ the free unitary compact quantum group (first defined slightly differently by van Daele and Wang in \cite{VanDaeleWang1996}) is the pair $U_K^+=(A_u(K),u)$ with $A_u(K)$ being the universal unital C$^*$-algebra generated by the entries of the matrix $u=(u_{ij})_{i,j=1}^n$, subject to the relations making $u$ and $K\overline{u}K^{-1}$ unitaries \cite[D{\'e}f.~1]{Banica1997-free_unitary_quantum_group}. The prescription $u\mapsto u, \overline u\mapsto v$ establishes a Hopf algebra isomorphism from the dense $*$-Hopf subalgebra $\Pol{U_K^+}=\operatorname{*-alg}(u_{ij}:i,j=1,\ldots,n)\subset A_u(K)$ to $\mathcal H(K^*K)$; this is well-known, cf.\ for example from \cite[Remarque following D{\'e}f.~1 ]{Banica1997-free_unitary_quantum_group}.
When $F=K^*K$ is positive, we thus recover the (algebraic version of the) free unitary compact quantum groups with involution determined by $\overline u=v$.
More generally, $\mathcal{H}(F)$ admits an involution that turns it into a CQG-algebra if and only if there exist $G\in \operatorname{GL}_n(\mathbb{C})$ and $\mu\in\mathbb{C}^\times$ such that $\mu GFG^{-1}$ is positive, see \cite[Proposition 3.6]{Bichon01}, in which case we always have $\mathcal H(F)\cong \mathcal H(\mu GFG^{-1})\cong \Pol{U_K^+}$ for $K$ such that $\mu GFG^{-1}=K^*K$.

A matrix $F\in\mathrm{GL}_n(\C)$, $n\geq2$, is called:
\begin{itemize}
	\item {\em normalizable} if either $\operatorname{tr}(F)\neq 0$ and $\operatorname{tr}(F^{-1})\neq 0$, or $\operatorname{tr}(F)=0=\operatorname{tr}(F^{-1})$,
	\item {\em generic} if it is normalizable and the solutions of the equation
	\begin{align}\label{eq:generic}
	q^2-\sqrt{\operatorname{tr}(F)\operatorname{tr}(F^{-1})}q+1&=0    
	\end{align}
	are not roots of unity of order $\geq 3$,
	\item an {\em asymmetry} if there exists $E\in \operatorname{GL}_n(\mathbb{C})$ such that $F=E^tE^{-1}$. 
\end{itemize}
An asymmetry is automatically normalizable, but need not be generic. A positive definite $n\times n$-matrix ($n\geq2$) is automatically generic,\footnote{It is elementary to show that $\tr(F)\tr(F^{-1})\geq n^2\geq4$ (we assumed $n\geq2$), therefore all solutions of \eqref{eq:generic} are real and, in particular, not roots of unity of order $\geq3$.} but need not be an asymmetry.
\begin{remark}\label{rem:generic-cosemisimple}
  For $F\in\operatorname{GL}_n(\C)$, $\mathcal H(F)$ is cosemisimple if and only if $F$ is generic \cite[Theorem 1.1(ii)]{Bichon07}. For $E\in \operatorname{GL}_n(\C)$, $\mathcal B(E)$ is cosemisimple if and only if the associated asymmetry $E^tE^{-1}$ is generic (this follows from \cite{KondratowiczPodles}, see \cite{Bichon03} for details).
\end{remark}
\begin{remark}
    In \cite[Theorem 1]{CortellaTignol2002}, Cortella and Tignol characterize asymmetries. In general, a rather complicated condition on generalized Eigenspaces appears. In the special case where $F\in M_n(\C)$ is diagonalizable, however, those conditions become (almost) trivial and it easily follows that $F$ is an asymmetry if and only if $F$ is invertible with $F$ and $F^{-1}$ conjugate in $M_n(\C)$ and the Eigenspace for Eigenvalue $-1$ has even dimension. For a positive definite matrix $F$, the only condition is that $F$ and $F^{-1}$ be conjugate. 

    Furthermore, two matrices $E,E'\in \operatorname{GL}_n(\C)$ are congruent in $M_n(\C)$, i.e., $E'=G^tEG$ for some $G\in \operatorname{GL}_n(\C)$, if and only if their asymmetries $F=E^tE^{-1}$ and $F'={E'}^t {E'}^{-1}$ are conjugate in $M_n(\C)$ \cite[Lemma 2.1]{HornSergeichuk2006}. Therefore, also $\mathcal B(E)$ depends, up to isomorphism, only on the conjugacy class of the associated asymmetry $F=E^tE^{-1}$.
\end{remark}

Let $A,B$ be Hopf algebras. Then the free product $A*B$ of the underlying unital algebras is again a Hopf algebra, where the comultiplication is extended multiplicatively; note the subtlety that one has to identify $A\otimes A$ and $B\otimes B$ with the corresponding subsets of $(A*B)\otimes (A*B)$,  see \cite[Theorem~2.2]{Agore11} for details.

\begin{remark}
If $A$ and $B$ are cosemisimple, then so is $A*B$ (see \cite[After Proposition 4.1]{Bichon07}). 
\end{remark}

\begin{definition}
  For $E\in \operatorname{GL}_n(\C)$, we define $\A(E):=\mathcal B(E)\free \C\Z_2$.
\end{definition}
If $F=E^tE^{-1}$ is generic, then $\A(E)$ is cosemisimple as the free product of two cosemisimple Hopf algebras.
In the special case $E=I$, we recover the Hopf algebra of the freely modified bistochastic compact quantum group $\mathcal A(I)= \mathcal B(I)\free \C\Z_2\cong \Pol{O_n^+*\Z_2}\cong\Pol{B_{n+1}^{\#+}}$, see
\cite[Remark 6.19.]{TerragoWeber2017} or \cite[Lemma 2.3 and Remark 2.4]{Weber13}.

\subsection{Hochschild cohomology of Hopf algebras via resolutions}

If one is interested in the Hochschild cohomology of a Hopf algebra $H$, besides working directly with the defining complex \eqref{eq:Hochschild-defining-complex},  there is another approach using resolutions (see \cite[Proposition 2.1]{Bic13} and the references therein), which we adopt in this paper: 

\begin{theorem} \label{thm:hom-spaces}
Let $A$ be a Hopf algebra, $M$ an $A$-bimodule, and $P_*\xrightarrow{\Phi} \C_\e \rightarrow 0$ a projective resolution of the counit. Denote by $M''$ the right $A$-module given by the vector space $M$ with the right action $m\leftarrow a:=S(a_{(1)}).m.a_{(2)}$.   
Then the Hochschild cohomology of $A$ with the coefficients in $M$ is the cohomology of the complex
\[
  0\longrightarrow \hom_A (P_0,M'')
  \xrightarrow{-\circ \Phi_1} \hom_A (P_1,M'')
  \xrightarrow{-\circ \Phi_2} \hom_A (P_2,M'')
  \xrightarrow{-\circ \Phi_3} \ldots,
\] i.e.\ 
\[
  \mathrm H^n(A,M)
  \cong \operatorname{Ext}^n_A(\C_\e,M'')
  =\ker (-\circ \Phi_{n+1})/\operatorname{im}(-\circ \Phi_{n}).
\]
\end{theorem}

It can be 
proved that the spaces $\operatorname{Ext}^n_A(\C_\e,M'')$ do not depend on the choice of the projective resolution $P_*\xrightarrow{\Phi} \C_\e \rightarrow 0$. 

\begin{remark}\label{1-dim-coefficients}
  In the sequel we will focus on the \emph{Hochschild 
    cohomology with 1-dimensional coefficients}, which means that the bimodule $M=\C$ is a one-dimensional vector space. In that case, the left and right action are necessarily given by $a.z.b=\sigma(a)z\tau(b)$ for a pair of characters (=unital homomorphisms) $\sigma,\tau\colon A\to \C$. The bimodule $\C$ with actions given by $\sigma$ and $\tau$ in that way is denoted by $_\sigma\C_\tau$. Note that $(_\sigma\C_\tau)''=\C_{(\sigma\circ S)*\tau}$; indeed,
  \[
    z\leftarrow a
    =S(a_{(1)}).z.a_{(2)}
    =\sigma\big( S(a_{(1)})\big) z \tau(a_{(2)})
    =z\sigma\big( S(a_{(1)}) \big)\tau(a_{(2)})
    =z ((\sigma\circ S)*\tau)(a).
  \]

  In particular, $(_\sigma\C_\tau)''=(_\e\C_{(\sigma\circ S)*\tau})''$. (Recall that $\sigma\circ S$ is the convolution inverse of $\sigma$.) Thus, we can and we will assume without loss of generality that $\sigma=\e$ from the start and, accordingly, restrict our considerations to the bimodules $_\e\C_\tau$ and the associated right-modules $(_\e\C_\tau)''=\C_\tau$. 
\end{remark}

\subsection{Free resolution and Hochschild cohomology for \texorpdfstring{$\mathcal B(E)$}{B(E)}}

The results of this paper heavily rely on the following fact for $\mathcal B(E)$, 
due to Bichon \cite{Bic13}. The Kac case ($E=I$) was treated first by 
Collins, H\"artel and Thom in \cite{CHT09}. 
\begin{theorem}[\cite{Bic13}]
 Let $E\in \operatorname{GL}_n(\C)$ and $\mathcal B=\mathcal B(E)$. Then the sequence
\begin{align}\label{resO+E}
	0\rightarrow \mathcal B\xrightarrow{\Phi^{\mathcal B}_3}M_n(\mathcal B) 
\xrightarrow{\Phi^{\mathcal B}_2}M_n(\mathcal B)\xrightarrow{\Phi^{\mathcal B}_1} \mathcal B
\xrightarrow{\varepsilon}\mathbb C_\varepsilon\rightarrow 0 
\end{align}
with the maps
\begin{align*}
	\Phi^{\mathcal B}_3 & (1) =\sum_{j,k=1}^n e_{jk}\otimes \big( 
(Ex E^{-t})_{jk}-(E^tE^{-1})_{jk}), \\
	\Phi^{\mathcal B}_2 & (e_{jk}\otimes 1) = e_{jk}\otimes 1 + \sum_{l,m=1}^n e_{lm} \otimes 
 (x E^{-t})_{jm} E_{kl},\\
	\Phi^{\mathcal B}_1 & (e_{jk}\otimes 1) = x_{jk}-\delta_{jk},
\end{align*}
is a finite free resolution of the counit of  $\mathcal B$.
\end{theorem}

It will be very useful for the calculations to rewrite the $\Phi_*^{\mathcal B}$ in matrix notation. 
\begin{proposition} For $a\in \mathcal B(E), A\in M_n(\mathcal B(E))$, 
  \begin{align*}
    	\Phi^{\mathcal B}_3 (a) &=\bigl(Ex E^{-t}-E^tE^{-1}\bigr)a, &
	\Phi^{\mathcal B}_2 (A) &= A + (E^{-1}x^t A E)^t,&
	\Phi^{\mathcal B}_1 (A) &= \tr(x^tA)-\tr(A).
  \end{align*}
\end{proposition}

\begin{proof}
  Recall that we identify $M_n(\mathcal B(E))\ni A\equiv\sum_{j,k=1}^n e_{jk}\otimes A_{jk}\in M_n(\C)\otimes \mathcal B(E)$. 
  The formula for $\Phi^{\mathcal B}_3$ is then obvious. For $\Phi^{\mathcal B}_1$, the matter is also quite clear,
  \[\Phi^{\mathcal B}_1(A)=\sum_{j,k=1}^n x_{jk}A_{jk}-\delta_{jk}A_{jk}=\tr(x^tA)-\tr(A). \]
  For $\Phi^{\mathcal B}_2$, we have
  \begin{align*}
  \MoveEqLeft
  \Phi^{\mathcal B}_2(A)
  = 
    \Phi^{\mathcal B}_2\big (\sum_{j,k=1}^n e_{jk}\otimes A_{jk}\big)=\sum_{j,k=1}^n e_{jk}\otimes A_{jk}
    +\sum_{j,k=1}^n \sum_{l,m=1}^n e_{lm} \otimes  (x E^{-t})_{jm} E_{kl}
    A_{jk}
\end{align*}
and, keeping track of the indices as well as the order of multiplication in $\mathcal B(E)$,
\begin{align*}
	(E^{-1}x^tAE)^t_{lm}&=(E^{-1}x^tAE)_{ml}= \sum_{j,k} (E^{-1}x^t)_{mj}A_{jk}E_{kl}= \sum_{j,k}(x E^{-t})_{jm} E_{kl}
    A_{jk}.\qedhere
\end{align*}
\end{proof}

We will use the resolution \eqref{resO+E} for $\mathcal B(E)$ to exhibit similar resolutions for $\mathcal A(E)$ and $\mathcal H(F)$ with $F=E^tE^{-1}$ in \Cref{sec:resolution-A(E),sec:resolution-H(F)}.

The following is a consequence of \cite[Cor.\ 6.3 and Prop.\ 6.4]{Bic13}, the main small improvement is that we express all cohomology spaces in terms of $F$ instead of $E$.

\begin{theorem}\label{thm:H(B(E))}
  Let $E\in \operatorname{GL}_n(\C)$, $\mathcal B:=\mathcal B(E)$, and $\tau\colon \mathcal B\to \mathbb C$ a character. Set $T:=\tau(x)$ and $F=E^tE^{-1}$.
  Let furthermore 
  $$p=\begin{cases}
    1 & \mbox{ if } T=I_n,\\
    0 & \mbox{ otherwise, }
  \end{cases} \quad
d=\operatorname{dim}\{\Kappa\in M_n:\Kappa^t=-T^t F\Kappa\}, \quad 
s=\begin{cases}
    0 & \mbox{ if } T^t=F^{-2},\\
    1 & \mbox{ otherwise. }
\end{cases}
$$
Then the Hochschild cohomology for $\mathcal B$ with 1-dimensional coefficients is given as follows:
\begin{align*}
    \dim \mathrm{H}^0(\mathcal B, {_\e\C_\tau}) &= p ,
    &\dim \mathrm{H}^1(\mathcal B, {_\e\C_\tau}) &= d-1+p,
    \\
    \dim \mathrm H^2(\mathcal B, {_\e\C_\tau}) &= d-s,
    &\dim \mathrm H^3(\mathcal B, {_\e\C_\tau}) &= 1-s,
\end{align*}
and $\dim \mathrm H^k(\mathcal B, {_\e\C_\tau}) = 0$ for all $ k \geq 4 $.
\end{theorem}

We present two proofs, the first one using Bichon's calculations of Hochschild homology and Poincar{\'e} duality, and the second one by directly applying the functor $\operatorname{hom}_{\mathcal B}(\cdot,\C_\tau)$ to the resolution \eqref{resO+E}. The concrete calculations in the second proof will provide useful references when we investigate $\mathcal A(E)$ and $\mathcal H(F)$ later.

\begin{proof}[First proof]
  By Poincar{\'e} duality \cite[Cor.\ 6.3]{Bic13}, we know that $\mathrm H^n(\mathcal B, {_\e\C_\tau})\cong H_{3-n}(\mathcal B,{_{\e\circ \alpha}\C_\tau})$, where $\alpha$ is the \emph{modular automorphism} determined by $\alpha(x)= E^{-1}E^{t}xE^{-1}E^t=F^{-t}xF^{-t}$. Note that $\varepsilon\circ\alpha (x)=(E^{-1}E^t)^2=F^{-2t}$ or, equivalently, $\varepsilon\circ\alpha=\Phi^{2}$ is the convolution square of the \emph{sovereign character} $\Phi$ determined by $\Phi(x)=E^{-1}E^t$ \cite[2.3 (1)]{Bic13}.  The Hochschild homology with one-dimensional coefficients ${_{\e\circ \alpha}\C_\tau}$ can be read off from \cite[Prop.\ 6.4]{Bic13}. It depends on the matrix (denoted by $\gamma(u)$ in  \cite{Bic13})
  \begin{align*}
  G:=\bigl(\tau^{-1}*(\e\circ \alpha)\bigr)(x)=T^{-1}(E^{-1}E^{t})^2=T^{-1}F^{-2t}.
  \end{align*}
  Finding the dimensions is mostly straightforward. The only point which deserves attention is $ \dim\mathrm{H}^1(\mathcal B, {_\e\C_\tau})=\dim \mathrm{H}_2(\mathcal B, {_{\e\circ\alpha}\C_\tau})=d'-(1-p)$ with
    \begin{align*}
      d'&=\dim\{\Mu\in M_n:\Mu+E^{t}\Mu^tGE^{-t}=0\}.
    \end{align*}
    The map $\Mu\mapsto E^t\Mu^t=:\Kappa$ is induces an isomorphism between the the spaces $\{\Mu\in M_n:\Mu+E^{t}\Mu^tGE^{-t}=0\}$ and $\{\Kappa\in M_n:\Kappa^t=-T^tF\Kappa \}$ because
    \begin{align*}
      \Mu+E^{t}\Mu^tGE^{-t}=0
      &\iff \Mu+E^{t}\Mu^tT^{-1}E^{-1}E^{t}E^{-1}=0\\
      &\iff \Kappa ^tE^{-1} + \Kappa T^{-1}F^{-t}E^{-1}=0\\
      &\iff \Kappa ^tF^{t}T+ \Kappa =0\\
      &\iff -T^tF\Kappa =\Kappa ^t.      
    \end{align*}
    Therefore, $d=d'$ and $\dim \mathrm{H}^1(\mathcal B, {_\e\C_\tau}) = d-1+p$ as claimed.
\end{proof}
\begin{proof}[Second proof]
    From the resolution \eqref{resO+E} for $\mathcal B$, we construct the cochain complex
\begin{multline}
\label{eq_resolution_of_B}
0\to \hom_{\mathcal B}\left(\mathcal B,\C_\tau\right) 
\xrightarrow{-\circ \Phi^{\mathcal B}_1} 
\hom_{\mathcal B} \left(M_n(\mathcal B),\C_\tau\right)
\\
\xrightarrow{-\circ \Phi^{\mathcal H}_2}
\hom_{\mathcal B}\left(M_n(\mathcal B),\C_\tau\right)
\xrightarrow{-\circ\Phi^{\mathcal H}_3}
\hom_{\mathcal B}\left(\mathcal B,\C_\tau\right)\to0.
\end{multline}

Recall that from Lemma \ref{lemma:traces-free-modules} we have vector space isomorphisms
\[\hom_{\mathcal B}(M_{n}(\mathcal B),\C_\tau)\cong M_{n}(\C),\quad \hom_{\mathcal B}(\mathcal B,\C_\tau)\cong \C\]
determined by the ``pairings''
    \[\Lambda^\tau A=\tau(\tr(\Lambda^t A)),\quad \lambda^{\tau} a=\lambda\tau(a)\]
    for $\Lambda\in M_{n}(\C),A\in M_{n}(\mathcal B),\lambda\in\C,a\in\mathcal B$.
    Also, the precompositions $-\circ\Phi^{\mathcal B}_j$ can be described by the corresponding maps $(\Phi^{\mathcal B}_j)^{*}$ which yield the isomorphic complex
    \[0\rightarrow \C \xrightarrow{(\Phi^{\mathcal B}_1)^{*}}M_n(\C) \xrightarrow{(\Phi^{\mathcal B}_2)^{*}}M_n(\C)\xrightarrow{(\Phi^{\mathcal B}_3)^{*}} \C \rightarrow 0,\]
    leading to
    \[\mathrm H^k\left( \mathcal B, {_\e\C_\tau}\right) \cong  \faktor{\ker((\Phi^{\mathcal B}_{k+1})^{*})}{\mathrm{Im}((\Phi^{\mathcal B}_k)^{*})},\]
    where $(\Phi^{\mathcal B}_{j})^{*}:=0$ for $j\notin\{1,2,3\}$.

    We will now have a closer look at the maps $(\Phi^{\mathcal B}_j)^{*}$ ($j=1,2,3$).
For $(\Phi^{\mathcal B}_1)^{*} \colon \C \to M_n(\C)$
we have
\begin{align*}
  (\lambda^{\tau}\circ\Phi_1^{\mathcal B})  (A) 
  & = \lambda \, \tau\big( \tr  (x^tA) -\tr(A)\big)
   = \tau(\tr (\lambda(T-I)^tA))= \bigl(\lambda(T-I)\bigr)^\tau A
  \end{align*}
i.e.,
\begin{align*}
  (\Phi^{\mathcal B}_1)^{*}  (\lambda)= \lambda (T-I).
\end{align*}
So
\begin{align*}
    \ker (\Phi^{\mathcal B}_1)^{*} &=\begin{cases}
    \C, & \mbox{if}\quad  T=I_n, \\
    \{0\}, & \mbox{otherwise}.
    \end{cases}
\end{align*}
Using the rank-nullity theorem, we conclude  that
\[\dim\ker (\Phi^{\mathcal B}_1)^{*} =p,
\qquad
\dim\im (\Phi^{\mathcal B}_1)^{*} =1-p.
  \]
\medskip
For \(
    (\Phi^{\mathcal B}_2)^{*}\colon M_n(\C)\to M_n(\C)
\) we find that
\begin{align*}
(\Lambda^\tau \circ \Phi^{\mathcal B}_2) 
  (A) 
  & = \Lambda^\tau ( A+ (E^{-1}x^tAE)^t)\\
  & =
    \tau  \tr (\Lambda^t A)+  \tau\tr(\Lambda^t (E^{-1}x^tAE)^t)\\
  &= \tau  \tr (\Lambda^t A)+ \tau  \tr (E\Lambda E^{-1}T^tA)\\
  &= (\Lambda + T E^{-t}\Lambda^t E^{t})^{\tau}(A),
\end{align*}
i.e.\
\begin{align*}
  (\Phi^{\mathcal B}_2)^{*} (\Lambda)= \Lambda + T E^{-t}\Lambda^t E^{t}.
\end{align*}
Now, $\Lambda\in  \ker (\Phi^{\mathcal B}_2)^{*}$ if and only if 
$$ \Lambda + T E^{-t}\Lambda^t E^{t}=0.$$
Note that from \eqref{eq_relations_in_BE}, it follows that $E^{-1}T^tET=I$ and, therefore, $E^{t}T^{-1}=T^tE^{t}$.
Using this relation, we see that the map $\Lambda\mapsto E\Lambda=:\Kappa$ induces an isomorphism between $\ker (\Phi^{\mathcal B}_2)^{*}$ and $\{\Kappa\in M_n:\Kappa^t=-T^tF\Kappa \}$ because
\begin{align*}
  \Lambda + T E^{-t}\Lambda^t E^{t}=0
  & \iff E^{-1}\Kappa  + T E^{-t} \Kappa ^{t}=0\\
  & \iff \Kappa ^{t}=-E^{t}T^{-1}E^{-1}\Kappa =-T^tE^{t}E^{-1}\Kappa =-T^tF\Kappa 
    .
\end{align*}
By the rank-nullity theorem,
\[\dim \ker (\Phi^{\mathcal B}_2)^{*} = d, \qquad \im (\Phi^{\mathcal H}_2)^{*} =  n^2-d.\]

Finally, for $(\Phi^{\mathcal B}_3)^{*} \colon M_n(\C) \to \C$ we find
\begin{align*}
    (\Lambda^{\tau} \circ \Phi^{\mathcal B}_3) (a) 
    	& = \Lambda^{\tau}\left(\bigl(Ex E^{-t}-E^tE^{-1}\bigr)a\right)
  \\ & = \tau\tr\left(\Lambda^t\bigl(ET E^{-t}-E^tE^{-1}\bigr)a\right)\\
        &=\tr\left(\Lambda^t\bigl(ET E^{-t}-E^tE^{-1}\bigr)\right)\tau(a),
   \end{align*}
so
\begin{align*}
    (\Phi^{\mathcal B}_3)^{*}
    (\Lambda) = \tr\left(\Lambda^t\bigl(ET E^{-t}-E^tE^{-1}\bigr)\right).
\end{align*}
Note that $ETE^{-t}=E^tE^{-1}$ if and only if $T=(E^{-1}E^t)^2$.
Using the rank-nullity theorem, we can thus conclude that
  \[\dim \ker  (\Phi^{\mathcal B}_3)^{*} =n^2-s,
\qquad 
\dim \im (\Phi^{\mathcal B}_3)^{*} =s.
\]
The result now follows from $\dim \mathrm H^k(\mathcal B, {_\e\C_\tau})=\dim\ker (\Phi^{\mathcal B}_{k+1})^* - \dim\im(\Phi^{\mathcal B}_k)^*$.
\end{proof}

\subsection{Projective resolution for free products of 
Hopf algebras}\label{freeprod}

The higher Hochschild cohomologies of the free product of Hopf algebras can be easily computed from the Hochschild cohomologies of the factors.
\begin{lemma}[{\cite[Theorem 5.2]{Bichon18}}] \label{lem_higher_Hn_of_free_product}
Let $A$ and $B$ be Hopf algebras. For any $A*B$-bimodule $M$ and for $n\geq 2$ we have 
$$\mathrm H^n(A*B,M)\cong \mathrm H^n(A,M)\oplus \mathrm H^n(B,M),$$
where the bimodules on the right-hand side are equipped with the restricted actions on $A$ and $B$.
\end{lemma}

In our considerations we will need the following more detailed result. 
\begin{theorem}\label{thm:res_for_free_product}
Let $A=B*C$ be a free product of Hopf algebras $B$ and $C$. Furthermore, suppose that 
\begin{align*}
    \cdots\xrightarrow{\Phi_3}P_2\xrightarrow{\Phi_2} P_1\xrightarrow{\Phi_1} B \xrightarrow{\varepsilon_B}\mathbb C_{\varepsilon_B}\rightarrow 0\\
    \cdots\xrightarrow{\Psi_3}Q_2\xrightarrow{\Psi_2} Q_1\xrightarrow{\Psi_1} C \xrightarrow{\varepsilon_B}\mathbb C_{\varepsilon_B}\rightarrow 0
\end{align*}
are projective resolutions for $B$ and $C$, respectively.  
Put $\widetilde P_k:=P_k\otimes_B A$, $\widetilde Q_k:= Q_k\otimes_C A$, $\widetilde \Phi_k:=\Phi_k\otimes_B{\mathrm{id}_A}$, $\widetilde\Psi_k:=\Psi_k\otimes_{C}\mathrm{id}_A$. Then 
\[\cdots\xrightarrow{\widetilde \Phi_3\oplus \widetilde \Psi_3}\widetilde P_2\oplus \widetilde Q_2\xrightarrow{\widetilde \Phi_2\oplus \widetilde \Psi_2} \widetilde P_1\oplus \widetilde Q_1\xrightarrow{\widetilde \Phi_1 + \widetilde \Psi_1} A \xrightarrow{\varepsilon_A}\mathbb C_{\varepsilon_A}\rightarrow 0\]
is a projective resolution for $A$. 
\end{theorem}
\begin{proof}
For the codomains of $\widetilde \Phi_1$ and $\widetilde \Psi_1$ we have $B\otimes_B A=A= C\otimes_C A$. 
It is easy to see that the modules are projective (they can be complemented to free modules using the fact that the modules in the original resolutions are projective). 
Note that the free product $A=B\free C$, as a module over the algebras $B$ and $C$, respectively, is free, hence flat, which means in our context that the functors $\cdot\otimes_B A$ and $\cdot \otimes_C A$ are exact. 
Therefore, 
\[\operatorname{ker}(\widetilde \Phi_k\oplus \widetilde{\Psi}_k)
=\operatorname{ker}(\widetilde \Phi_k)\oplus \operatorname{ker}(\widetilde{\Psi}_k)
=\operatorname{im}(\widetilde \Phi_{k+1})\oplus \operatorname{im}(\widetilde{\Psi}_{k+1})
=\operatorname{im}(\widetilde \Phi_{k+1}\oplus \widetilde{\Psi}_{k+1})\]
proves exactness at $\widetilde P_k\oplus \widetilde Q_k$ with $k>1$. Under the natural identifications, we have (using the shorthand notation $A^+:=\operatorname{ker}\varepsilon_A$ and analogously for $B$ and $C$)
\[A^+= B^+ A \oplus C^+ A= (B^+\otimes_B A) \oplus (C^+ \otimes_C A) = \operatorname{im}(\widetilde \Phi_1)\oplus \operatorname{im}(\widetilde{\Psi}_1)=\operatorname{im}(\widetilde \Phi_1\oplus \widetilde{\Psi}_1).\]
This directly shows exactness at $A$, but also allows us to conclude 
\[\operatorname{ker}(\widetilde \Phi_1+\widetilde{\Psi}_1)=\operatorname{ker}(\widetilde \Phi_1\oplus\widetilde{\Psi}_1)= \operatorname{im}(\widetilde \Phi_2\oplus\widetilde{\Psi}_2).\qedhere\]
\end{proof}

This result allows to compute the Hochschild cohomologies of the free product of Hopf algebras by passing to the hom-spaces as described in Theorem \ref{thm:hom-spaces}. 

\subsection{Yetter-Drinfeld modules and Gerstenhaber-Schack 
cohomology (generalities)}\label{subsec:YD-generalities}

A (right-right) {\em Yetter--Drinfeld module} over a Hopf algebra $H$ is a right $H$-comodule $X$ which is also a right $H$-module satisfying
\[
(xh)_{(0)}\otimes (xh)_{(1)}=x_{(0)}h_{(2)}\otimes S(h_{(1)})x_{(1)}h_{(3)},\qquad x\in X,\quad h\in H.
\]
In other words, a Yetter-Drinfeld module is an $H$-module $H$-comodule such that the coaction is a module map for the $H$-action \begin{align}
    (x\otimes h)\YDaction h':=xh'_{(2)}\otimes S(h'_{(1)})h h'_{(3)}
\end{align}
on $X\otimes H$. 
Morphisms between Yetter-Drinfeld modules over $H$ are the $H$-linear $H$-colinear maps. We denote the category thus obtained by $\mathcal{YD}^H_H$. The coaction is often denoted by $\gamma(x):=x_{(0)}\otimes x_{(1)}$.

Before we continue to discuss the category of Yetter-Drinfeld modules, we record for future reference a simple lemma, which is a slight generalization of the forward implication of \cite[Proposition 4.5]{Bic16}.

\begin{lemma}\label{lem:coinvariant}
  Let $H$ be a Hopf subalgebra of $A$ with
  \begin{align}\label{eq:assumption_lem_coinvariant}
a_{(2)}\otimes S(a_{(1)})Ha_{(3)}\subset A\otimes H
  \end{align}
  for all $a\in A$, i.e.\ $A\otimes H$ is invariant for the $\YDaction$-action of $A$ on $A\otimes A$ (or in the terminology of \cite[Definition 4.6]{Bic16}: $H\subset A$ is \emph{adjoint}).
  Assume further that $X$ is a Yetter-Drinfeld $A$-module with a subcomodule $C\subset X$ which generates $X$ as an $A$-module, i.e.\ $X=CA$. Denote $\gamma\colon X\to X\otimes A$ the $A$-coaction on $X$. Then $\gamma(X)\subset X\otimes H$ if and only if  $\gamma(C)\subset C\otimes H$.
\end{lemma}

\begin{proof}
  The ``only if'' part is trivial, the ``if'' part easily follows from
  $\gamma(ca)=\gamma(c)\YDaction a=c_{(0)}a_{(2)}\otimes S(a_{(1)})c_{(1)}a_{(3)}$.
\end{proof}

For the purposes of cohomology computations, we need to identify the projective objects in the category $\mathcal{YD}^H_H$. To this end, recall that given a right $H$-comodule $C$ one can construct a Yetter-Drinfeld module $C\boxtimes H$, referred to as \emph{the free Yetter-Drinfeld module over $C$}, as follows. As a right module it is just $C\otimes H$ (the free $H$-module over $C$) and the right coaction is given by
\[
\delta_\boxtimes(c\boxtimes h):=c_{(0)}\boxtimes h_{(2)}\otimes S(h_{(1)})c_{(1)}h_{(3)},
\]
the unique extension of the coaction on $C=C\boxtimes 1\subset C\boxtimes H$ to a module map. 
It is easy to check that the free Yetter-Drinfeld module over an $H$-comodule $C$ enjoys the following universal property: for every comodule map $f\colon C\to X$ to a Yetter-Drinfeld module $X\in\mathcal{YD}_H^H$, there exists a unique extension to a morphism of Yetter-Drinfeld modules $\widetilde f\colon C\boxtimes H\to X$, namely $\widetilde f(c\boxtimes h):=f(c)h$. 

A Yetter-Drinfeld $H$-module is called {\em free} if it is isomorphic to the free Yetter-Drinfeld module $C\boxtimes H$ over some $H$-comodule $C$. Note that a free Yetter-Drinfeld module is in particular a free $H$-module.

To simplify the notation, if $H$ is a Hopf algebra, we denote the free Yetter-Drinfeld module over $\mathbb C$, $\mathbb C\boxtimes H$, simply by $H$; it is often called \emph{right coadjoint Yetter-Drinfeld module} and denoted by $H_{\mathrm{coad}}$. 
Subsequently, we call a Yetter-Drinfeld module {\em relatively projective} if it is a direct summand of a free Yetter--Drinfeld module. One can prove that, for a cosemisimple Hopf algebra $H$, the projective objects in $\mathcal{YD}^H_H$ are exactly the relatively projective Yetter-Drinfeld modules~\cite[Proposition~4.2]{Bic16}.

Note that $\mathbb C_\varepsilon$ is a (typically not projective) Yetter-Drinfeld module with coaction $1\mapsto 1\otimes 1\in \mathbb C\otimes H$, indeed, 
\[1\otimes 1\YDaction h = 1\varepsilon(h_{(2)})\otimes S(h_{(1)})1h_{(3)}=\varepsilon(h)1\otimes 1.\] 
Therefore, it makes sense to search for projective resolutions of the counit by projective Yetter-Drinfeld modules.

Let us now discuss the \emph{Gerstenhaber-Schack cohomology}~\cite{G-Sch90}. Let $H$ be a Hopf algebra and let $X$ be a Yetter-Drinfeld module over $H$. The Gerstenhaber-Schack cohomology of $H$ with coefficients in $X$ is denoted by $\mathrm{H_{GS}^*}(H,X)$. In this paper, we use the following result of Bichon~\cite[Proposition~5.2]{Bic16} as our definition of GS cohomology (we state it for cosemisimple Hopf algebras but it works for more general co-Frobenius Hopf algebras).
\begin{proposition}
Let $H$ be a cosemisimple Hopf algebra and let
\[
P_*:=\ldots\to P_{n+1}\to P_n\to\ldots\to P_1\to \C_\varepsilon\to 0
\]
be a projective resolution of the counit in $\mathcal{YD}^H_H$. For any Yetter-Drinfeld $H$-module $X$, there is an isomorphism
\[
\mathrm{H^*_{GS}}(H,X)\cong \mathrm H^*\left(\operatorname{hom}_{\mathcal{YD}^H_H}(P_*,X)\right).
\]
\end{proposition}
\noindent We refer the reader to~\cite[Section~1]{G-Sch90} for a definition of $\mathrm{H_{GS}^*}(H,X)$ via an explicit complex. The {\em bialgebra cohomology} is defined as $\mathrm{H_b^*}(H):=\mathrm{H_{GS}^*}(H,\C_\e)$.

Let us now discuss the restriction and the induction functors. For the remainder of the section, let $H$ be a Hopf subalgebra of a Hopf algebra $A$. The {\em restriction functor} is defined as follows
\[
\operatorname{_{\mathnormal H}\mathrm{res}_\mathnormal{A}}\colon\mathcal{YD}^A_A\longrightarrow \mathcal{YD}^H_H,\qquad X\longmapsto X^{(H)},
\]
where $X^{(H)}:=\{x\in X~:~x_{(0)}\otimes x_{(1)}\in X\otimes H\}$. One can show that the above functor is isomorphic to the functor $(-)\mathbin{\Box}_A H$ (recall that $\mathbin{\Box}_A$ denotes the cotensor product over $A$); the isomorphism is given by $X^{(H)}\ni x\mapsto x_{(0)}\otimes x_{(1)}\in X\mathbin{\Box}_A H$. Hence, the restriction functor is always left exact and we say that $H$ is {\em coflat} over $A$ when it is exact. Next, we define the {\em induction functor} to be
\[
\operatorname{_{\mathnormal A}\mathrm{ind}_\mathnormal{H}}\colon\mathcal{YD}^H_H\longrightarrow \mathcal{YD}^A_A,\qquad X\longmapsto X\otimes_H A.
\]
The Yetter-Drinfeld $A$-structure on $X\otimes_H A$ is given as follows: the $A$-module action is just the multiplication on the right and the $A$-comodule structure is defined by
\[
x\otimes_A h\longmapsto x_{(0)}\otimes_A h_{(2)}\otimes S(h_{(1)})x_{(1)}h_{(3)}.
\]
It is clear that the above functor is always right exact and we say that $A$ is flat over $H$ if it is exact. 

It is worth noting that the functors $(\operatorname{_{\mathnormal H}\mathrm{res}_\mathnormal{A}},
\operatorname{_{\mathnormal A}\mathrm{ind}_\mathnormal{H}})$ 
form a pair of adjoint functors (this is a special case of \cite[Section~2.5 Theorem~15]{CMZ02}; in Section 5.3 Proposition 123 therein, the case of left-right Yetter-Drinfeld modules is discussed, the theorem applies to right-right Yetter-Drinfeld modules in an analogous way).

If $A$ is cosemisimple, then so is $H$ and
$A$ is flat over $H$~\cite{Chirvasitu14}, so the induction functor is exact. Hence, in the cosemisimple case, a general categorical result \cite[Problem 4.16 and its dual statement]{Pareigis1970} implies that $\operatorname{_{\mathnormal A}\mathrm{ind}_\mathnormal{H}}$ preserves projective objects. This last observation is crucial in the next theorem (although in the special cases we treat later, projectivity of the involved modules will be easy to check by hand).

\begin{theorem}[Projective Yetter-Drinfeld resolution for free products]\label{thm:YD-res_free-product}
Let $A=B*C$ be a free product of Hopf algebras $B$ and $C$. Furthermore, suppose that 
\begin{align*}
    \cdots\xrightarrow{\Phi_3}X_2\xrightarrow{\Phi_2} X_1\xrightarrow{\Phi_1} B \xrightarrow{\varepsilon_B}\mathbb C_{\varepsilon_B}\rightarrow 0\\
    \cdots\xrightarrow{\Psi_3}Y_2\xrightarrow{\Psi_2} Y_1\xrightarrow{\Psi_1} C \xrightarrow{\varepsilon_C}\mathbb C_{\varepsilon_C}\rightarrow 0
\end{align*}
are resolutions of $\mathbb{C}_{\varepsilon_B}$ and $\mathbb{C}_{\varepsilon_C}$ by relatively projective Yetter--Drinfeld modules, respectively. Put $\widetilde X_k:=X_k\otimes_B A$, $\widetilde Y_k:= Y_k\otimes_C A$, $\widetilde \Phi_k:=\Phi_k\otimes_B{\mathrm{id}_A}$, and $\widetilde\Psi_k:=\Psi_k\otimes_{C}\mathrm{id}_A$. Then, for the codomains of $\widetilde \Phi_1$ and $\widetilde \Psi_1$, we have $B\otimes_B A=A= C\otimes_C A$ (as Yetter-Drinfeld modules!) and
\[\cdots\xrightarrow{\widetilde \Phi_3\oplus \widetilde \Psi_3}\widetilde X_2\oplus \widetilde Y_2\xrightarrow{\widetilde \Phi_2\oplus \widetilde \Psi_2} \widetilde X_1\oplus \widetilde Y_1\xrightarrow{\widetilde \Phi_1 + \widetilde \Psi_1} A \xrightarrow{\varepsilon_A}\mathbb C_{\varepsilon_A}\rightarrow 0\]
is a resolution for $A$ by projective objects in $\mathcal{YD}^A_A$. 
\end{theorem}

\begin{proof}
  It is straightforward that the maps $\widetilde{\Phi}_k$ and $\widetilde{\Psi}_k$ are $A$-comodule maps. As in the proof of \Cref{thm:res_for_free_product}, since $A=B\free C$ is flat over $B$ and $C$, respectively, 
  the induction functors $\cdot\otimes_B A$ and $\cdot \otimes_C A$ are exact. Exactness holds in both, the category of right modules and the category of Yetter-Drinfeld modules, since the induction functors act identically. Therefore, 
  \[\operatorname{ker}(\widetilde \Phi_k\oplus \widetilde{\Psi}_k)
    =\operatorname{ker}(\widetilde \Phi_k)\oplus \operatorname{ker}(\widetilde{\Psi}_k)
    =\operatorname{im}(\widetilde \Phi_{k+1})\oplus \operatorname{im}(\widetilde{\Psi}_{k+1})
    =\operatorname{im}(\widetilde \Phi_{k+1}\oplus \widetilde{\Psi}_{k+1})\]
  proves exactness at $\widetilde P_k\oplus \widetilde Q_k$ with $k>1$. Under the natural identifications, we have for $A^+:=\operatorname{ker}\varepsilon_A$
  \[A^+= B^+ A \oplus C^+ A= (B^+\otimes_B A) \oplus (C^+ \otimes_C A) = \operatorname{im}(\widetilde \Phi_1)\oplus \operatorname{im}(\widetilde{\Psi}_1)=\operatorname{im}(\widetilde \Phi_1\oplus \widetilde{\Psi}_1).\]
  Finally, the induction functors preserve projective objects, so we are done.
\end{proof}

Let us close this section by investigating what the induction functor does to free Yetter-Drinfeld modules.

\begin{proposition}\label{prop:ind-free-YD}
  Let $H\subset A$ be a Hopf subalgebra and $C\boxtimes H$ be a free Yetter-Drinfeld module. Then
  \[(C\boxtimes H)\otimes_H A=C\boxtimes A,\]
  i.e.\ the canonical module isomorphism is a Yetter-Drinfeld isomorphism. 
\end{proposition}

\begin{proof}
  The map $C\boxtimes A\supset C\boxtimes 1=C\to C\otimes H\subset (C\boxtimes H)\otimes_H A, c\mapsto c\otimes 1$ is obviously a comodule map. Its unique extension to $C\boxtimes A$ as a module map is the canonical module isomorphism $c\boxtimes a\mapsto c\boxtimes 1\otimes_Ha$. By the universal property of the free Yetter-Drinfeld module, it is a Yetter-Drinfeld map.
\end{proof}

\section{Projective resolution and Hochschild cohomology for \texorpdfstring{$\mathcal A(E)$}{A(E)} }
\label{sec:resolution-A(E)}

In this section we use the results from Section~\ref{freeprod} on projective 
resolutions of free products to describe a projective resolution for 
$\mathcal A(E)=\mathcal B(E)*\C\Z_2$, and compute the Hochschild cohomology with one-dimensional coefficients $\mathrm{H}^*(\mathcal A(E),{_\e\C_\tau})$. We denote by $g$ the unitary generator of 
$\mathbb{C}\mathbb{Z}_2$ ($g^2=1$), and by $x_{jk}$, $j,k=1,\ldots, n$ the matrix elements generating $\mathcal B(E)$ as an algebra. 

Since $\Z_2$ is a finite group, it is easy to see that $\C_\varepsilon$ is a projective module (indeed, $\C\Z_2=\C(1+g)\oplus \C(1-g))$ and $\C(1+g)\cong \C_\varepsilon$). Hence, there is a projective resolution of length 0, \[0\rightarrow \C(1+g) \rightarrow \C_\varepsilon \rightarrow 0.\]
However, in order to apply \Cref{thm:res_for_free_product,thm:YD-res_free-product}, we need to find a projective resolution with $P_0=\C\Z_2$. The free resolution given in, e.g., \cite[p.~167]{Wei94} would work, but is infinite. We prefer to use the finite projective resolution from the following lemma.
\begin{lemma}
The sequence
 \begin{align}\label{eq:resZ2}
0\to\mathbb{C}(1-g)\xrightarrow{\Psi_1}\mathbb{C}\mathbb{Z}_2\xrightarrow{ 
\varepsilon}\mathbb{C}_\varepsilon\to 0,
\end{align}
where
\[
\varepsilon(1)=\varepsilon(g)=1,\qquad \Psi_1(1-g)=1-g,
\]
is a projective resolution of the counit of $\mathbb{C}\mathbb{Z}_2$.
\end{lemma}

\begin{proof}
Indeed, $\mathbb{C}(1-g)$ is a projective $\mathbb{C}\mathbb{Z}_2$-module since 
it can be complemented to a free $\mathbb{C}\mathbb{Z}_2$-module. Injectivity of 
$\Psi_1$  is clear. Since $\varepsilon(\Psi_1(1-g))=0$, we have that $\operatorname{
Im}(\Psi_1)\subseteq\ker\varepsilon$ and it is straightforward to prove that 
this containment is in fact an equality.
\end{proof}

From this we get the Hochschild cohomology of $\C\Z_2$ with 1-dimensional coefficients.

\begin{proposition} \label{prop:HH-of-CZ2}
  There are exactly two characters on $\C\Z_2$, namely $\e$ and the sign representation $\sigma$, determined by $\sigma (g)=-1$. The Hochschild cohomology of $\C\Z_2$ with 1-dimensional coefficients is given by
  \begin{align*}
    \dim \mathrm{H}^0(\C\Z_2 , {_\e\C_\e}) &= 1 ,& \dim \mathrm{H}^0(\C\Z_2 , {_\e\C_\sigma}) &=0 \\
    \dim \mathrm H^k(\C\Z_2 , {_\e\C_\tau}) &= 0, \text{ for any } k \geq 1, \, \tau\in\{\e,\sigma\}.&&
\end{align*}
\end{proposition}

\begin{proof}
The first statement is obvious.
Let $\tau\in\{\e,\sigma\}$. The resolution \eqref{eq:resZ2} yields, by passing to hom-spaces, the complex 
  \[0\rightarrow \hom_{\C\Z_2}(\C\Z_2,\C_\tau)\xrightarrow{-\circ \Psi_1} 
    \hom_{\C\Z_2}(\C(1-g),\C_\tau) \rightarrow 0,\]
  whose cohomology coincides with the Hochschild cohomology spaces we aim to determine. Clearly, $\dim(\hom_{\C\Z_2}(\C\Z_2,\C_\tau))=1$ because $\C\Z_2$ is a free module with basis $1$. On the other hand, we find that \[\dim (\hom_{\C\Z_2}(\C(1-g),\C_\tau))=
    \begin{cases}
      0&\tau=\e,\\
      1&\tau=\sigma;
    \end{cases}
  \]
  indeed, the linear map $f_z$ determined by $(1-g)\mapsto z$ is $\C\Z_2$-equivariant if and only if \[z=f(1-g)=f(-(1-g)g)=-f(1-g)\tau(g)=-\tau(g)z,\]
  which is equivalent to $z=0$ or $\tau(g)=-1$, i.e.\ $\tau=\sigma$. For $\tau=\e$, this completes the proof. For $\tau=\sigma$, we also need to check injectivity of $-\circ\Psi_1$. Suppose that $f\in \hom_{\C\Z_2}(\C\Z_2,\C_\tau)$. Then, necessarily,  $f(1+g)=f(1)+f(1)\sigma(g)=0$. If $f\circ\Psi_1=0$, then also $f(1-g)=0$ and, therefore $f=0$, which proves injectivity.
\end{proof}

We apply Theorem~\ref{thm:res_for_free_product} to the resolutions \eqref{eq:resZ2} and \eqref{resO+E} and obtain the following projective resolution of the counit for $\mathcal A(E)$.

\begin{theorem} Let $E\in\operatorname{GL}_n(\C)$ and $\mathcal A=\mathcal A(E)$. Then the sequence
  \begin{align} \label{eq_resolution_of_A}
    0\to \mathcal A \xrightarrow{\Phi^{\mathcal A}_3}M_n(\mathcal A )\xrightarrow{\Phi^{\mathcal A}_2} M_n(\mathcal A )\oplus(\mathbb{C}(1-g)\otimes_{\mathbb{C}\mathbb{Z}_2}\mathcal A )
    \xrightarrow{\Phi^{\mathcal A}_1}\mathcal A \xrightarrow{\varepsilon}\mathbb{C}_\varepsilon\to0,
  \end{align} 
  with the mappings
  \begin{align*}
    	\Phi^{\mathcal A}_3 (a) &=\bigl(Ex E^{-t}-E^tE^{-1}\bigr)a, \\
	\Phi^{\mathcal A}_2 (A) &= (A + (E^{-1}x^t A E)^t, 0),\\
	\Phi^{\mathcal A}_1 (A,(1-g)\otimes_{\C\Z_2}a) &= \tr(x^tA)-\tr(A)+(1-g)a
  \end{align*} 
  yields a projective resolution of the counit of $\mathcal A$.
\end{theorem}

We now aim to compute the Hochschild cohomology of $\mathcal A(E)$ with the coefficients in $_\e\C_\tau$, where $\tau$ is a character on the free product ${\mathcal A}(E)=\mathcal B(E)\free \C\Z_2$. By \Cref{thm:hom-spaces}, this is the cohomology of the complex obtained by applying the functor $\hom_{\A}(-,\C_\tau)$ to the resolution \eqref{eq_resolution_of_A}.

\begin{theorem}
    Let $E\in \operatorname{GL}_n(\C)$, $\mathcal A:=\mathcal A(E)$, and $\tau\colon \mathcal A\to \mathbb C$ a character. Also assume that $F:=E^tE^{-1}$ is generic, so that $\mathcal A(E)$ is cosemisimple. Set $T:=\tau(x)$, $t=\tau(g)$.
  Furthermore, let 
  \begin{align*}
  p&=\begin{cases}
    1 & \mbox{ if } T=I_n\\
    0 & \mbox{ otherwise, }
  \end{cases}&
  q&=\begin{cases}
    1 & \mbox{ if } t=1\\
    0 & \mbox{ otherwise, }
  \end{cases}\\ 
    d&=\operatorname{dim}\{\Kappa\in M_n:\Kappa^t=-T^tF\Kappa \},&
                                                            s&=\begin{cases}
    0 & \mbox{ if } T^t=F^{-2}\\
    1 & \mbox{ otherwise. }
  \end{cases}
  \end{align*}
Then the Hochschild cohomology for $\mathcal A$ with 1-dimensional coefficients is given as follows:
\begin{align*}
    \dim \mathrm{H}^0(\mathcal A, {_\e\C_\tau}) &= pq ,
    &\dim \mathrm{H}^1(\mathcal A, {_\e\C_\tau}) &= d -(1-p)q,
    \\
    \dim \mathrm H^2(\mathcal A, {_\e\C_\tau}) &= d-s,
    &\dim \mathrm H^3(\mathcal A, {_\e\C_\tau}) &= 1-s,
\end{align*}
and $\dim \mathrm H^k(\mathcal A, {_\e\C_\tau}) = 0$ for all $ k \geq 4 $.
\end{theorem}
\begin{proof}
  We extend the conventions from \Cref{lemma:traces-free-modules} to the projective module $\mathbb{C}(1-g)\otimes_{\mathbb{C}\mathbb{Z}_2}\mathcal A$ by defining
  \[\lambda^\tau\bigl((1-g)\otimes_{\C\Z_2}a\bigr):=\lambda\tau((1-g)a)=\lambda\cdot (1-t)\tau(a);\]
  note that the map $\lambda\mapsto \lambda^\tau\colon \C\to \hom_{\mathcal A}(\mathbb{C}(1-g)\otimes_{\mathbb{C}\mathbb{Z}_2}\mathcal A,\C_\tau)$ is an isomorphism if $t=-1$ and the zero map if $t=1$.
 Applying the functor $\hom_{\A}(-,\C_\tau)$ and the discussed isomorphisms to the resolution \eqref{eq_resolution_of_A},  Hochschild cohomology coincides with cohomology of the complex
  \[
    0 \rightarrow \C \xrightarrow{(\Phi_1^{\mathcal A})^*} 
    M_n(\C) \oplus \delta_{\{-t\}}\C 
    \xrightarrow{(\Phi_2^{\mathcal A})^*} 
    M_n(\C) 
    \xrightarrow{(\Phi_3^{\mathcal A})^*} 
    \C \to 0,
  \]
  where $\delta_{\{-t\}}\C $ means that this summand is omitted if $t=\tau(g)=1$. 
  
  Let us have a closer look at the mappings $(\Phi_j^{\mathcal A})^*$. For $j=1$, we find
  \begin{align*}
    (\lambda^{\tau}\circ\Phi_1^{\mathcal A})(A,(1-g)\otimes_{\C\Z_2} a)
    &= \lambda\tau \big(\tr(x^tA)-\tr(A)\big) + (1-t) \cdot \lambda \tau(a)
    \\ & 
         =\lambda\tau \tr \big((T^t-I)A\big)  +\lambda\tau\bigl( (1-t) \cdot \tau(a)\bigr)
    \\ &
         = \lambda\bigl((T-I) \oplus \delta_{\{-t\}} \bigr)^{\tau} (A,(1-g)\otimes_{\C\Z_2} a),
  \end{align*}
    so \[(\Phi_1^{\mathcal A})^*(\lambda)= \lambda\bigl((T-I)\oplus\delta_{\{-t\}}\cdot 1\bigr)\in M_n(\C)\oplus \delta_{\{-t\}} \C.\] 
  For $j>1$, we note that $(\Phi_2^{\mathcal A})^{*}=(\Phi_2^{\mathcal B})^{*} + \delta_{\{-t\}}\cdot 0$ and  $(\Phi_3^{\mathcal A})^{*}=(\Phi_3^{\mathcal B})^{*}$ for the maps $(\Phi_3^{\mathcal B})^{*}$ defined in the second proof of \Cref{thm:H(B(E))}. So we find that
  \begin{align*}
    \dim\ker (\Phi_1^{\mathcal A})^*&=pq,& \dim\im (\Phi_1^{\mathcal A})^*&=1-pq\\
    \dim \ker (\Phi_2^{\mathcal A})^{*}&=\dim \ker (\Phi_2^{\mathcal B})^{*}+1-q=d+1-q
    &
      \dim \im (\Phi_2^{\mathcal A})^{*}&=\dim \im (\Phi_2^{\mathcal B})^{*}=n^2-d\\
    \dim \ker (\Phi_3^{\mathcal A})^{*}&=\dim \ker (\Phi_3^{\mathcal B})^{*} = n^2-s
    &
      \dim \im (\Phi_3^{\mathcal A})^{*}&=\dim \im (\Phi_3^{\mathcal B})^{*}=s
  \end{align*}
  and taking differences completes the proof. Note that $ \dim \mathrm H^k(\mathcal{A} , {_\e\C_\tau})$ for $k\geq2$ can alternatively be deduced from \Cref{thm:H(B(E)),prop:HH-of-CZ2,lem_higher_Hn_of_free_product}.
\end{proof}

\section{Matrix Hopf algebras and glued products}
\label{sec_glued}

The main aim of this section is to identify $\mathcal H(F)$ with a certain subalgebra of $\mathcal A(E)$, which in the case $E=F=I$ will specialize to
\[\Pol{U_n^+}\cong \Pol{O_n^+\gluedfree \Z_2} \subset  \Pol{O_n^+\free\Z_2}\cong\Pol{B_{n+1}^{\#+}}\]
 (with $\gluedfree$ the glued free product of compact matrix quantum groups). To this end, we extend the notion of glued free products, defined by Tarrago and Weber \cite{TerragoWeber2017} for compact matrix quantum groups, to Hopf algebras. As for quantum groups, the glued free product refers to an additional matrix structure. 
The following definition of matrix Hopf algebra due to \v{S}koda gives the right setting.  \v{S}koda's definition makes use of the concept of Hopf envelope, which we do not explain here, but we will show below that matrix Hopf algebras can equivalently be characterized without referring to Hopf envelopes.

\begin{definition}[{cf.\ \cite[Section 13]{Skoda2003}}]
  A \emph{matrix bialgebra} is a pair $(B,u)$ where $B$ is a bialgebra and $u=(u_{ij})_{i,j=1}^n\in M_n(B)$ is a finite-dimensional corepresentation such that $B$ is generated by the $u_{ij}$ as a unital algebra. A \emph{matrix Hopf algebra} is a pair $(A,u)$ where $A$ is a Hopf algebra and $u$ is a finite-dimensional corepresentation of $A$ such that the canonical map $\operatorname{Hopf}(B)\to A$ is onto, where $B$ is the unital algebra generated by the $u_{ij}$, and $\operatorname{Hopf}$ denotes functor which maps a bialgebra to its Hopf envelope; note that $(B,u)$ is automatically a matrix bialgebra.

  If $(A,u)$ is a matrix Hopf algebra, we also say that $A$ is matrix Hopf algebra with \emph{fundamental corepresentation}\footnote{In \cite{Skoda2003}, \v{S}koda calls $u$ a \emph{basis} instead, but we believe that this terminology is somewhat misleading.} $u$. 
\end{definition}

\begin{proposition}
  Let $A$ be a Hopf algebra with antipode $S$, and $u\in M_n(A)$ a finite-dimensional corepresentation. Then the following are equivalent.
  \begin{enumerate}
  \item $(A,u)$ is a matrix Hopf algebra,
  \item the $u_{ij}$, $i,j\in\{1,\ldots,n\}$, generate $A$ as a Hopf algebra, i.e.\ a Hopf subalgebra $A'\subseteq A$ which contains all \(u_{ij}\) must coincide with $A$,  
  \item the $S^k(u_{ij})$, $k\in\mathbb N_0, i,j\in\{1,\ldots,n\}$, generate $A$ as an algebra. 
  \end{enumerate}
\end{proposition}

\begin{proof}
  Denote $B$ the subalgebra of $A$ generated by the $u_{ij}$ and $\operatorname{Hopf}(B\hookrightarrow A)\colon \operatorname{Hopf}(B)\to A$ the canonical map. Suppose that $\operatorname{Hopf}(B\hookrightarrow A)$ is onto and $B\subset A'\subset A$ for a Hopf subalgebra $A$. With $\iota_{A'}^A\colon A'\hookrightarrow A$ the embedding map, $\operatorname{Hopf}(B\hookrightarrow A)=\iota_{A'}^A\circ\operatorname{Hopf}(B\hookrightarrow A')$ by the universal property of the Hopf envelope. It follows that $\iota_{A'}^A\colon A'\hookrightarrow A$ is onto, i.e.\ $A=A'$, proving that (1) implies (2). Conversely, if  $\operatorname{Hopf}(B\hookrightarrow A)$ is not onto, then its image is a Hopf subalgebra $A'\subset A$, $A'\neq A$, which contains all $u_{ij}$, i.e.\ $A$ is not generated by the $u_{ij}$ as a Hopf algebra.

  Equivalence with the third statement follows from the simple observation that the algebra generated by all $S^k(u_{ij})$ is a Hopf algebra, hence it is the Hopf algebra generated by the $u_{ij}$. 
\end{proof}

\begin{example}\ 
  \begin{itemize}
  \item If $G$ is a finitely generated group with generators $g_1,\ldots g_n$, then $(\C G,\operatorname{diag}(g_1,\ldots,g_n))$ is a matrix Hopf algebra.
  \item If $(\mathcal C(\mathbb G),u)$ is a compact matrix quantum group, then  $(\Pol{\mathbb G},u)$ is a matrix Hopf algebra. This follows easily from the fact that $S(u_{ij})=u_{ji}^*$.
  \item $(\mathcal B(E),x)$ is a matrix Hopf algebra. Note that the linear span of the $x_{ij}$ is invariant under $S$, $S(x)=E^{-1}x^tE$, and $\mathcal B(E)$ is (by definition) generated by the $x_{ij}$ as an algebra. In particular, $(\mathcal B(E),x)$ is also a matrix bialgebra.
  \item $(\mathcal H(F),u)$ is a matrix Hopf algebra. Indeed, $S(u_{ij})=v_{ji}$, so $\mathcal H(F)$ is generated as an algebra by the collection of $u_{ij}$ and $S(u_{ij})$.  
  \end{itemize}
\end{example}

Before we come to the definition of glued free product of matrix Hopf algebras, we show that the free product of matrix Hopf algebras is again a matrix Hopf algebra.

\begin{proposition}
  Let $A$ and $B$ be matrix Hopf algebras with fundamental corepresentations $u_A=(u^A_{ij})_{i,j=1}^n$ and $u_B=(u^B_{kl})_{k,l=1}^m$, respectively. Then $A*B$ is a matrix Hopf algebra with fundamental corepresentation $u\oplus v\in M_{n+m}(A*B)$. In particular, $\mathcal A(E)=\mathcal B(E)\free\C\Z_2$ is a matrix Hopf algebra.
\end{proposition}

\begin{proof}
  We have to show that $\{u^A_{ij},u^B_{kl}:i,j=1,\ldots n, k,l=1,\ldots m\}\subset H\subset A*B$ for a Hopf subalgebra $H$ implies $H=A*B$. This is obvious: if $H$ contains all $u^A_{ij}$, then it contains $A$, if $H$ contains all $u^B_{kl}$, then it contains $B$. This means that under the given assumption, $H$ contains both $A$ and $B$. Since $A*B$ is generated by $A$ and $B$ as a unital algebra, we conclude $H=A*B$. 
\end{proof}

\begin{definition}\label{def-glued} 
  Let $A$ and $B$ be matrix Hopf algebras with fundamental corepresentations $u_A=(u^A_{ij})_{i,j=1}^n$ and $u_B=(u^B_{kl})_{k,l=1}^m$, respectively. We define the {\em  glued free product} $A\gluedfree B$ as the Hopf subalgebra of $A*B$ generated by the products $u^A_{ij}u^B_{kl}$, which becomes a matrix Hopf algebra with fundamental corepresentation $(u^A_{ij}u^B_{kl})\in M_{nm}(A\gluedfree B)$. Similarly, the {\em glued tensor product} $A\gluedtensor B$ of matrix Hopf algebras is the Hopf subalgebra of $A\otimes B$ generated by the elements of the form $u^A_{ij}\otimes u^B_{kl}$, which becomes a matrix Hopf algebra with fundamental corepresentation $(u^A_{ij}\otimes u^B_{kl})\in M_{nm}(A\gluedtensor B)$. 
\end{definition} 

If $(\mathcal C(\mathbb G),u)$ and $(\mathcal C(\mathbb H),v)$ are compact matrix quantum groups, then 
\[\Pol{\mathbb G}*\Pol{\mathbb H}=\Pol{\mathbb G*\mathbb H}, \quad \mbox{ and } \quad \Pol{\mathbb G}\gluedfree \Pol{\mathbb H} = \Pol{\mathbb G\gluedfree \mathbb H}\]
agree with the usual notions due to Wang \cite{Wang95} and Tarrago and Weber \cite{TerragoWeber2017}.

In what follows, we adopt a part of Gromada's theory \cite{Gromada2022} of gluing compact quantum groups to the case of matrix Hopf algebras. However, we will focus only on those statements that we need for the analysis of the cosovereign Hopf algebra.

\begin{definition}
  A matrix Hopf algebra $(A,u)$ is \emph{$k$-reflective} for $k\in\N$ if the assignment
  \[
    u_{ij}\longmapsto\delta_{ij}g_k,
  \]
  where $g_k$ is the generator of $\C\Z_k$, extends to a Hopf algebra map from $\sigma\colon A\to \C\Z_k$.

  The \emph{degree of reflection} of $(A,u)$ is the largest $k$ such that $(A,u)$ is $k$-reflective and $0$ if $(A,u)$ is $k$-reflective for infinitely many $k\in\N$. 
\end{definition}

Note that if $(A,u)$ is $k$-reflective and $\ell$ divides $k$, then $(A,u)$ is also $\ell$-reflective. Furthermore,  the quotient of $(A,u)$ by the relations $u_{ij}=\delta_{ij}u_{11}$ is easily seen to be the group algebra of a cyclic group. Therefore, if $(A,u)$ has degree of reflection $0$, then this quotient must be isomorphic to $\C\Z$ and $(A,u)$ is $k$-reflective for all $k\in\N$.
Recall that $g$ denotes the generator of $\C\Z_2$.

\begin{lemma}\label{lem:glued-tensor-Z2}
  Let $(H,u)$ be a $2$-reflective matrix Hopf algebra. Then $H\cong H\gluedtensor\C\Z_2$ as Hopf algebras, where we view $(\C\Z_2,g)$ as a matrix Hopf algebra.
\end{lemma}
\begin{proof}
    The assumption of the lemma implies that the assignment
    \[
    H\longrightarrow H\otimes\C\Z_2: u_{ij}\longmapsto u_{ij}\otimes g
    \]
    gives rise to an algebra homomorphism $\phi=(\mathrm{id}\otimes \sigma)\circ\Delta_H\colon H\to H\otimes\C\Z_2$. Note that the antipode on $\C\Z_2$ is the identity map and, therefore, $\sigma\circ S=\sigma$. On the algebra generators $S^nu_{ij}$, $\phi$ evaluates as
    \[\phi(S^nu_{ij})=(S^n\otimes\sigma)\Delta^{(op_n)}(u_{ij})=S^n u_{ij}\otimes g;\]
    here $\Delta^{(op_n)}$ stands for the comultiplication $\Delta$ if $n$ is even and the opposite comultiplication $\Delta^{(op)}$ if $n$ is odd. It follows that the image of $\phi$ is the glued tensor product $H\gluedtensor \C\Z_2$, the algebra generated by all $S^n u_{ij}\otimes g$ inside $H\otimes \C\Z_2$.

    Next, consider the homomorphism
    \[
    \psi:={\rm id}\otimes\varepsilon_{\C\Z_2}:H\otimes\C\Z_2\longrightarrow H.
  \]
  This is easily seen to be a Hopf algebra map.
  Using the fact that $\varepsilon_{\C\Z_2}(g)=1$, we obtain
    \begin{align*}
    &(\phi\circ\psi)(S^nu_{ij}\otimes g)=\phi(S^nu_{ij})=S^nu_{ij}\otimes g,\\
    &(\psi\circ\phi)(S^nu_{ij})=({\rm id}\otimes\varepsilon)(S^nu_{ij}\otimes g)=S^nu_{ij}.
    \end{align*}
    Since the $S^nu_{ij}\otimes g$ generate $H\gluedtensor \C\Z_2$ as an algebra, we can conclude that $\psi|_{H\gluedtensor \C\Z_2}$ is a Hopf isomorphism with inverse $\phi$.
\end{proof}

\begin{lemma}\label{lem:gluedtensor-fluedfree-isomorphism}
  Let $(H,u)$ be a matrix Hopf algebra. Then there is an isomorphism of Hopf algebras
  \[
    H\gluedfree\C\Z\cong (H\gluedtensor\C\Z_2)\gluedfree\C\Z_2.
  \]
\end{lemma}
\begin{proof}
  We view $(\C\Z\cong\C[z,z^{-1}],z)$, $(\C\Z_2,s)$, and $(\C\Z_2,r)$ as matrix Hopf algebras; here $s$ and $r$ denote the canonical generator of $\C\Z_2$ in the first and second copy in $(H\gluedtensor\C\Z_2)\gluedfree\C\Z_2$, respectively. All antipodes are denoted by $S$ without distinction. First, we define a homomorphism $\phi\colon H\ast\C\Z\to (H\otimes\C\Z_2)\ast\C\Z_2$ via the assignments
  \[
    a\longmapsto a\otimes 1,\qquad z\longmapsto (1\otimes s)r,\qquad z^{-1}\longmapsto r(1\otimes s).
  \]
  One easily checks that $\phi$ is a Hopf algebra map. We find that
  \begin{align}\label{phi-on-generators}
    \phi(S^k(u_{ij}z))=S^k\phi(u_{ij}z)=S^k((u_{ij}\otimes s)r),
  \end{align}
  i.e.\ $\phi$ maps the generators of $H\gluedfree\C\Z$ to the generators of $(H\gluedtensor\C\Z_2)\gluedfree\C\Z_2$. In particular, $\phi$ restricts to a surjective Hopf algebra map $H\gluedfree\C\Z\to (H\gluedtensor\C\Z_2)\gluedfree\C\Z_2$.
  
  To prove injectivity of $\phi$, we will construct its inverse. We define an algebra homomorphism $\psi\colon (H\otimes\C\Z_2)\ast\C\Z_2\to M_2(H\ast\C\Z)$:
  \[
    a\otimes 1\longmapsto \begin{bmatrix} a & 0\\0 & a \end{bmatrix},\qquad 1\otimes s\longmapsto\begin{bmatrix}0 & 1\\1 & 0\end{bmatrix},\qquad 
    r\longmapsto \begin{bmatrix} 0 & z^{-1}\\z & 0\end{bmatrix}.
  \]
  On the generators $S^k((u_{ij}\otimes s)r)$ of $(H\gluedtensor\C\Z_2)\gluedfree\C\Z_2$, for even $k$, we find that 
  \[\psi(S^k((u_{ij}\otimes s)r))=
    \psi((S^k(u_{ij}) \otimes s)r) =
    \begin{pmatrix}
      S^k(u_{ij})z&0\\0&S^k(u_{ij})z^{-1}
    \end{pmatrix}
    =\begin{pmatrix}
      S^k(u_{ij}z)&0\\0&S^k(u_{ij}z^{-1})
    \end{pmatrix}
  \]
  and for odd $k$, we obtain that
  \[
    \psi(S^k((u_{ij}\otimes s)r))=\psi(r(S^k(u_{ij}) \otimes s)) =
    \begin{pmatrix}
      z^{-1}S^k(u_{ij})&0\\0&zS^k(u_{ij})
    \end{pmatrix}=
    \begin{pmatrix}
      S^k(u_{ij}z)&0\\0&S^k(u_{ij}z^{-1})
    \end{pmatrix}.
  \]
  Summarizing, we have that
  \begin{align}\label{psi-on-generators}
    \psi(S^k((u_{ij}\otimes s)r))
    =\begin{pmatrix}
      S^k(u_{ij}z)&0\\0&S^k(u_{ij}z^{-1})
    \end{pmatrix}
  \end{align}
  for all $k\in\mathbb N$. In particular, we note that $\psi((H\gluedtensor\C\Z_2)\gluedfree\C\Z_2)$ is contained in the diagonal matrices.
  Let $\operatorname{pr}_1\colon M_2(H\ast\C\Z)\to H\ast\C\Z$ be the projection onto the upper left corner. Then the restriction $\operatorname{pr}_1\circ\psi|_{(H\gluedtensor\C\Z_2)\gluedfree\C\Z_2}$ is an algebra homomorphism due to the observed containment of $\psi((H\gluedtensor\C\Z_2)\gluedfree\C\Z_2)$ in the diagonals.  Using \eqref{phi-on-generators} and \eqref{psi-on-generators}, it is apparent that
  \begin{align*}
    ((\operatorname{pr}_1\circ\psi)\circ\phi)(S^k(u_{ij}z))=\operatorname{pr}_1(\psi(S^k((u_{ij}\otimes s)r)))
    =S^k(u_{ij}z)
  \end{align*}
  and we conclude that  $(\operatorname{pr}_1\circ\psi)\circ\phi)|_{H\gluedfree\C\Z}=\mathrm{id}_{H\gluedfree\C\Z}$ because the maps involved are algebra homomorphisms. This shows that $\phi|_{H\gluedfree\C\Z}$ is injective and, hence, a Hopf isomorphism onto $(H\gluedtensor\C\Z_2)\gluedfree\C\Z_2$ as claimed.
\end{proof}

\begin{observation}
  Let $A$, $B$ be matrix Hopf algebras with fundamental corepresentations $u^A\in M_n(A), u^B\in M_m(B)$. If the antipodes of $A$ and $B$ restrict to linear automorphisms of $\operatorname{Lin}(u^A):=\operatorname{Lin}\{u^A_{ij}\colon i,j=1,\ldots,n\}$ and $\operatorname{Lin}(u^B):=\operatorname{Lin}\{u^B_{ij}\colon i,j=1,\ldots,m\}$, respectively, then $A\gluedfree B$ is generated as an algebra by the collection of all $u^A_{ij}u^B_{kl}$ and $u^B_{kl}u^A_{ij}$, $i,j\in\{1,\ldots, n\}, k,l\in\{1,\ldots, m\}$. Indeed, it easily follows from the assumption that
  \begin{align*}
    S(u^A_{xy}u^B_{zw})=S(u^B_{zw})S(u^A_{xy})&\in \operatorname{Lin}\{u^B_{kl}u^A_{ij}:i,j\in\{1,\ldots, n\}, k,l\in\{1,\ldots, m\}\}\\ u^B_{zw}u^A_{xy}= S(S^{-1}(u^A_{xy})S^{-1}(u^B_{zw}))&\in\operatorname{Lin} (S(u^A_{ij}u^B_{kl}):i,j\in\{1,\ldots, n\}, k,l\in\{1,\ldots, m\})
  \end{align*}
  and, therefore, the generated algebras agree.
\end{observation}

In \cite[Theoreme 1(iv)]{Banica1997} (see also \cite[Remark 6.18]{TerragoWeber2017}) Banica observed that  $U_n^+\cong O_n^+\gluedfree \mathbb Z$. In fact, Gromada \cite[Theorem C/4.28]{Gromada2022} showed that one can replace $\mathbb Z$ by 
$\mathbb Z_2$, using the fact that the degree of reflection of $O_n^+$ is 2. We will now generalize those observations to the matrix Hopf algebras $\mathcal B(E)$ and $\mathcal H(F)$.

Let us observe that $\mathcal B(E)$ and $\C\Z_2$ satisfy the assumption that the antipode restricts to linear automorphisms of $\operatorname{Lin}\{x_{ij}\colon i,j=1,\ldots,n\}$ and $\operatorname{Lin}\{g\}$, respectively. Indeed, in $\mathcal B(E)$ we have $S(x)=E^{-1}x^tE$ and in $\C\Z_2$ we have $S(g)=g$. Thus, $\mathcal B(E)\gluedfree \C\Z_2$ is generated as an algebra by the elements $x_{ij}g$ and $gx_{ij}$, $i,j\in\{1,\ldots,n\}$.

\begin{lemma}\label{lem:degree_of_reflection-B(E)}
  The Hopf algebra $\mathcal B(E)$ has the degree of reflection 2.
\end{lemma}

\begin{proof}
  Denote by $g$ the canonical generator of $\C\Z_2$ and $\hat x_{ij}:=\delta_{ij}g$, $\hat x=(\hat x_{ij})_{i,j=1}^n=I_ng\in M_n(\C\Z_2)$. Clearly,
  \begin{align*}
E^{-1}\hat x^tE\hat x=g^2I_n=I_n=\hat xE^{-1}\hat x^tE,
  \end{align*}
  which proves that there is a Hopf algebra map sending $x_{ij}$ to $\hat x_{ij}=\delta_{ij}g$, so $\mathcal B(E)$ is 2-reflective. On the other hand, let $A$ be an arbitrary unital algebra and $a\in A$. If $x_{ij}\mapsto \delta_{ij} a$ extends to a unital algebra homomorphism, it follows from the relations of $\mathcal B(E)$ that $a^2=1$. This shows that $\mathcal B(E)$ is not $k$-reflective for any $k>2$.    
\end{proof}

The following observation reveals a glued free product structure on $\mathcal H(F)$.
\begin{theorem}\label{thm:H(F)=B(E)gluedCZ2}
	For $F=E^tE^{-1}$ a generic asymmetry, we have $\mathcal H(F) \cong \mathcal B(E) \gluedfree \C\Z_2$.
\end{theorem}

\begin{proof}
  From \Cref{lem:glued-tensor-Z2,lem:gluedtensor-fluedfree-isomorphism,lem:degree_of_reflection-B(E)}, we conclude that $\mathcal B(E) \gluedfree \C\Z_2\cong  \mathcal B(E) \gluedfree \C\Z$. Denote the canonical generator of $\C\Z$ by $z$.
  Put $\hat u_{ij}:= x_{ij}z, \hat v_{ij}:=S(\hat u_{ji})=z^{-1} S(x_{ji})\in  \mathcal B(E) \gluedfree \C\Z$. In matrix notation,
  \[\hat u=xz, \quad\hat u^t=x^tz,\quad \hat v=z^{-1}E^{t}x  E^{-t} ,\quad\hat v^t=z^{-1}E^{-1}x^{t}E.\] We claim that there is a (necessarily unique) algebra homomorphism $\Phi\colon \mathcal H(F) \to \mathcal B(E) \gluedfree \C\Z$ with $\Phi(u_{ij})=\hat u_{ij}$ and $\Phi(v_{ij})=\hat v_{ij}$ for all $i,j\in\{1,\ldots, n\}$. By definition of $\mathcal H(F)$, it is enough to check the following relations (note that $z$ commutes with scalar matrices):
  \begin{align*}
	\hat u\hat v^t & = xz\,z^{-1}E^{-1}x^tE = I\\
	\hat v^t\hat u & = z^{-1}E^{-1}x^tE\,xz = I \\
	\hat vF\hat u^tF^{-1} & = z^{-1}E^tx E^{-t}\, E^tE^{-1}\, x^t z\,E E^{-t}
	 = z^{-1}E^t z E^{-t} = I \\ 
	F\hat u^tF^{-1}\hat v & = E^tE^{-1}\, x^t z\, E E^{-t}\,z^{-1} E^tx E^{-t} 
	= E^t(E^{-1} x^t Ex) E^{-t} = E^t  E^{-t} =I.
  \end{align*}
  The assumption of genericity implies that both $\mathcal H(F)$ and $\mathcal B(E)$ are cosemisimple (see Preliminaries). 
  In order to show that $\Phi$ is an isomorphism, we invoke the theory of cosemisimple Hopf algebras. By \cite[Lemma 5.1]{Bichon03}, if $f\colon A\to B$ is a Hopf algebra morphism between cosemisimple Hopf algebras $A,B$ which induces an isomorphism of their respective representation semirings $R^+(A)\cong R^+(B)$, then $\Phi$ is a Hopf algebra isomorphism. Let $q$ be the solution of the equation $q^2-\tr(F)q+1=0$. We know from \cite[Theorem 1.1 (i)]{Bichon07} that for generic $F$, $\mathcal H(F)$ is monoidaly equivalent to $\mathcal H(F_q)$, where $F_q=\begin{pmatrix} q^{-1} & 0 \\ 0 & q \end{pmatrix}$. Next, by~\cite[Theorem~1.1]{Bichon03}, $\mathcal{B}(E)$ is monoidaly equivalent to $\mathcal{B}(E_q)\cong\operatorname{Pol}(SL_q(2))$, where $E_q=\begin{pmatrix} 0 & 1 \\ q^{-1} & 0 \end{pmatrix}$. The isomorphism of~\cite[Lemma~3.3]{Bichon07} induces an isomorphism of Hopf algebras $\mathcal{H}(F_q)\cong \operatorname{Pol}(SL_q(2))\gluedfree \mathbb{C}\mathbb{Z}$. We infer that $R^+(\mathcal{H}(F))\cong R^+(\mathcal{B}(E)\gluedfree\mathbb{C}\mathbb{Z})$ as semirings. From \cite[Theorem 1.1 (iii)]{Bichon07}, we know that $R^+(\mathcal H(F))$ is generated by $u$ and $v$. From the construction of the monoidal equivalence of $\mathcal{B}(E)$ and $\mathcal{B}(E_q)$ (see e.g.~\cite{Schauenburg96}) and the representation theory of $SL_q(2)$ (see e.g.~\cite[p.~4845]{Bichon03}), we obtain that $R^+(\mathcal B(E))$ is generated by $x$. Therefore, $\Phi$ induces the aforementioned isomorphism of semirings.
\end{proof}

\begin{remark}
    In a \cite[Prop.~4.3]{Bichon23pre}, Bichon shows that \Cref{thm:H(F)=B(E)gluedCZ2} holds without the genericity assumption.
\end{remark}

In the sequel we will identify $\mathcal H(F)$ with $\mathcal B(E)\gluedfree\C\Z_2$, writing (in the matrix notation) just $u=xg$ and $v=gE^tx E^{-t}$. Then it follows that $u^t=x^t g$, $v^t=gE^{-1}x^tE$ and 
\begin{equation} \label{eq_gx}
	gx =  E^{-t}vE^t \quad \mbox{and} \quad gx^t = Ev^tE^{-1}.
\end{equation}

\section{Free resolution for \texorpdfstring{$\mathcal H(F)$}{H(F)}}
\label{sec:resolution-H(F)}

\begin{proposition}
  \label{prop:A=H+gH_general}
  Consider $\mathcal A(E)$ as a right $\mathcal H(F)$-module, with the canonical module structure induced by the embedding $\mathcal H(F)\subset \mathcal A(E)$, i.e.\ $a.h:=ah\in\mathcal A(E)$ for all $a\in \mathcal A(E), h\in\mathcal H(F)$.  Then $\mathcal H(F),g\mathcal H(F)\subset\mathcal A(E)$ are submodules and 
  \[\mathcal A(E)= \mathcal H(F) \oplus g \mathcal H(F).\]
\end{proposition}

\begin{proof}
  Any element of the $\mathcal A(E)=\mathcal B(E) \ast \C\Z_2$ is a linear combination of \emph{monomials}, i.e.\ products of generators $x_{jk}$ from $\mathcal B(E)$ and $g\in \C\Z_2$.
  
	Take any such monomial $w$, and assume that its length (the number of generators of both kinds) is $k$. For $k=0$ we have $w=1=1\oplus g\cdot 0$. If $k=1$ then either 
	$w=x_{ij}=0\oplus g(gx_{ij})$ or $w=g=0\oplus  g1$. Now assume that any element $w$ of length $\leq k$ can be 
	written as $w=a\oplus gb$ with $a,b\in \mathcal H(F)$. Consider an element $w'$ of length $k+1$. We have either $w'=gw=g(a\oplus gb) = b\oplus ga\in \mathcal H(F) \oplus g\mathcal H(F)$ or $$w'=x_{ij}w=g(gx_{ij})(a\oplus gb) = 
	g(gx_{ij})a\oplus g^2(x_{ij}g)b= \widetilde{b} \oplus g\widetilde{a},$$ 
	with $\widetilde{b}=(u_{ij}g)b, \widetilde{a}=(gu_{ij})a\in \mathcal H(F)$. 
	
	We still need to show that $ \mathcal H(F) \cap g \mathcal H(F)=\{0\}$. This is due to the fact that $\mathcal A(E)$ is a $\Z_2$-graded algebra, the free product of $\Z_2$-graded algebras $\mathcal B(E)$ and $\C\Z_2$. Elements of $\mathcal H(F)$ are combinations of monomials of even length, whereas elements of $g\mathcal H(F)$ are made up from monomials of odd length.
\end{proof}

\begin{theorem}\label{thm:resHF}
Let $F=E^tE^{-1}\in \operatorname{GL}_n(\C)$ be a generic asymmetry and $\mathcal H=\mathcal H(F)$. Then the following sequence 
\begin{equation} \label{eq:resHF}
	0\to \begin{pmatrix} \mathcal H \\ \mathcal H \end{pmatrix} 
	\xrightarrow{\Phi^{\mathcal H}_3} 
	\begin{pmatrix} M_n(\mathcal H) \\ M_n(\mathcal H) \end{pmatrix} 
	\xrightarrow{\Phi^{\mathcal H}_2}
	\begin{pmatrix} M_n(\mathcal H) \\ M_n(\mathcal H)
		\\ \mathcal H \end{pmatrix} \xrightarrow{\Phi^{\mathcal H}_1}
	\begin{pmatrix} \mathcal H \\ \mathcal H \end{pmatrix} \xrightarrow{\varepsilon}\mathbb{C}_\varepsilon\to0    
\end{equation}
with the mappings
\begin{align*}
	\Phi^{\mathcal H}_3 \begin{pmatrix}a \\ b \end{pmatrix} 
	& = \begin{pmatrix} -Fa+(Eu E^{-t})b \\ F^{-1}va-Fb
	\end{pmatrix},\\
	\Phi^{\mathcal H}_2  \begin{pmatrix} A \\ B \end{pmatrix} 
	& = \begin{pmatrix} A+ (E^{-1}u^tBE)^t \\ B+(v^tE^{-1}AE)^t \\0  \end{pmatrix},\\ 
	\Phi^{\mathcal H}_1 \begin{pmatrix} A \\ B \\c \end{pmatrix} 
	& = \begin{pmatrix} \operatorname{tr} (-A+u^tB) +c \\ \operatorname{tr} (Ev^tE^{-1}A-B)-c  \end{pmatrix},
\end{align*}
for $a,b,c\in \mathcal H$, $A,B\in M_n(\mathcal H)$,
yields a projective resolution of the counit of $\mathcal H$.
\end{theorem}
\begin{remark}
    It is worth noticing that the given resolution is in fact free.
\end{remark}

\begin{proof}
We know from Section \ref{sec_glued}, with $F=E^tE^{-1}$, that 
\begin{equation} \label{eq_glued-free-product}
	\mathcal H(F)=\mathcal B(E) \gluedfree \mathbb C\mathbb Z_2 
	\subset \mathcal B(E) * \mathbb C\mathbb Z_2 = \mathcal A(E).
      \end{equation}
      is a Hopf subalgebra.

Our plan is to interpret the projective modules in the resolution \eqref{eq_resolution_of_A} as $\mathcal H(F)$-modules by restricting the action, to identify bases which will show that they are actually free $\mathcal H(F)$-modules, and finally to express the maps $\widetilde\Phi_i$ in terms of the chosen bases. 

To simplify the notation, we write $\mathcal H=\mathcal H(F)$ 
and $\mathcal A=\mathcal A(E)$ with $F=E^tE^{-1}$. Then Proposition \ref{prop:A=H+gH_general} 
states that $\mathcal A=\mathcal H\oplus g\mathcal H$ is a free $\mathcal H$-module with basis $(1,g)$. Consequently, $M_n(\mathcal A) \cong M_n(\mathcal H) \oplus g M_n(\mathcal H)$ is a free $\mathcal H$-module with basis $(e_{ij},ge_{ij}:i,j=1,\ldots n)$, 
where we naturally set $g(a_{ij})_{i,j+1}^n := (ga_{ij})_{i,j+1}^n$. Moreover, 
since for $a,b\in \mathcal H$, we have
\[
  (1-g)\otimes (a+gb)= (1-g)\otimes a + (1-g)g\otimes b = (1-g)\otimes (a-b)= 1\otimes (1-g) (a-b),
\]
the mapping $(1-g)\otimes h\mapsto (1-g)h$ constitutes an isomorphism $\mathbb{C}(1-g)\otimes_{\mathbb{C}\mathbb{Z}_2}\mathcal A\cong (1-g) \mathcal H$ to a free $\mathcal H$-module with basis $(1-g)$ (we use this isomorphism to identify those expressions in the following). 
This allows us to write down the resolution \eqref{eq_resolution_of_A} in terms of free $\mathcal H$-modules:
\begin{equation*}
	0\to \mathcal H\oplus g\mathcal H \xrightarrow{\Phi^{\mathcal A}_3} 
	M_n(\mathcal H) \oplus  gM_n(\mathcal H)
	\xrightarrow{\Phi^{\mathcal A}_2}
	M_n(\mathcal H) \oplus  gM_n(\mathcal H) \oplus (1-g) \mathcal H \xrightarrow{\Phi^{\mathcal A}_1} 
	\mathcal H \oplus g\mathcal H 
	\xrightarrow{\varepsilon}\mathbb{C}_\varepsilon\to0.
\end{equation*}

Let us now see how the action of the respective mappings 
can be expressed in the chosen bases. It will be convenient to use the matrix notation with $I=(\delta_{jk}1_{\mathcal H})_{j,k=1}^n$, $x:=(x_{jk})_{j,k=1}^n$ $u:=xg=(x_{jk}g)_{j,k=1}^n$, $v:=gE^tx E^{-t}$. Recall that $g$ commutes with scalar matrices and $g^2=1$.
For $a,b\in \mathcal H$,
\begin{align*}
                                    \Phi^{\mathcal A}_3(a\oplus gb)
	&=\bigl(Ex E^{-t}-F\bigr)(a\oplus gb)\\
	& =(-Fa+Ex E^{-t}gb)\oplus(Ex E^{-t}a-Fgb)\\
	& =(-Fa+Eu E^{-t}b)\oplus g(E E^{-t}gE^tx E^{-t}a-Fb)\\
	& = (-Fa+(Eu E^{-t})b) \oplus g(F^{-1}va-Fb).
\end{align*}

For $A,B\in M_n(\mathcal H)$, using that $v:=gE^tx E^{-t}$ is equivalent to $gE^{-1}x^t=v^tE^{-1}$, we have that
\begin{align*}
  \Phi^{\mathcal A}_2(A\oplus gB)
  &=A+gB+(E^{-1}x^t(A+gB)E)^t\\
  &=(A+(E^{-1}x^tgBE)^t)\oplus ((E^{-1}x^tAE)^t+gB)\oplus 0\\
  &=(A+(E^{-1}u^tBE)^t)\oplus g((gE^{-1}x^tAE)^t+B)\oplus 0\\
  &=(A+(E^{-1}u^tBE)^t)\oplus g((v^tE^{-1}AE)^t+B)\oplus 0.  
\end{align*}

Finally, for $A,B\in M_n(\mathcal H)$ and $c\in \mathcal H$, we have
\begin{align*}
  \Phi^{\mathcal A}_1 (A\oplus gB  \oplus (1-g) c) &=  \tr(x^t(A+gB)-\tr(A+gB)+(1-g)c\\
  & = (\tr (-A+x^tgB) +c) \oplus ( \tr (x^tA-gB)- gc) \\
  &=(\tr (-A+u^tB) +c) \oplus g( \tr (gx^tA-B)- c) \\
  &=(\tr (-A+u^tB) +c) \oplus g( \tr (Ev^tE^{-1}A-B)- c),
\end{align*}
this time using the fact that $v:=gE^tx E^{-t}$ is equivalent to $gx^t=Ev^tE^{-1}$.

Now, let us observe that the bases described above give rise to $\mathcal H$-module isomorphisms:
\begin{gather*}
  \mathcal H \oplus g \mathcal H \ni a\oplus gb \mapsto \begin{pmatrix} a \\ b \end{pmatrix} \in \begin{pmatrix}\mathcal H  \\ \mathcal H \end{pmatrix},\\
  M_n(\mathcal H)\oplus gM_n(\mathcal H)\ni A\oplus gB \mapsto \begin{pmatrix} A \\ B  \end{pmatrix} \in \begin{pmatrix} M_n(\mathcal H) \\ M_n(\mathcal H)\end{pmatrix}\\
  M_n(\mathcal H)\oplus gM_n(\mathcal H)\oplus (1-g)\mathcal H\ni A\oplus gB \oplus (1-g)h \mapsto \begin{pmatrix} A \\ B \\ h \end{pmatrix} \in \begin{pmatrix} M_n(\mathcal H) \\ M_n(\mathcal H) \\ \mathcal H\end{pmatrix}
\end{gather*}
and the maps $\Phi^{\mathcal H}_*$ are simply the maps $\Phi^{\mathcal A}_*$ rewritten in terms of the bases. In particular, the sequence \eqref{eq:resHF} is exact.
\end{proof}

\begin{lemma} \label{lem-res-with-F-only}
    The maps $f_E,T\colon M_n(\mathcal H)\to M_n(\mathcal H)$ with
    \begin{align*}
        f_E(A)&=E^{-1}AE&T(A)=A^t
    \end{align*}
    are $\mathcal H$-module automorphisms. These lead to a transformed version of the resolution \eqref{eq:resHF} in \Cref{thm:resHF} only depending on $F$:
    \begin{equation} 
	\label{eq:resHF_transformed}
	0\to \begin{pmatrix} \mathcal H \\ \mathcal H \end{pmatrix} 
	\xrightarrow{\Psi^{\mathcal H}_3} 
	\begin{pmatrix} M_n(\mathcal H) \\ M_n(\mathcal H) \end{pmatrix} 
	\xrightarrow{\Psi^{\mathcal H}_2}
	\begin{pmatrix} M_n(\mathcal H) \\ M_n(\mathcal H)
		\\ \mathcal H \end{pmatrix} \xrightarrow{\Psi^{\mathcal H}_1}
	\begin{pmatrix} \mathcal H \\ \mathcal H \end{pmatrix} \xrightarrow{\varepsilon}\mathbb{C}_\varepsilon\to0   
    \end{equation}
    with 
    \begin{align*}
        \Psi_3^{\mathcal H}\begin{pmatrix} a \\ b \end{pmatrix}&=\begin{pmatrix} f_E \\ \mathrm{id} \end{pmatrix}\circ  \Phi_3^{\mathcal H}\begin{pmatrix} a \\ b \end{pmatrix} = 
        \begin{pmatrix}
            -F^{-t} a + u F^tb\\ F^{-1} va - Fb
        \end{pmatrix},
        \\
        \Psi_2^{\mathcal H}\begin{pmatrix} A \\ B \end{pmatrix}&=\begin{pmatrix} T f_E \\ T \\ \mathrm{id} \end{pmatrix}\circ  \Phi_2^{\mathcal H}\circ \begin{pmatrix} f_E^{-1} \\ \mathrm{id} \end{pmatrix}\begin{pmatrix} A \\ B \end{pmatrix}
        =\begin{pmatrix} A^t + Fu^tBF^{-1} \\  v^t A + B^t \\ 0 \end{pmatrix},
        \\
        \Psi_1^{\mathcal H}\begin{pmatrix} A \\ B \\ c\end{pmatrix}&=  \Phi_1^{\mathcal H}\circ \begin{pmatrix} f_E^{-1}T \\ T \\ \mathrm{id} \end{pmatrix}\begin{pmatrix} A \\ B \\ c \end{pmatrix}
        =\begin{pmatrix} \tr(-A+uB)+c \\ \tr(vA-B)-c \end{pmatrix}.
    \end{align*}
\end{lemma}

\begin{proof}
    It is clear that $f_E$ and $T$ are $\mathcal H$-module automorphisms and, accordingly, they transform a resolution into a resolution (note that $T^{-1}=T$). We are left with checking that the concrete expressions for the $\Psi_i^{\mathcal H}$ are correct. Indeed, we have that
    \begin{align*}
        \Psi_3^{\mathcal H}\begin{pmatrix} a \\ b \end{pmatrix}&=\begin{pmatrix} f_E \\ \mathrm{id} \end{pmatrix}\circ  \Phi_3^{\mathcal H}\begin{pmatrix} a \\ b \end{pmatrix} = 
        \begin{pmatrix}
            -E^{-1}E^t E^{-1} Ea + E^{-1} E u E^{-t}E b\\ F^{-1} va - Fb
        \end{pmatrix}
        =
        \begin{pmatrix}
            -F^{-t} a + u F^tb\\ F^{-1} va - Fb
        \end{pmatrix},
        \\
        \Psi_2^{\mathcal H}\begin{pmatrix} A \\ B \end{pmatrix}&=\begin{pmatrix} T f_E \\ T \\ \mathrm{id} \end{pmatrix}\circ  \Phi_2^{\mathcal H}\begin{pmatrix} EAE^{-1} \\ B \end{pmatrix}
        =
        \begin{pmatrix} T f_E \\ T \\ \mathrm{id} \end{pmatrix}\begin{pmatrix} EAE^{-1}+ (E^{-1}u^tBE)^t \\ (v^tE^{-1}EAE^{-1}E)^t + B \\0  \end{pmatrix}
        \\
        &=\begin{pmatrix} A^{t}+ (E^{-1}(E^{-1}u^tBE)^tE)^t \\ v^tE^{-1}EAE^{-1}E + B^t \\0  \end{pmatrix}
        \\
        &=\begin{pmatrix} A^{t}+ E^{t}E^{-1}u^tBEE^{-t} \\ v^tA + B^t \\0  \end{pmatrix}
        \\
        &=
        \begin{pmatrix} A^t + Fu^tBF^{-1} \\ v^t A + B^t \\ 0 \end{pmatrix},
        \\
        \Psi_1^{\mathcal H}\begin{pmatrix} A \\ B \\ c \end{pmatrix}
        &=  \Phi_1^{\mathcal H}\begin{pmatrix} EA^tE^{-1} \\ B^t \\ c \end{pmatrix}
        = \begin{pmatrix} \operatorname{tr} (-EA^tE^{-1}+u^tB^t) +c \\ \operatorname{tr} (Ev^tE^{-1}EA^tE^{-1}-B^t)-c  \end{pmatrix} 
        =\begin{pmatrix} \tr(-A+uB)+c \\ \tr(vA-B)-c \end{pmatrix}.
    \end{align*}
\end{proof}

In fact, we will prove later that the sequence \eqref{eq:resHF_transformed} is exact, hence a free resolution, whenever $F$ is normalizable. 

\section{Hochschild cohomology for \texorpdfstring{$\mathcal H(F)$}{H(F)}}
\label{sec:Hochschild-cohomology-H(F)}

Recall that we restrict to 1-dimensional bimodules $_\e\C_\tau$ with trivial left action, whose associated right modules are simply $(_\e\C_\tau)''=\C_\tau$, as detailed in \Cref{1-dim-coefficients}. 

\begin{theorem}\label{thm:Hochschild-U+}
Let $F\in \operatorname{GL}_n(\C)$ be such that \eqref{eq:resHF_transformed} is exact,\footnote{For now, we have shown exactness whenever $F$ is a generic asymmetry. In \Cref{cor:normalizable_implies_exactness} below, we will see that it is enough to assume $F$ to be generic, which will complete the proof of \Cref{main-result:Hochschild}.} and let $\tau\colon \mathcal H(F) \to \C$ be a character. Set $S:=\tau(u)$,  $T:=\tau(v)=S^{-t}$ and
denote by $\mathcal{K}$ the space of matrices commuting with $F^tS$. Let $$d=\operatorname{dim}\, \mathcal{K}, \quad p=\begin{cases}
    1 & \mbox{ if } S=T=I_n\\
    0 & \mbox{ otherwise; }
\end{cases}, \quad
t=\begin{cases}
    1 & \mbox{ if } F^2=\alpha T \mbox{ for some $\alpha\in\C$, }\\
    2 & \mbox{ otherwise. }
\end{cases}
$$
Then the Hochschild cohomology for $\mathcal H(F)$ with 1-dimensional coefficients is given as follows:
\begin{align*}
    \dim \mathrm{H}^0(\mathcal H, {_\e\C_\tau}) &= p ,
    &\dim \mathrm{H}^1(\mathcal H, {_\e\C_\tau}) &= d+p-1 ,
    \\
    \dim \mathrm H^2(\mathcal H, {_\e\C_\tau}) &= d-t ,
    &\dim \mathrm H^3(\mathcal H, {_\e\C_\tau}) &= 2-t ,
\end{align*}
and $\dim \mathrm H^k(\mathcal H, {_\e\C_\tau}) = 0$ for all $k \geq 4 $.
\end{theorem}

\begin{proof}
From the resolution \eqref{eq:resHF_transformed} for $\mathcal H$, we construct the cochain complex
\begin{multline}
\label{eq:complex-for-Hoschschild-H(F)}
0\to \hom_{\mathcal H}\left(\begin{pmatrix} \mathcal H \\ \mathcal H \end{pmatrix},\C_\tau\right) 
\xrightarrow{-\circ \Psi^{\mathcal H}_1} 
\hom_{\mathcal H}\left(\begin{pmatrix} M_n(\mathcal H) \\ M_n(\mathcal H)\\ \mathcal H \end{pmatrix},\C_\tau\right)
\\
\xrightarrow{-\circ \Psi^{\mathcal H}_2}
\hom_{\mathcal H}\left(\begin{pmatrix} M_n(\mathcal H) \\ M_n(\mathcal H) \end{pmatrix},\C_\tau\right)
\xrightarrow{-\circ\Psi^{\mathcal H}_3}
\hom_{\mathcal H}\left(\begin{pmatrix} \mathcal H \\ \mathcal H \end{pmatrix},\C_\tau\right)\to0.
\end{multline}

Recall from Lemma \ref{lemma:traces-free-modules} that we have vector space isomorphisms
\[\hom_{\mathcal H}(M_{n_1}(\mathcal H)\oplus\ldots \oplus M_{n_k}(\mathcal H),\C_\tau)\cong M_{n_1}(\C)\oplus\ldots \oplus M_{n_k}(\C)\]
determined by the ``pairing''
    \[(\Lambda_1,\ldots,\Lambda_k)^\tau(A_1,\ldots, A_k)=\sum_{r=1}^k \tau\tr(\Lambda_r^t A_r),\]
    for $\Lambda_r\in M_{n_r}(\C)$ and $A_r\in M_{n_r}(\mathcal H)$.
    Also, the precompositions $-\circ\Psi^{\mathcal H}_j$ can be described by the corresponding maps $(\Psi^{\mathcal H}_j)^{*}$ which yield the isomorphic complex

    \[0\rightarrow \C \oplus \C \xrightarrow{(\Psi^{\mathcal H}_1)^{*}}M_n(\C) \oplus M_n(\C) \oplus \C\xrightarrow{(\Psi^{\mathcal H}_2)^{*}}M_n(\C) \oplus M_n(\C)\xrightarrow{(\Psi^{\mathcal H}_3)^{*}} \C \oplus \C\rightarrow 0,\]
    leading to
    \[\mathrm H^k\left( \mathcal H, {_\e\C_\tau}\right) \cong  \faktor{\ker((\Psi^{\mathcal H}_{k+1})^{*})}{\mathrm{Im}((\Psi^{\mathcal H}_k)^{*})}.\]

    We will now have a closer look at the maps $(\Psi^{\mathcal H}_j)^{*}$ ($j=1,2,3$). For $j=1$, we have 
\begin{align*}
(\Psi^{\mathcal H}_1)^{*} \colon \C \oplus \C &\to M_n(\C) \oplus M_n(\C) \oplus \C
\end{align*}
and 
\begin{align*}
  (\lambda,\mu)^{\tau}\circ\Psi_1^{\mathcal H}  \begin{pmatrix} A \\ B \\c \end{pmatrix} 
  & = \lambda \, \tau\big( \tr  (-A+uB) +c\big)
    + \mu \, \tau\big(\tr  (vA - B)-c\big) 
  \\ 
  & = \lambda \, \tau\tr  (-A)+\lambda \, \tau\tr(SB) +\lambda \tau(c)
    + \mu \, \tau \tr (TA)-
    \mu \, \tau\tr (B)
    -\mu\tau(c)
  \\ 
  & = \tau\tr\bigl((-\lambda I_n + \mu T^t)^tA \bigr) + \tau\tr\bigl((\lambda S^t-\mu I_n)^tB\bigr)+\tau\bigl((\lambda-\mu)c\bigr)\\
  &=\left(-\lambda I_n + \mu T^t, \lambda S^t-\mu I_n, \lambda - \mu\right)^{\tau} \begin{pmatrix} A \\ B \\c \end{pmatrix},
\end{align*}
i.e.,
\begin{align*}
  (\Psi^{\mathcal H}_1)^{*}  (\lambda, \mu)= \left(-\lambda I_n + \mu T^t, \lambda S^t-\mu I_n, \lambda - \mu\right).
\end{align*}
If $(\lambda,\mu)\in \ker (\Psi^{\mathcal H}_1)^{*}$, then necesarily $\lambda=\mu$. Furthermore, $\lambda=\mu\neq0$ implies $T=I_n=S$. So
\begin{align*}
    \ker (\Psi^{\mathcal H}_1)^{*} &=\begin{cases}
    \{(\lambda,\mu):\lambda=\mu\}\cong \C, & \mbox{if}\,  S=T=I_n, \\
    \{0\}, & \mbox{otherwise}.
    \end{cases}\\
\end{align*}
Using the rank-nullity theorem, we conclude that
\[\dim\ker (\Psi^{\mathcal H}_1)^{*} =p,
\qquad
\dim\im (\Psi^{\mathcal H}_1)^{*} =2-p.
  \]
\medskip
For \(
    (\Psi^{\mathcal H}_2)^{*}\colon M_n(\C) \oplus M_n(\C) \oplus \C \to M_n(\C) \oplus M_n(\C)
\) we find that

\begin{align*}
(\Lambda,\Mu,\nu)^\tau \circ \Psi^{\mathcal H}_2 
  \begin{pmatrix} A \\ B \end{pmatrix} 
  & = (\Lambda,\Mu,\nu)^\tau \begin{pmatrix} A^t + Fu^tBF^{-1} \\  v^t A + B^t \\ 0 \end{pmatrix}\\
  & =
    \tau  \tr (\Lambda^t [A^t+  Fu^tBF^{-1}])
    + \tau  \tr (\Mu^t [v^t A + B^t])\\
  & =
    \tau  \tr (\Lambda^t A^t)+ \tau  \tr (\Lambda^tF S^t BF^{-1})
    + \tau  \tr (\Mu^t T^tA)+\tau  \tr (\Mu^tB^t])\\
  &= \tau  \tr (\Lambda + \Mu^t T^t) A+ \tau  \tr (F^{-1}\Lambda^tF S^{t}+ \Mu) B,
\end{align*}
i.e.\
\begin{align*}
  (\Psi^{\mathcal H}_2)^{*} (\Lambda, \Mu , \nu)= \left(\Lambda^t + T\Mu, SF^t\Lambda F^{-t}+\Mu^t  \right).
\end{align*}
Now, $(\Lambda,\Mu,\nu )\in  \ker (\Psi^{\mathcal H}_2)^{*}$ if and only if 
$$ \Lambda^t + T\Mu=0,\qquad SF^t\Lambda F^{-t} + \Mu^t =0.$$
Inserting $\Lambda$ from the first equation into the second one, we get
$$0= -S F^t \Mu^t T^t F^{-t} + \Mu^t,$$
which holds if and only if  
\begin{align*}
\Mu^t F^t S = S F^t \Mu^t \iff S^{-1} \Mu^t F^t S = F^t S S^{-1}\Mu^t  \iff [S^{-1}\Mu^t ,F^tS]=0.
\end{align*}
Let us denote by $\mathcal{K}$ the space of matrices commuting with $F^t S$, and by $d=\operatorname{dim}\, \mathcal{K}$. Then the preceding calculations show that
\begin{align*}
\mathcal K \oplus \C\ni(K,\nu)\mapsto \left(-SKS^{-1},(SK)^t,\nu\right) \in \ker  (\Psi^{\mathcal H}_2)^{*} 
\end{align*}
defines an isomorphism of vector spaces.
Hence, by the rank-nullity theorem,
\[\dim \ker (\Psi^{\mathcal H}_2)^{*} = d+1, \qquad \im (\Psi^{\mathcal H}_2)^{*} = 2n^2+1-(d+1) = 2n^2-d.\]

Finally, for $(\Psi^{\mathcal H}_3)^{*} \colon M_n(\C) \oplus M_n(\C) \to \C \oplus \C$ we find
\begin{align*}
    (\Lambda, \Mu )^{\tau} \circ \Psi^{\mathcal H}_3 \begin{pmatrix}a \\ b \end{pmatrix} 
    	& = (\Lambda, \Mu )^{\tau}\begin{pmatrix} -F^{-t} a + u F^tb\\ F^{-1} va - Fb
    	\end{pmatrix}
   \\ & = \tau \tr \big(\Lambda^t(-F^{-t} a + u F^tb)\big)
   + \tau \tr \big( \Mu ^t(F^{-1} va - Fb)\big)
   \\ & = \tr \big({-\Lambda^t F^{-t}}+\Mu ^tF^{-1}T\big)\tau(a)  
   + \tr \big(\Lambda^t SF^t -\Mu ^tF\big)\tau(b)
   \\ & = \tr \big({-F^{-1}\Lambda}+ T^{t}F^{-t} \Mu) \tau(a)  
   + \tr \big(FS^t\Lambda - F^t \Mu) \tau(b),
\end{align*}
so (recall that $S^t = T^{-1}$)
\begin{align*}
    (\Psi^{\mathcal H}_3)^{*}
    (\Lambda, \Mu ) = \left( \tr ({-F^{-1}\Lambda}+ T^{t}F^{-t} \Mu), \tr (FT^{-1}\Lambda - F^t \Mu) \right).
\end{align*}
With a matrix $A=(a_{i,j})\in \mathbb C^{n\times n}$ we associate the row vector and a column vector,
\[A^{\rm row}:=(a_{1,1},a_{1,2},\ldots, a_{2,1}, a_{2,2},\ldots, a_{n,n})\in \mathbb C^{1\times n^2},\quad A^{\rm col}:=(A^{\rm row})^t\in\mathbb C^{n^2\times 1},\] so that $\operatorname{tr}(X^tY)=X^{\rm row} Y^{\rm col}$ is the standard bilinear form. Then the map $({-}\circ\hat\Psi_3)$ is represented by the $2\times (2n^2)$-matrix 
    $M:=\begin{pmatrix}
        (-F^{-1})^{\rm row} & (T^tF^{-t})^{\rm row} \\ (FT^{-1})^{\rm row} & (-F^t)^{\rm row}   
    \end{pmatrix}$ and $\dim \left(\operatorname{im} ({-}\circ\hat\Psi_3)\right)$ is the rank of $M$, which is either 2 or 1, depending on whether the two rows are linearly independent or not. Of course, the rows of $M$ are linearly independent if and only if $(-F^{-1},T^tF^{-t})$ and $(FT^{-1},-F^t)$ are linearly independent, since this is just a reordering of the entries. If there is a nonzero $\alpha\in\mathbb C$ with $-F^{-1}=\alpha FT^{-1}$, we see that $T=-\alpha F^2$. In this case, the second condition, $T^t F^{-t}=-\alpha F^t$, is  automatically fulfilled, proving linear dependence.
  Using the rank-nullity theorem, we can thus conclude that
  \[\dim \ker  (\Psi^{\mathcal H}_3)^{*} =2n^2-t
\qquad 
\dim \im (\Psi^{\mathcal H}_3)^{*} =t.\qedhere
\]
\end{proof}
We now consider several special cases.

\begin{example}[Cohomology of \texorpdfstring{$\operatorname{Pol} (U_n^+)$}{Pol(Un+)} with trivial coefficients]
  For the special case $E=F=I$, i.e.\ $\mathcal H(F)=\operatorname{Pol} (U_n^+)$, and $M={_\e \C_\e}$ (so that $S=T=I$) we have
  $d=n^2$, $p=1$, $t=1$, so the Hochschild cohomology for $\operatorname{Pol} (U_n^+)$ with trivial coefficients is given as follows:
  \begin{align*}
    \dim \mathrm{H}^0(\mathcal H, {_\e\C_\e}) &= 1 ,
    &\dim \mathrm{H}^1(\mathcal H, {_\e\C_\e}) &= n^2 ,
    \\
    \dim \mathrm H^2(\mathcal H, {_\e\C_\e}) &= n^2-1 ,
    &\dim \mathrm H^3(\mathcal H, {_\e\C_\e}) &= 1 ,
  \end{align*}
  and $\dim \mathrm H^k(\mathcal H, {_\e\C_\e}) = 0$ for all $k \geq 4$.
  Let us note that all but $\dim \mathrm H^3(\mathcal H, {_\e\C_\e}) = 1 $ was already known, cf.\ \cite{DFKS18,Bichon18}
\end{example}

\begin{example}[Cohomology of \texorpdfstring{$\operatorname{Pol} (U_F^+)$}{Pol(UF+)} with \texorpdfstring{$F=\operatorname{diag}(q^{-1},1,q) $}{F=diag(q(-1),1,q)} and with trivial coefficients]\label{H^2(Pol(U_F^+))-geometric_progression}
  In \cite{DFKS23} it was shown that for $F\in M_3(\C)$ positive, $\mathrm H^2(\operatorname{Pol} (U_F^+), {_\e\C_\e}) \cong sl_F(n) =\{A\in M_n(\C) \colon AF=FA, \operatorname{Tr}(AF) = \operatorname{Tr}(AF^{-1}) = 0\}$ provided the eigenvalues of $F$ do not form a geometric progression. This agrees with Theorem \ref{thm:Hochschild-U+} as the only case in which $F$ is a positive matrix and satisfies $F^2=\pm I$ is for $F=I$. So for $F\neq I$ we have $\dim \mathrm H^2(\operatorname{Pol} (U_F^+), {_\e\C_\e}) = d-2$.

  To treat the lacking case we assume without loss of generality that $F=\operatorname{diag}(q^{-1},1,q)$ with $q\in (0,1)$. Then $F=E^tE^{-1}$ for 
  $$E=\begin{pmatrix}
    0 & 0 & q \\ 0 & 1 & 0 \\ 1 & 0 & 0
  \end{pmatrix}.$$
  In this case $d=3$, $p=1$, $t=2$ and so the Hochschild cohomology is
  \begin{align*}
    \dim \mathrm{H}^0(\mathcal H, {_\e\C_\e}) &= 1 ,
    &\dim \mathrm{H}^1(\mathcal H, {_\e\C_\e}) &= 3 ,
    \\
    \dim \mathrm H^2(\mathcal H, {_\e\C_\e}) &= d-t=1 ,
    &\dim \mathrm H^3(\mathcal H, {_\e\C_\e}) &= 1 ,
  \end{align*}
  and $\dim \mathrm H^k(\mathcal H, {_\e\C_\e}) = 0$ for all $k \geq 4$.
  So in the lacking case the second Hochschild cohomology has dimension 1 as in the neighbouring cases.
\end{example}
\begin{example}[Cohomology of \texorpdfstring{$\operatorname{Pol} (U_F^+)$}{Pol(UF+)} with the action 
    \texorpdfstring{$\overline{\e}$}{ε‾}]
  Denote by $\overline{\e}$ the character on $\operatorname{Pol} (U_F^+)$ defined by $\overline{\e}(u_{jk})=-\delta_{jk}$. In this case we have $S=T=-I$, $p=0$ and $t\in \{1,2\}$. For a positive $F\in \operatorname{GL}_n(\C)$ with $k$ different eigenvalues of multiplicities $d_i$, $i=1,2,\ldots,k$, we have $d=(\sum_{i=1}^k d_i^2)$.
  Then the Hochschild cohomology is:
  \begin{align*}
    \dim \mathrm{H}^0(\mathcal H, {_{\overline{\e}}\C_\e}) &= 0 ,
    &\dim \mathrm{H}^1(\mathcal H, {_{\overline{\e}}\C_\e}) &= d-1 ,\\
    \dim \mathrm H^2(\mathcal H, {_{\overline{\e}}\C_\e}) &= d-t ,
    &\dim \mathrm H^3(\mathcal H, {_{\overline{\e}}\C_\e}) &= 2-t ,
  \end{align*}
  and $\dim \mathrm H^k(\mathcal H, {_{\overline{\e}}\C_\e}) = 0$ for all $k \geq 4 $.
\end{example}

\section{Yetter-Drinfeld resolutions}
\label{sec:YD-resolutions}
The purpose of this section is two-fold. The resolutions we have considered so far are resolutions in the category of right modules and, therefore, do not reflect the Hopf structure of the algebras.  Turning the resolutions into resolutions in the category of Yetter-Drinfeld modules allows us to extract additional cohomological data related to the Hopf structure, namely Gerstenhaber-Schack and, thus, bialgebra cohomology.  The second point is that we only found resolutions for $\mathcal H(F)$ with $F=E^tE^{-1}$ an asymmetry, so that we may identify $\mathcal H(F)$ with a (Hopf) subalgebra of $\mathcal A(E)$. Yetter-Drinfeld resolutions can be transformed using monoidal equivalences of the underlying Hopf algebras --- it is this feature which allows us to find resolutions for $\mathcal H(F)$ even when $F$ is not an asymmetry. 

\subsection{The asymmetry case}
Throughout this subsection, we will again use the shorthand notation \[\mathcal A:=\mathcal A(E)\ (=\Pol{B_{n+1}^{\#+}}),\quad \mathcal B:=\mathcal B(E)\ (=\Pol{O_n^+}),\quad \mathcal H:=\mathcal H(F)\ (=\Pol{U_n^+}),\]
(equalities in brackets for $E=F=I_n$), and we assume that $F=E^tE^{-1}$ is the asymmetry associated with $E$, 
 so that \(\mathcal H \cong \mathcal B \gluedfree \mathbb C\mathbb Z_2 
\subset \mathcal B * \mathbb C\mathbb Z_2 
\cong \mathcal A.\)

It was shown by Bichon \cite{Bic16} that the resolution \eqref{resO+E} for $\mathcal B$ can be 
interpreted as a resolution by free Yetter-Drinfeld modules. Let us quickly 
review the involved coactions. First, as discussed above, $\mathbb 
C_\varepsilon$ is a Yetter-Drinfeld module with the trivial coaction 
$\gamma(1)=1\otimes 1$. Also, $\mathcal B$ is a (free over $\mathbb C$) Yetter-Drinfeld 
module with the coadjoint coaction $\gamma(a)=a_{(2)}\otimes S(a_{(1)})a_{(3)}$. 
To deal with $M_n(\mathcal B)$, first consider the $\mathcal B$-comodules 
$V=\operatorname{span}\{e_1,\ldots,e_n\}$ with coaction $\gamma(e_k):=\sum_p 
e_p\otimes x_{pk}$ and 
$V^*:=\operatorname{span}\{e_1^*,\ldots,e_n^*\}$ with coaction $\gamma(e_j^*):= \sum_q e_q^*\otimes S(x_{jq})$. We equip the tensor product of comodules $C,D$ with the diagonal coaction $\gamma(c\otimes d)=c_{(0)}\otimes d_{(0)}\otimes c_{(1)}d_{(1)}$. In the following, this kind of comodule tensor product is simply written by juxtaposition. In particular, we write $V^*V$ for the $\mathcal A$-comodule with coaction $\gamma(e_j^*e_k)=\sum e_q^*e_p\otimes S(u_{jq})u_{pk}$. Note that, as a $\mathcal B$-module, $M_n(\mathcal B)$, the free module over $M_n$, is isomorphic to $V^*V\boxtimes \mathcal B$, the free Yetter-Drinfeld module over $V^*V$; the isomorphism is the obvious one, sending $e_{jk}\otimes b$ to $e_j^*e_k\boxtimes b$. Therefore, one has the resolution
\begin{align}\label{YD-resO+}
    0\rightarrow \mathcal B\xrightarrow{\Phi^{\mathcal B}_3}V^*V\boxtimes \mathcal B\xrightarrow{\Phi^{\mathcal B}_2}V^*V \boxtimes \mathcal B\xrightarrow{\Phi^{\mathcal B}_1} \mathcal B \xrightarrow{\varepsilon}\mathbb C_\varepsilon\rightarrow 0
\end{align}
with the same maps $\Phi_k$ as in \eqref{resO+E} under the discussed 
identification of $M_n$ with $V^*V$ (it is easy to check that the $\Phi_k$ are 
also comodule maps, note that it is enough to check the condition on the 
generating subcomodules $\mathbb C1$ and $V^*V$, respectively).

The projective resolution \eqref{eq:resZ2} for $\mathbb C\mathbb Z_2$ we use is also easily identified as a resolution by projective Yetter-Drinfeld modules.

\begin{proposition}\label{prop:YD-res-Z2}
    $\mathbb C(1-g)$ is a relatively projective Yetter-Drinfeld module with coaction determined by $(1-g)\mapsto (1-g)\otimes 1$. Hence, the resolution \eqref{eq:resZ2} is a resolution by projective Yetter-Drinfeld modules.
\end{proposition}

\begin{proof}
  We can identify $\mathbb C(1-g)$ with the submodule of the free Yetter-Drinfeld module $\mathbb C\mathbb Z_2=\mathbb C\boxtimes \mathbb C\mathbb Z_2$, complemented by $\mathbb C(1+g)$. Note that the coadjoint coaction yields $\gamma(g)=g\otimes S(g)g=g\otimes 1$, therefore also $\gamma(1-g)=(1-g)\otimes 1$.
That the maps $\psi_1$ and $\varepsilon$ in \eqref{eq:resZ2} are Yetter-Drinfeld is obvious.
\end{proof}

\begin{corollary}
  Under the identification of $M_n$ with $V^*V$ discussed above, \eqref{eq_resolution_of_A} yields a resolution
  \begin{align}\label{eq:YD-res-B^hash+}
    0\to \mathcal A \xrightarrow{\Phi^{\mathcal A}_3}V^*V\boxtimes \mathcal A\xrightarrow{\Phi^{\mathcal A}_2} V^*V\boxtimes \mathcal A\oplus(\mathbb{C}(1-g)\otimes_{\mathbb{C}\mathbb{Z}_2}\mathcal A)\xrightarrow{\Phi^{\mathcal A}_1}\mathcal A\xrightarrow{\varepsilon}\mathbb{C}_\varepsilon\to0,
  \end{align}
  by relatively projective Yetter-Drinfeld modules.
\end{corollary}

\begin{proof}
  
  We apply Theorem \ref{thm:YD-res_free-product} to the resolutions \eqref{YD-resO+} for $\mathcal B$ and \eqref{eq:resZ2} for $\mathbb C\mathbb Z_2$ in its interpretation given by Proposition \ref{prop:YD-res-Z2}. Proposition \ref{prop:ind-free-YD} allows to identify the occuring free Yetter-Drinfeld $\mathcal A$-modules.
\end{proof}

In order to find the resolution for $\mathcal H$, we want to apply the restriction functor to \eqref{eq:YD-res-B^hash+}. First, we show that the assumptions of \Cref{lem:coinvariant} are satisfied.

\begin{lemma}\label{lem:H_adjoint_in_A}
  $\mathcal H\subset \mathcal A$ is adjoint, i.e.\
  \[a_{(2)}\otimes S(a_{(1)})ha_{(3)}\in \mathcal A\otimes \mathcal H\]
  for all $h\in \mathcal H,a\in \mathcal A$.
\end{lemma}

\begin{proof}
  Since $\mathcal A$ is generated as an algebra by the $x_{ij}$ and $g$, it is enough to show that
  \[\sum x_{pq}\otimes S(x_{ip})hx_{qj}\in \mathcal A\otimes \mathcal H,\quad g\otimes ghg\in \mathcal A\otimes \mathcal H\]
  for all $h\in \mathcal H$. This is obvious when we recall the fact that $\mathcal A$ is $\mathbb 
Z_2$-graded and $\mathcal H$ consists of all even polynomials in $\mathcal H$ (cf.\ the proof of 
Proposition \ref{prop:A=H+gH_general}).
\end{proof}

The following lemma will not be used in the sequel, but, by \cite[Proposition 4.7]{Bic16}, it provides an alternative proof of \Cref{lem:H_adjoint_in_A}.

\begin{lemma}
  Let $p\colon \mathcal B\to \mathbb C\mathbb Z_2$ denote the surjective Hopf-morphism $p(x_{ij})=\delta_{ij}g$ (which is well-defined by \Cref{lem:degree_of_reflection-B(E)}). Then $P=p\free\mathrm{id}\colon \mathcal A\cong \mathcal B\free\mathbb C\mathbb Z_2\to \mathbb C\mathbb Z_2$ is a cocentral and surjective Hopf algebra map such that $\mathcal H=\mathcal A^{\operatorname{co}\mathbb C\mathbb Z_2}=\{a\in \mathcal A:a_{(1)}\otimes P(a_{(2)})=a\otimes 1\}$.  
\end{lemma}

\begin{proof}
  For a monomial $M$ in $\mathcal A(E)$ it is easy to see that
  \[M_{(1)}\otimes P(M_{(2)})=
    \begin{cases}
      M\otimes 1&\text{if $M$ is even}\\
      M\otimes g&\text{if $M$ is odd}
    \end{cases}
  \]
  with respect to the natural $\mathbb Z_2$-grading.  We already discussed that $\mathcal H(F)$ is the even part of $\mathcal A(E)$, so we are done.
\end{proof}

\begin{lemma}
  The Yetter-Drinfeld modules in \eqref{eq:YD-res-B^hash+} are $\mathcal H$-coinvariant.
\end{lemma}

\begin{proof}
  Due to Lemma \ref{lem:H_adjoint_in_A}, we can apply Lemma \ref{lem:coinvariant}.  The Yetter-Drinfeld modules $\mathcal A, \mathbb C(1-g)\otimes_{\mathbb C\mathbb Z_2}\mathcal A,   V^*V\boxtimes \mathcal A$ are generated as $\mathcal A$-modules by $\mathbb C1, \mathbb C(1-g), (e_j^*e_k:i,j=1,\ldots,n)$, respectively. Therefore, the claimed coinvariance follows from
  \begin{align*}
    \gamma(1)&=1\otimes 1 \in \mathbb C\otimes \mathcal H,\\
    \gamma(1-g)&=1\otimes 1 - g\otimes S(g)g=(1-g)\otimes 1\in \mathbb C(1-g)\otimes \mathcal H, \\
    \gamma(e_j^*e_k)&=\sum e_q^* e_p\otimes S(x_{jq})x_{pk}=\sum e_q^* e_p\otimes S(gx_{jq})gx_{pk}\in V^*V\otimes \mathcal H. 
  \end{align*}
  For the free Yetter-Drinfeld modules, we could also invoke \cite[Proposition 4.5]{Bic16}
\end{proof}

This means that the images under the restriction functor are the same spaces with action and coaction (co-)restricted to $\mathcal H$. The general categorical result does not say that the restriction functor maps projectives to projectives. For the free Yetter-Drinfeld $\mathcal A$-modules, one could apply \cite[Proposition 4.8]{Bic16} to $\mathcal H\subset \mathcal A$.
Anyway, we will now see that the situation is even better than a general result could promise!

\begin{lemma}
  In $\mathcal{YD}_{\mathcal H}^{\mathcal H}$, we have isomorphisms
  \begin{center}
    \renewcommand{\arraystretch}{1.5}
    \begin{tabular}{cc}
      $\mathcal A\cong \mathcal H\oplus g\mathcal H=\mathbb C\mathbb Z_2\boxtimes \mathcal H\cong \mathcal H\oplus \mathcal H, $&$
        (h+gh')\mapsto h \oplus gh'\mapsto h\oplus h';$\\$
      \mathbb (1-g)\otimes_{\mathbb C\mathbb Z_2}\mathcal A\cong \mathbb C(1-g)\boxtimes \mathcal H\cong \mathcal H, $&$
     (1-g)\otimes_{\mathbb C\mathbb Z_2}(h+gh')\mapsto (1-g)\boxtimes (h-h')\mapsto h-h' ;$\\$
      V^*V \boxtimes \mathcal A\cong V^*V \boxtimes \mathcal H \oplus gV^*Vg\boxtimes \mathcal H, $&$ e_j^*e_k\boxtimes (h+gh')\mapsto e_j^*e_k\boxtimes h + ge_j^*e_kg\boxtimes h'.$ 
    \end{tabular}
  \end{center}
for $h,h'\in \mathcal H$. In particular, all given modules are free Yetter-Drinfeld $\mathcal H$-modules.
\end{lemma}

\begin{proof}
    The given prescriptions obviously define module isomorphisms. In the first two rows, it is easy to see that they are also Yetter Drinfeld. Denote the isomorphism in the third row by $\Psi$ for the rest of this proof. The right hand side is the free Yetter-Drinfeld module over $V^*V+gV^*Vg$. By the universal property of free Yetter-Drinfeld modules, it is enough to check that the restrictions of $\Psi^{-1}$ to  $V^*V$ and to $gV^*Vg$ (identified with $V^*V\boxtimes 1$ and $gV^*Vg\boxtimes 1$, respectively) are comodule maps. For $V^*V$ this is obvious and for $gV^*Vg$ we calculate
    \[\gamma(ge_j^*e_kg)=\sum ge_q^*e_pg\otimes gS(x_{jq})x_{pk}g \xmapsto{\Psi^{-1}\otimes\mathrm{id}} \sum(e_q^*e_p\boxtimes g)\otimes gS(x_{jq})x_{pk}g=\gamma(e_j^*e_k\boxtimes g).\]
\end{proof}

\begin{corollary}\label{cor:YD-resolution-for-generic-asymmetry}
  Let $F=E^t E^{-1}$ be a generic asymmetry and $\mathcal H=\mathcal H(F)$. Then the sequence
  \begin{align}
        0\to
    \begin{pmatrix}
      \mathcal H\\
      \mathcal H
    \end{pmatrix}
    \xrightarrow{\Phi^{\mathcal H}_3}
    \begin{pmatrix}
      V^*V\boxtimes \mathcal H\\
      gV^*Vg\boxtimes \mathcal H
    \end{pmatrix}
    \xrightarrow{\Phi^{\mathcal H}_2}
    \begin{pmatrix}
      V^*V\boxtimes \mathcal H\\
      gV^*Vg\boxtimes \mathcal H\\
      \mathcal H
    \end{pmatrix}
    \xrightarrow{\Phi^{\mathcal H}_1}
    \begin{pmatrix}
      \mathcal H\\
      \mathcal H
    \end{pmatrix}
    \xrightarrow{\varepsilon}\mathbb{C}_\varepsilon\to0,
  \end{align}
  is a resolution of $\mathbb C_{\varepsilon_{\mathcal H}}$ by free Yetter-Drinfeld $\mathcal H$-modules with respect to the maps for the resolution \eqref{eq:resHF} under the natural identifications of $V^*V\boxtimes \mathcal H$ and $gV^*Vg\boxtimes \mathcal H$ as with $M_n(\mathcal H)$ as $\mathcal H$-modules. \end{corollary}

\Cref{cor:YD-resolution-for-generic-asymmetry} only solves the problem for $F$ a generic asymmetry. In order to find a free Yetter-Drinfeld resolution for more general matrices $F$, we take the sequence \eqref{eq:resHF_transformed} as a starting point.

\subsection{Yetter-Drinfeld structures on \texorpdfstring{$M_n(\mathcal H)$}{Mn(H)}}

For arbitrary $F\in\operatorname{GL}_n(\C)$, we consider the following four $\mathcal H(F)$-coactions on $M_n(\C)$
\begin{align*}
  \gamma_u(e_{ij})&=\sum_{k,\ell} e_{k\ell}\otimes S(u_{ik})u_{\ell j}
  &
  \gamma_v(e_{ij})&=\sum_{k,\ell} e_{k\ell}\otimes S(v_{ik})v_{\ell j}
  \\  
  \gamma_{\tilde u}(e_{ij})&=\sum_{k,\ell} e_{k\ell}\otimes S(u_{j\ell})u_{k i}
  &
  \gamma_{\tilde v}(e_{ij})&=\sum_{k,\ell} e_{k\ell}\otimes S(v_{j \ell})v_{k i};
\end{align*}
the corresponding comodules are denoted by $M_n^{u}(\C),M_n^{v}(\C),M_n^{\tilde u}(\C),M_n^{\tilde v}(\C)$, respectively. To get a better grasp on these formulas, note that for a matrix $A\in M_n(\C)$, we have:
\begin{align*}
  \gamma_u(A)&=\sum_{i,j} A_{ij} \gamma_u (e_{ij}) = S(u)^t A u^t = vAu^t\\
  \gamma_v(A)&=\sum_{i,j} A_{ij} \gamma_v (e_{ij}) = S(v)^t A v^t = F^{-t}uF^t A v^t\\
  \gamma_{\tilde u}(A) &= \sum_{i,j} A_{ij} \gamma_{\tilde u} (e_{ij}) = (S(u)^t A^t u^t)^t = (vA^tu^t)^t\\
  \gamma_{\tilde v}(A)&= \sum_{i,j} A_{ij}\gamma_{\tilde v}(e_{ij}) = (F^{-t}uF^t A^t v^t)^t.
\end{align*}
For $\#\in\{u,v,\tilde u, \tilde v\}$, we denote the free Yetter-Drinfeld module over $M_n^{\#}(\C)\boxtimes \mathcal H$ by $M_n^{\#}(\mathcal H)$.

\begin{theorem}\label{thm:Psi_i-YD}
  Let $F\in\operatorname{GL}_n(\C)$ be arbitrary, $\mathcal H=\mathcal H(F)$. The maps $\Psi_i^{\mathcal H}$ in the sequence
  \begin{equation} 
	\label{eq:YD-res_H(F)}
	0\to \begin{pmatrix} \mathcal H \\ \mathcal H \end{pmatrix} 
	\xrightarrow{\Psi^{\mathcal H}_3} 
	\begin{pmatrix} M_n^v(\mathcal H) \\ M_n^u(\mathcal H) \end{pmatrix} 
	\xrightarrow{\Psi^{\mathcal H}_2}
	\begin{pmatrix} M_n^{\tilde v}(\mathcal H) \\ M_n^{\tilde u}(\mathcal H) \\ \mathcal H \end{pmatrix}
        \xrightarrow{\Psi^{\mathcal H}_1}
	\begin{pmatrix} \mathcal H \\ \mathcal H \end{pmatrix} \xrightarrow{\varepsilon}\mathbb{C}_\varepsilon\to0   
      \end{equation}
      defined as
       \begin{align*}
        \Psi_3^{\mathcal H}\begin{pmatrix} a \\ b \end{pmatrix}&= 
        \begin{pmatrix}
            -F^{-t} a + u F^tb\\ F^{-1} va - Fb
        \end{pmatrix},
        \\
        \Psi_2^{\mathcal H}\begin{pmatrix} A \\ B \end{pmatrix}&=
        \begin{pmatrix} A^t + Fu^tBF^{-1} \\  v^t A + B^t \\ 0 \end{pmatrix},
        \\
         \Psi_1^{\mathcal H}\begin{pmatrix} A \\ B \\ c\end{pmatrix}&=
        \begin{pmatrix} \tr(-A+uB)+c \\ \tr(vA-B)-c \end{pmatrix},
       \end{align*}
       are morphisms of Yetter-Drinfeld modules. 
    \end{theorem}

    \begin{proof}
      The only thing we have to prove is that the $\Psi_i^{\mathcal H}$ are comodule maps. According to the column vector form $P_i=((P_i)_k)$ of the involved Yetter-Drinfeld modules, the maps $\Psi_i^{\mathcal H}\colon P_i\to P_{i-1}$ inherit a matrix structure with components $(\Psi_i^{\mathcal H})_{k,\ell}\colon (P_i)_\ell\to (P_{i-1})_k$. Because we know that the maps are module maps already, it is enough to check $\mathcal H$-colinearity componentwise on generators. Let us start $\Psi_3$ and note that $\gamma(1)=1\otimes 1$ for the coaction on $\mathcal H=\C\boxtimes \mathcal H$. In two cases we only have to deal with scalar matrices and the calculations are very short:
      \begin{align*}
        \gamma_v((\Psi_3^{\mathcal H})_{11}(1))=\gamma_v(-F^{-t})= -F^{-t}uF^t F^{-t} v^t=-F^{-t}
         = (\Psi_3^{\mathcal H})_{11}\otimes\mathrm{id}(\gamma(1))
       \end{align*}
       \begin{align*}
         \gamma_u((\Psi_3^{\mathcal H})_{22}(1))=\gamma_u(-F)= -vFu^t= -F = (\Psi_3^{\mathcal H})_{22}\otimes\mathrm{id} (\gamma(1))
         \end{align*}
      For the remaining two components, we have to apply the coaction to non-scalar matrices, which makes the calculations a bit more cumbersome. (Recall from \Cref{subsec:YD-generalities} that for $X\in\mathcal{YD}_H^H, x\in X, h\in H$, we put $(x\otimes h)\YDaction h':=xh'_{(2)}\otimes S(h'_{(1)})h h'_{(3)}$, and the coaction $M\to M\otimes H$ is a module map for the $\YDaction$-action.) Observe that
      \begin{align*}
        \MoveEqLeft\gamma_u((\Psi_3^{\mathcal H})_{21}(1))=\gamma_u(F^{-1}v)\\
        &=\sum_{i,j,p} \left(\gamma_u(e_{ij}\boxtimes 1)\right)\YDaction F^{-1}_{ip}v_{pj}\\
        &=\sum_{i,j,k,\ell,p} e_{k\ell}\boxtimes 1\otimes  S(u_{ik})u_{\ell j} \YDaction F^{-1}_{ip}v_{pj}\\
        &=\sum _{\substack{i,j,k,\ell,\\p,s,t}} e_{k\ell}\boxtimes v_{st}\otimes S(v_{ps})F^{-1}_{ip}  S(u_{ik})\underbrace{u_{\ell j} v_{tj}}_{\delta_{\ell t}}\\
        &=\sum_{i,k,\ell,s,p} e_{k\ell}\boxtimes v_{s\ell}\otimes S(\underbrace{u_{ik} F^{-1}_{ip}v_{ps}}_{(u^tF^{-1}v)_{ks}=F^{-1}_{ks}})\\
        &=F^{-1}v\otimes 1 = (\Psi_3^{\mathcal H})_{21}\otimes \mathrm{id} (\gamma(1)),
      \end{align*}
      \begin{align*}
        \MoveEqLeft\gamma_v((\Psi_3^{\mathcal H})_{12}(1))=\gamma_v(uF^t)\\
        &=\sum_{i,j,p} \left(\gamma_u(e_{ij}\boxtimes 1)\right)\YDaction u_{ip} F_{jp}\\
        &=\sum_{i,j,k,\ell,p} e_{k\ell}\boxtimes 1 \otimes S(v_{ik}) v_{\ell j} \YDaction u_{ip} F_{jp}\\
        &=\sum_{\substack{i,j,k,\ell,\\p,s,t}} e_{k\ell}\boxtimes u_{st}\otimes S(u_{is}) S(v_{ik})\underbrace{v_{\ell j} F_{jp} u_{tp}}_{F_{\ell t}}\\
        &=\sum_{i,k,\ell,s,t} e_{k\ell}\boxtimes u_{st}  F_{\ell t}\otimes S(\underbrace{v_{ik}u_{is}}_{\delta_{ks}})\\
        &=uF^t\otimes 1 = (\Psi_3^{\mathcal H})_{12}\otimes \mathrm{id} (\gamma(1)).
      \end{align*}
      Now we come to $\Psi_2^{\mathcal H}$. Clearly, 
      \begin{align*}
        \gamma_{\tilde v}(\Psi_2^{\mathcal H})_{11}(A)&=\gamma_{\tilde v}(A^t)=\gamma_v(A)^t= (\Psi_2^{\mathcal H})_{11}\otimes\mathrm{id}(\gamma_{v}(A)),\\
        \gamma_{\tilde u}(\Psi_2^{\mathcal H})_{22}(B)&=\gamma_{\tilde u}(B^t)=\gamma_u(B)^t= (\Psi_2^{\mathcal H})_{11}\otimes\mathrm{id}(\gamma_{u}(B)).
      \end{align*}
      Note that
      \[(\Psi_2^{\mathcal H})_{21}(e_{ij}\boxtimes 1)=v^t e_{ij}= \sum_{p,q}v_{qp}e_{pq}e_{ij}=\sum_p v_{ip}e_{pj}=\sum_p e_{pj}\boxtimes v_{ip}.\]
      With that in mind, we find
      \begin{align*}
        \gamma_{\tilde u}(\Psi_2^{\mathcal H})_{21}(e_{ij})
        &=\sum_p \gamma_{\tilde u}(v_{ip}e_{pj})\\
        &=\sum_p \gamma_{\tilde u}(e_{pj})\YDaction v_{ip}\\
        &=\sum_{k,\ell,p} e_{k\ell}\boxtimes 1 \otimes S(u_{j\ell}) u_{kp}\YDaction v_{ip}\\
        &=\sum_{k,\ell,p,s,t} e_{k\ell}\boxtimes v_{st}\otimes S(v_{is})S(u_{j\ell}) \underbrace{u_{kp} v_{tp}}_{\delta_{kt}}\\
        &=\sum_{k,\ell,s} e_{k\ell}\boxtimes v_{sk}\otimes S(v_{is})v_{\ell j}\\
        &=\sum_{\ell,s} (\Psi_2^{\mathcal H})_{21}(e_{s \ell}\boxtimes 1)\otimes S(v_{is})v_{\ell j}\\
        &=((\Psi_2^{\mathcal H})_{21}\otimes \mathrm{id})(\gamma_{v}(e_{ij})).
      \end{align*}
      We can deal with $(\Psi_2^{\mathcal H})_{12}$ similarly. Using
      \[(\Psi_2^{\mathcal H})_{12}(e_{ij}\boxtimes 1)=Fu^te_{ij}F^{-1}= \sum_{p,q,r,s,t} e_{pt}e_{ij}e_{sq}\boxtimes F_{pr}u_{tr}F^{-1}_{sq}=\sum_{p,q,r} e_{pq}\boxtimes F_{pr}u_{ir}F^{-1}_{jq}, \]
      we get
      \begin{align*}
        \MoveEqLeft\gamma_{\tilde v}(\Psi_2^{\mathcal H})_{12}(e_{ij})
        =\sum_{p,q,r} \gamma_{\tilde v}(e_{pq})\YDaction F_{pr}u_{ir}F^{-1}_{jq}\\
        &=\sum_{k,\ell,p,q,r} e_{k\ell}\boxtimes 1 \otimes S(v_{q\ell}) v_{kp}\YDaction  F_{pr}u_{ir}F^{-1}_{jq}\\
        &=\sum_{\substack{k,\ell,p,q,\\r,s,t}} e_{k\ell}\boxtimes u_{st}\otimes S(u_{is}) \underbrace{ F^{-1}_{jq} S(v_{q\ell}) }_{u_{qj} F^{-1}_{q\ell}}\underbrace{ v_{kp}   F_{pr}  u_{tr}}_{F_{kt}}\\
        &=\sum_{k,\ell,q,s,t} e_{k\ell}\boxtimes F_{kt} u_{st} F^{-1}_{q\ell}\otimes S(u_{is})u_{qj} \\
        &=((\Psi_2^{\mathcal H})_{12}\otimes\mathrm{id})(\gamma_{u}(e_{ij})).
      \end{align*}
      Finally, the calculations for $\Psi^{\mathcal H}_1$ are relatively short. For $(\Psi^{\mathcal H}_1)_{11}(e_{ij})=(\Psi^{\mathcal H}_1)_{22}(e_{ij})=-\delta_{ij}$ we get
      \begin{align*}
        ((\Psi^{\mathcal H}_1)_{11}\otimes\mathrm{id})(\gamma_{\tilde v}(e_{ij}))&=\sum -\delta_{k\ell}\otimes S(v_{j\ell})v_{ki}=-\delta_{ij} 1\otimes 1=\gamma((\Psi^{\mathcal H}_1)_{11}(e_{ij}))\\
        ((\Psi^{\mathcal H}_1)_{22}\otimes\mathrm{id})(\gamma_{\tilde u}(e_{ij}))&=\sum -\delta_{k\ell}\otimes S(u_{j\ell})u_{ki}=-\delta_{ij} 1\otimes 1=\gamma((\Psi^{\mathcal H}_1)_{22}(e_{ij})).
      \end{align*}
      Moreover, $(\Psi^{\mathcal H}_1)_{21}(e_{ij})=\tr(ve_{ij})=v_{ji}$ and $(\Psi^{\mathcal H}_1)_{12}(e_{ij})=\tr(ue_{ij})=u_{ji}$, so that
      \begin{align*}
        ((\Psi^{\mathcal H}_1)_{21}\otimes\mathrm{id})(\gamma_{\tilde v}(e_{ij}))&=\sum v_{\ell k}\otimes S(v_{j\ell})v_{ki}=1\YDaction v_{ji}=\gamma((\Psi^{\mathcal H}_1)_{11}(e_{ij}))\\
        ((\Psi^{\mathcal H}_1)_{12}\otimes\mathrm{id})(\gamma_{\tilde u}(e_{ij}))&=\sum u_{\ell k}\otimes S(u_{j\ell})u_{ki}=1\YDaction u_{ji}=\gamma((\Psi^{\mathcal H}_1)_{11}(e_{ij})).
      \end{align*}
      The remaining components are either zero or $\pm$ identity maps on $\mathcal H$, so these are trivially comodule maps. 
    \end{proof}

\subsection{Monoidal equivalence and the general case}

The sequence from \Cref{thm:Psi_i-YD} is easily seen to be a complex for arbitrary $F\in \operatorname{GL}_n(\C)$, but our previous arguments only prove exactness if $F$ is a generic asymmetry. In order to prove exactness at least whenever $F$ is generic, we use the theory of Hopf-bi-Galois objects, which provide concrete monoidal equivalences between the categories of (free) Yetter-Drinfeld modules over $\mathcal H(F)$ and $\mathcal H(F')$ whenever $\tr{F}=\tr(F')$ and $\tr(F^{-1})=\tr({F'}^{-1})$. In fact, following Bichon's reasoning for $\mathcal B(E)$ \cite{Bic13}, we use the language of cogroupoids; an overview of the theory of cogroupoids and their relation to Hopf-bi-Galois objects can be found in \cite{bichon14}. Whether or not the sequence remains exact when $F$ is not normalizable remains an open problem, but as far as the compact quantum groups $U_K^+$ are concerned, the problem is resolved completely, because $\Pol{U_K^+}\cong \mathcal H(K^*K)$ and a positive matrix $K^*K$ is always normalizable.

For $F\in \operatorname{GL}_n(\C),F'\in \operatorname{GL}_{n'}(\C)$, consider the algebra 
$\mathcal H(F,F')$ generated by the entries of the rectangular matrices $u^{F,F'}=(u^{F,F'}_{ij})$ and $v^{F,F'}=(v^{F,F'}_{ij})$, where $i$ runs from $1$ to $n$ and $j$ runs from $1$ to $n'$, with relations
\begin{align*}
u^{F,F'}(v^{F,F'})^t&=I_n = v^{F,F'} F' (u^{F,F'})^t F^{-1},& (v^{F,F'})^tu^{F,F'}&=I_{n'}=F' (u^{F,F'})^t F^{-1}v^{F,F'} .
\end{align*}
Note that $\mathcal H(F,F)$ is simply $\mathcal H(F)$. 

Recall that, for a coalgebra $C$, the cotensor product of a right $C$-comodule $M$ with a left $C$-comodule $N$ is $M\mathbin{\Box} N:=\{X\in M\otimes N: \gamma_M\otimes \mathrm{id}(X)=\mathrm{id}\otimes\gamma_N(X)\}$.

The algebras $\mathcal H(F,F')$ carry a natural $\mathcal H(F)$-$\mathcal H(F')$-bicomodule structure, moreover, they form a \emph{cogroupoid} \cite[Lemma 3.8 and Definition 3.9]{bichon14}. Furthermore, by \cite[Proposition 2.4]{Bichon07}, $H(F,F')$ is non-zero whenever $\tr{F}=\tr(F')$ and $\tr(F^{-1})=\tr({F'}^{-1})$. From the general theory of cogroupoids (see \cite[Theorem 4.4]{Bic13}), it follows that the functor  $(-)\mathbin{\Box} \mathcal H(F,F')$ is an equivalence between the categories of Yetter-Drinfeld modules over $\mathcal H(F)$ and over $\mathcal H(F')$ in that case. Furthermore, for every right $\mathcal H(F)$ comodule $M$, there is a canonical isomorphism between the free Yetter-Drinfeld module over the cotensor product, $(M\mathbin\Box \mathcal H(F,F'))\boxtimes \mathcal H(F')$, and the cotensor product with the free Yetter-Drinfeld module,  $(M\boxtimes \mathcal H(F))\mathbin\Box \mathcal H(F,F')$.

We assume for the rest of this section that $\tr{F}=\tr(F')$ and $\tr(F^{-1})=\tr({F'}^{-1})$ so that $\mathcal H(F,F')$ is non-zero. Because the functor $(-)\mathbin{\Box} \mathcal H(F,F')$ is an equivalence of abelian categories, it is exact. 
Therefore, this allows to transform any exact sequence of free $\mathcal H(F)$ Yetter-Drinfeld modules $P_i$ into an exact sequence of $\mathcal H(F')$ Yetter-Drinfeld modules $P_i\mathbin{\Box} \mathcal H(F,F')$. In \Cref{lem:YD-module-isomorphisms,lem:monoidal-equivalence-transforms-Psi(F)-into-Psi(F')}, we will see that if we apply this technique to the resolutions above, the resolution for $F$ transforms into the resolution for $F'$ (or the other way round if we exchange the roles of $F$ and $F'$).

\begin{lemma}\label{lem:YD-module-isomorphisms}
  The following prescriptions uniquely extend to isomorphisms of Yetter-Drinfeld modules:
  \begin{align*}
    \mathcal H(F')\ni 1 &\mapsto 1\otimes 1 \in \mathcal H(F)\mathbin{\Box} \mathcal H (F,F'),\\
    M_n^{u^{F'}}\ni e_{ij}&\mapsto \sum_{k,\ell} e_{k\ell}\otimes S^{F'F}(u_{ik}^{F'F})u_{\ell j}^{F F'}\in M_n^{u^{F}}\mathbin{\Box}
                            \mathcal H(F,F'),\\
    M_n^{v^{F'}}\ni e_{ij}&\mapsto \sum_{k,\ell} e_{k\ell}\otimes S^{F'F}(v_{ik}^{F'F})v_{\ell j}^{F'F}\in M_n^{v^{F}}\mathbin{\Box}
                            \mathcal H(F,F'),\\
    M_n^{\tilde u^{F'}}\ni e_{ij}&\mapsto \sum_{k,\ell} e_{k\ell}\otimes S^{F'F}(u_{j\ell}^{F'F})u_{ki}^{FF'}\in M_n^{\tilde u^{F}}\mathbin{\Box}
                            \mathcal H(F,F'),\\
    M_n^{\tilde v^{F'}}\ni e_{ij}&\mapsto \sum_{k,\ell} e_{k\ell}\otimes S^{F'F}(v_{j\ell}^{F'F})v_{ki}^{FF'}\in M_n^{\tilde v^{F}}\mathbin{\Box}
                            \mathcal H(F,F').
  \end{align*}
\end{lemma}

\begin{proof}
  It is easy to check that the linear extensions to $M_n(\C)$ are comodule maps, hence they uniquely extend to morphisms between the free Yetter-Drinfeld modules.  In order to understand why they are isomorphisms, recall the following canonical isomorphisms.
  \begin{itemize}
  \item $\mathcal H(F,F'')\cong\mathcal H(F,F')\mathbin\Box \mathcal H(F',F''), u_{ij}^{FF''}\mapsto \sum u_{ik}^{FF'}\otimes u_{kj}^{F'F''}, v_{ij}^{FF''}\mapsto \sum v_{ik}^{FF'}\otimes v_{kj}^{F'F''}$ defines a Hopf algebra isomorphism.
  \item For every $\mathcal H(F)$ comodule $M$,
    \[M\cong M\mathbin\Box \mathcal H(F), m\mapsto \gamma(m) \]
    defines an isomorphism of comodules. If $M$ is a Yetter-Drinfeld module, the isomorphism is an isomorphism of Yetter Drinfeld modules.
  \end{itemize}
  Now one can check that the iteration \[M_n^{u^{F'}}\to M_n^{u^{F}}\mathbin{\Box}\mathcal H(F,F')\to M_n^{u^{F'}}\mathbin{\Box}\mathcal H(F',F)\mathbin{\Box}\mathcal H(F,F')\cong M_n^{u^{F'}}\mathbin{\Box}\mathcal H(F')\cong M_n^{u^{F'}}\]
  yields the identity on $M_n^{u^{F'}}$:
  \begin{align*}
    \sum e_{pq}\otimes S^{F'F}(u_{kp}^{F'F})u_{q \ell}^{F F'}\otimes S^{F'F}(u_{ik}^{F'F})u_{\ell j}^{F F'}\cong \sum e_{pq}\otimes S^{F'}(u_{pi}^{F'}) u_{qj}^{F'}\cong e_{ij}.
  \end{align*}
  The other cases work similarly. 
\end{proof}

\begin{lemma}\label{lem:monoidal-equivalence-transforms-Psi(F)-into-Psi(F')}
  Under identification given by the isomorphisms of the previous lemma, the maps $\Psi_i^{\mathcal H(F)}\mathbin{\Box} \mathcal H(F,F')$ coincide with $\Psi_i^{\mathcal H(F')}$. 
\end{lemma}

\begin{proof}
  Note that the formulas for the isomorphisms look exactly the same as the formulas for the coactions, once we hide in the notation the superscripts and the ranges for the indices. The calculations therefore are almost identical to those in the proof of \Cref{thm:Psi_i-YD}, only $F$ has to be replaced by $F'$ in the appropriate places. 
\end{proof}

\begin{corollary}\label{cor:normalizable_implies_exactness}
  The sequence in \Cref{thm:Psi_i-YD} is exact for all generic $F$. In particular, we obtain a free resolution of the counit for the free unitary compact quantum groups $U_K^+$ with arbitrary $K\in \operatorname{GL}_n(\C)$.
\end{corollary}

\begin{proof}
    Let $F$ be generic, in particular, $F$ is normalizable.
    As can be read from the proof of \cite[Corollary 2.5(i)]{Bichon07}, there exists an $F'$ with $\tr F=\tr F'$ and $\tr F^{-1}= \tr{F'}^{-1}$ such that $F'=\alpha E^tE^{-1}$ is a scalar multiple of an asymmetry, $\alpha\neq0$.\footnote{Namely, for $\lambda\neq0$ such that $\tr \lambda F=\tr ((\lambda F)^{-1})$ and $q$ a (automatically non-zero) solution of $q^2-\tr(\lambda F)q+1=0$,  
    \[F':=\frac{1}{\lambda}\begin{pmatrix}
        q^{-1}&0\\0&q
    \end{pmatrix}=\frac{1}{\lambda}\begin{pmatrix}
        0&1\\q^{-1}&0
    \end{pmatrix}^t\begin{pmatrix}
        0&1\\q^{-1}&0
    \end{pmatrix}^{-1}.\]} Clearly, $F$ is generic if and only if $F'$ is generic if and only if $E^tE^{-1}$ is generic. The sequence from \Cref{thm:Psi_i-YD} applied to $E^tE^{-1}$ is exact. Note that the sequence from \Cref{thm:Psi_i-YD} applied to $F'=\alpha E^tE^{-1}$ is almost the same as the one for $E^tE^{-1}$ because $\mathcal H(F')= \mathcal H(E^tE^{-1})$ and the scalar $\alpha$ cancels out in most places, the only difference being that 
    \[\Psi_3^{\mathcal H(F')}\begin{pmatrix}
        a\\b
    \end{pmatrix}=\Psi_3^{\mathcal H(E^tE^{-1})}\begin{pmatrix}
        \alpha^{-1} &0\\0& \alpha\end{pmatrix}\begin{pmatrix}
        a\\b
    \end{pmatrix}.\] Consequently, the images agree and the kernels have the same dimension, which shows that the sequence for $F'$ is also exact. Finally, the sequence for $F'$ transforms under cotensor product with $\mathcal H(F',F)$ into the corresponding sequence for $F$ , which is therefore also exact.
\end{proof}

With this, the proofs of \Cref{main-result:resolution-for-H(F),main-result:Hochschild} are completed.

\section{Bialgebra cohomology for  \texorpdfstring{$\mathcal H(F)$}{H(F)} and \texorpdfstring{$\mathcal A(E)$}{A(E)}}
\label{sec:bialgebra-cohomology}

Let $F$ be generic, so that \eqref{eq:YD-res_H(F)} in \Cref{thm:Psi_i-YD} is a free resolution by \Cref{cor:normalizable_implies_exactness}.
As explained in  \cite[Section 2.6]{Bic16}, the bialgebra cohomology of $\mathcal H=\mathcal H(F)$ is the cohomology of the complex
\begin{multline}
  \label{eq:YD-complex-for-H}
  0\to \hom_{\mathcal H}^{\mathcal H}\left(
    \begin{pmatrix} \mathcal H \\ \mathcal H \end{pmatrix},
    \C_\e\right) 
  \xrightarrow{-\circ \Psi^{\mathcal H}_1} 
  \hom_{\mathcal H}^{\mathcal H}\left(
    \begin{pmatrix} M_n^{\tilde v}(\mathcal H) \\ M_n^{\tilde u}(\mathcal H) \\ \mathcal H \end{pmatrix},
    \C_\e\right)
  \\
  \xrightarrow{-\circ \Psi^{\mathcal H}_2}
  \hom_{\mathcal H}^{\mathcal H}\left(
    \begin{pmatrix} M_n^v(\mathcal H) \\ M_n^u(\mathcal H) \end{pmatrix} ,
    \C_\e\right)
  \xrightarrow{-\circ\Psi^{\mathcal H}_3}
  \hom_{\mathcal H}^{\mathcal H}\left(
    \begin{pmatrix} \mathcal H \\ \mathcal H \end{pmatrix},
    \C_\e\right)\to0.
\end{multline}

\begin{observation}\label{obs:hom_H^H}
  By the universal property of free Yetter-Drinfeld modules,
  \[\hom_{\mathcal H}^{\mathcal H}(W\boxtimes \mathcal H,\mathbb C_\varepsilon)\cong \hom^{\mathcal H}(W,\mathbb C)\]
  for every $\mathcal H$-comodule $W$. Obviously, $\hom^{\mathcal H}(\mathbb C,\mathbb C)\cong \mathbb C$, where the inverse isomorphism maps $\lambda\in\mathbb C$ to the map $(z\mapsto \lambda z)\in\hom^{\mathcal H}(\mathbb C,\mathbb C)$. Now let $f\colon  M_n^u(\C)\to \mathbb C$ be a linear map, $f(u_{jk})=f_{jk}$. Then 
  \begin{align*}
    f\in \hom^{\mathcal H}( M_n^u(\mathcal H), \mathbb C)
    \iff \forall_{j,k}\ f\otimes\mathrm{id}(\gamma^{u}(e_{jk}))=\gamma(f(e_{jk}))
  \end{align*}
  and
  \begin{align*} f\otimes\mathrm{id}(\gamma(e_{jk}))=\gamma(f(e_{jk}))&\iff \sum f_{qp}S(u_{jq})u_{pk}=f_{jk}1\\
                                                                              &\iff\sum f_{jp} u_{pk} = \sum f_{qk} u_{jq}\\
                                                                              &\iff f_{jk}=\delta_{jk} f_{jj}  
  \end{align*}
  show that
  \[\hom_{\mathcal H}^{\mathcal H}( M_n^u(\mathcal H),\mathbb C_\varepsilon)\cong \hom^{\mathcal H}( M_n^u(\C),\mathbb C)\cong \mathbb C\]
  where the inverse isomorphism maps $\lambda\in\mathbb C$ to the linear map $f$ with $f_{jk}=\delta_{jk}\lambda$. The same holds of course for all $\hom_{\mathcal H}^{\mathcal H}(M_n^{\#}(\mathcal H),\mathbb C_\varepsilon)$ with $\#\in \{u,v,\tilde u,\tilde v\}$. When we follow the calculations in the proof of \Cref{thm:Hochschild-U+}, we  have to replace matrices by constant diagonal matrices! Therefore, the bialgebra cohomology of $\mathcal H$ is the cohomology of the complex (using the notation $(\Psi^{\mathcal H}_j)^{*}$ in the appropriately adapted meaning) 
  \[0\to\mathbb C^2\xrightarrow{(\Psi^{\mathcal H}_1)^{*}}\mathbb C^3\xrightarrow{(\Psi^{\mathcal H}_2)^{*}} \mathbb C^2\xrightarrow{(\Psi^{\mathcal H}_3)^{*}}\mathbb C^2\to 0\]
  with maps
  \begin{align*}
    (\Psi^{\mathcal H}_1)^{*}
    \begin{pmatrix}
      \lambda\\\mu
    \end{pmatrix}
    =
    \begin{pmatrix}
      \mu-\lambda\\\lambda-\mu\\\lambda-\mu
    \end{pmatrix}
    ,
    \quad
    (\Psi^{\mathcal H}_2)^{*}
    \begin{pmatrix}
      \lambda\\\mu\\\nu
    \end{pmatrix}
    =
    \begin{pmatrix}
      \lambda+\mu\\\lambda+\mu
    \end{pmatrix}
    ,
    \quad
    (\Psi^{\mathcal H}_3)^{*}
    \begin{pmatrix}
      \lambda\\\mu
    \end{pmatrix}
    =\operatorname{tr}(F)
    \begin{pmatrix}
      \mu-\lambda\\\lambda-\mu
    \end{pmatrix}.
  \end{align*}
  It is straightforward to determine image and kernel.
\end{observation}

\begin{theorem}[\Cref{main-result:bialgebra-cohomology}] Let $F$ be generic (in particular, $\tr(F)\neq 0$).
 Then the bialgebra cohomology of $\mathcal H=\mathcal H(F)$ is given as follows:
\begin{align*}
    \dim \mathrm{H}_{\mathrm b}^0(\mathcal H) &= 1 ,
    &\dim \mathrm{H}_{\mathrm b}^1(\mathcal H) &= 1 ,
    &\dim \mathrm H_{\mathrm b}^2(\mathcal H) &= 0,& 
    \dim \mathrm H_{\mathrm b}^3(\mathcal H) &= 1, 
\end{align*}
and $\dim \mathrm H_{\mathrm b}^k(\mathcal H) = 0$ for all $ k \geq 4$.
\end{theorem}

\begin{observation}
  Analogously to Observation~\ref{obs:hom_H^H} we get
  \[f\in\hom^{\mathcal A}(V^*V,\mathbb C)\iff f_{jk}=\delta_{jk} f_{jj}\]
  for a linear map $f\colon V^*V\to\mathbb C, f(e_j^*e_k)=f_{jk}$. Of course,
  \[\hom_{\mathcal A}^{\mathcal A}((1-g)\otimes_{\mathbb C\mathbb Z_2}\mathcal A,\mathbb C_\varepsilon)\subset \hom_{\mathcal A} ((1-g)\otimes_{\mathbb C\mathbb Z_2}\mathcal A,\mathbb C_\varepsilon)=0.\]
\end{observation}

  The bialgebra cohomology of $\mathcal A=\mathcal A(E)$ can now again be deduced in the same manner. \mbox{(Co-)}Restricting the $(\Phi_i^{(\mathcal B)})^*$ to constant diagonal matrices results in the complex \[0\to\mathbb C\xrightarrow{0}\mathbb C\xrightarrow{2} \mathbb C\xrightarrow{0}\mathbb C\to 0.\] Therefore we obtain:

\begin{theorem}
    Let $F=E^tE^{-1}$ be a generic asymmetry. 
    Then the bialgebra cohomology of $\mathcal A=\mathcal A(E)$ is given as follows:
    \begin{align*}
    \dim \mathrm{H}_{\mathrm b}^0(\mathcal A) &= 1 ,
    &\dim \mathrm{H}_{\mathrm b}^1(\mathcal A) &= 0 ,
    &\dim \mathrm H_{\mathrm b}^2(\mathcal A) &= 0 ,&
    \dim \mathrm H_{\mathrm b}^3(\mathcal A) &= 1,
\end{align*}
and $\dim \mathrm H_{\mathrm b}^k(\mathcal A) = 0$ for all $ k \geq 4$.
\end{theorem}


\section*{Acknowledgements}

We are grateful to Alexander Mang and Moritz Weber for fruitful discussions at an initial stage of this project, and for bringing the glued products to our intention. We are also indebted to Julien Bichon for many useful hints and explanations and for sharing a preliminary draft of \cite{Bichon23pre} with us.


\bibliographystyle{myalphaurl}
\bibliography{biblio}
\setlength{\parindent}{0pt}
\end{document}